\documentclass[11pt]{amsart}
\usepackage{
amssymb
,amsthm
,amsmath
,amscd
,mathtools,
amsfonts, amssymb, mathrsfs, 
centernot,
,url
,accents, hyphenat, calc, stmaryrd, mathdots,array, avant
}
\usepackage[all]{xy}
\usepackage[top=35truemm,bottom=35truemm,left=30truemm,right=30truemm]{geometry}
\usepackage[usenames]{color}

\usepackage[all]{xy}
\usepackage[T3,T1]{fontenc}

\newcommand{\Z}{\mathbb{Z}}
\newcommand{\R}{\mathbb{R}}
\newcommand{\C}{\mathbb{C}}

\renewcommand{\AA}{\mathcal{A}}
\newcommand{\Sc}{\mathcal{S}}
\newcommand{\PP}{\mathcal{P}}
\newcommand{\EE}{\mathcal{E}}

\newcommand{\GL}{\mathrm{GL}}
\newcommand{\SL}{\mathrm{SL}}
\newcommand{\SO}{\mathrm{SO}}
\newcommand{\Sp}{\mathrm{Sp}}

\newcommand{\Irr}{\mathrm{Irr}}
\newcommand{\unit}{\mathrm{unit}}
\newcommand{\gp}{\mathrm{gp}}
\newcommand{\temp}{\mathrm{temp}}
\newcommand{\Ind}{\mathrm{Ind}}
\newcommand{\Jac}{\mathrm{Jac}}
\newcommand{\soc}{\mathrm{soc}}
\newcommand{\adj}{\mathrm{adj}}
\newcommand{\supp}{\mathrm{supp}}

\newcommand{\Cusp}{\mathscr{C}}

\newcommand{\iif}{&\quad&\text{if }}

\newcommand{\resp}{resp.~}
\renewcommand{\1}{\mathbf{1}}
\newcommand{\ep}{\varepsilon}

\newcommand{\pair}[1]{\left\langle #1 \right\rangle}
\newcommand{\half}[1]{\frac{#1}{2}}
\newcommand{\ub}[1]{\underline{#1}}

\newtheorem{thm}{Theorem}[section]
\newtheorem{lem}[thm]{Lemma}
\newtheorem{prop}[thm]{Proposition}
\newtheorem{cor}[thm]{Corollary}
\newtheorem{rem}[thm]{Remark}
\newtheorem{defi}[thm]{Definition}
\newtheorem{ex}[thm]{Example}
\newtheorem{alg}[thm]{Algorithm}

\numberwithin{equation}{section}

\title{Construction of local $A$-packets}
\author{Hiraku Atobe}
\date{}
\address{
Department of Mathematics, Hokkaido University,
Kita 10, Nishi 8, Kita-Ku, Sapporo, Hokkaido, 060-0810, Japan 
}
\email{
atobe@math.sci.hokudai.ac.jp
}

\allowdisplaybreaks
\begin{document}
\maketitle

\begin{abstract}
In this paper, we reformulate M{\oe}glin's explicit construction of local $A$-packets of 
split odd special orthogonal groups and symplectic groups. 
By this reformulation together with results of the previous paper with M{\'i}nguez, 
we can compute the $A$-packets explicitly. 
Also, we give a non-vanishing criterion of our parametrization, 
and an algorithm to compute certain derivatives. 
Finally, we prove a formula for the Aubert duals of irreducible representations of Arthur type. 
\end{abstract}

\tableofcontents

\section{Introduction}
Let $G_n$ be a split special odd orthogonal group $\SO_{2n+1}(F)$ or a symplectic group $\Sp_{2n}(F)$ 
of rank $n$ over a non-archimedean local field $F$ of characteristic zero. 
An \emph{$A$-parameter} for $G_n$ is a homomorphism 
\[
\psi \colon W_F \times \SL_2(\C) \times \SL_2(\C) \rightarrow \widehat{G}_n, 
\]
where $W_F$ is the Weil group of $F$ and $\widehat{G}_n$ is the complex dual group of $G_n$. 
In his magnificent work \cite{Ar}, 
Arthur constructed the \emph{(local) $A$-packet $\Pi_\psi$} associated to $\psi$, 
which is a multi-set over $\Irr_\unit(G_n)$ together with a map
\[
\Pi_\psi \rightarrow \widehat{\Sc_{\psi}},\; \pi \mapsto \pair{\cdot, \pi}_\psi, 
\]
where $\Sc_{\psi}$ is the component group of $\psi$ 
and $\widehat{\Sc_{\psi}}$ is its Pontryagin dual (\cite[Theorem 1.5.1]{Ar}). 
Here, $\Irr_{(\unit)}(G_n)$ is the set of equivalence classes of irreducible (unitary) representations of $G_n$.
Arthur's multiplicity formula \cite[Theorem 1.5.2]{Ar} says that 
the elements in the local $A$-packets are 
the local components of square-integrable automorphic representations. 
For the theory of automorphic representations, 
it is an important problem to understand the local $A$-packets ``explicitly''. 
This problem would be also central for the unitary dual problem (see e.g., \cite[Conjecture 1.2]{T-unitary}).
However, the local $A$-packets are defined by endoscopic character identities
so that the local meaning of $A$-packets is very unclear. 
\vskip 10pt

In her consecutive works \cite{Moe06a, Moe06b, Moe09, Moe10, Moe11c}, 
M{\oe}glin constructed the local $A$-packets $\Pi_\psi$ concretely. 
As a consequence, she proved that $\Pi_\psi$ is multiplicity-free, i.e., a subset of $\Irr_\unit(G_n)$.
We explain this construction briefly. 
To construct $A$-packets, M{\oe}glin considered the following filtration of $A$-parameters: 
\[
(\text{elementary}) \subset (\text{discrete diagonal restriction}) 
\subset (\text{of good parity}) \subset (\text{general}).
\]
The methods in each step are as follows:
\begin{itemize}
\item
Irreducible parabolic inductions from the good parity case to the general case
(see Theorem \ref{irr} below);
\item
derivatives with respect to \emph{non-self-dual} cuspidal representations
from the case of the discrete diagonal restriction (DDR) to the good parity case; 
\item
taking certain socles (maximal semisimple subrepresentations) 
from the elementary case to the case of DDR; 
\item
partial Aubert involutions for the elementary case. 
\end{itemize}
See \cite{X2} for more details.
Here, we say that an $A$-parameter $\psi$ is \emph{of good parity} 
if $\psi$ is a sum of irreducible self-dual representations of $W_F \times \SL_2(\C) \times \SL_2(\C)$ 
of the same type as $\psi$. 
Since irreducible parabolic inductions are easy to understand, 
we will consider $A$-parameters of good parity only. 
\vskip 10pt

\emph{Derivatives} are certain partial information of Jacquet modules. 
Fix an irreducible unitary supercuspidal representation $\rho$ of $\GL_d(F)$. 
Let $P_d$ be the standard parabolic subgroup of $G_n$ with Levi subgroup $\GL_d(F) \times G_{n-d}$. 
For a smooth representation $\pi$ of $G_n$ of finite length, 
if the semisimplification of Jacquet module $\Jac_{P_d}(\pi)$ is 
\[
[\Jac_{P_d}(\pi)] = \bigoplus_i \tau_i \boxtimes \pi_i
\]
with irreducible representations $\tau_i \boxtimes \pi_i$ of $\GL_d(F) \times G_{n-d}$, 
we define the \emph{$\rho|\cdot|^x$-derivative} $D_{\rho|\cdot|^x}(\pi)$ by 
\[
D_{\rho|\cdot|^x}(\pi) = \bigoplus_{\substack{i \\ \tau_i \cong \rho|\cdot|^x}} \pi_i
\]
for $x \in \R$. 
Also, set $D_{\rho|\cdot|^x}^{(k)}(\pi) = \frac{1}{k!} D_{\rho|\cdot|^x} \circ \dots \circ D_{\rho|\cdot|^x}(\pi)$ 
($k$-times composition). 
If $D_{\rho|\cdot|^x}^{(k)}(\pi) \not=0$ but $D_{\rho|\cdot|^x}^{(k+1)}(\pi) = 0$, 
we say that $D_{\rho|\cdot|^x}^{(k)}(\pi)$ is the \emph{highest $\rho|\cdot|^x$-derivative} of $\pi$. 
It is known that if $\pi$ is irreducible and if $\rho|\cdot|^x$ is \emph{not self-dual}, 
then the highest $\rho|\cdot|^x$-derivative $D_{\rho|\cdot|^x}^{(k)}(\pi)$ is also irreducible. 
Moreover, by results of the previous paper (\cite[Proposition 6.1, Theorem 7.1]{AM}), 
one can compute the correspondence $\pi \mapsto D_{\rho|\cdot|^x}^{(k)}(\pi)$ and its converse explicitly. 
In particular, one would compute the $A$-packets for the good parity case 
if one were to know the $A$-packets in the case of DDR explicitly.
\vskip 10pt

The elementary case is a main cause of the complexity of M{\oe}glin's construction. 
This case contains not only all discrete series representations and their (usual) Aubert duals, 
but also their ``intermediates''. 
\emph{Partial Aubert involutions} were introduced to construct these ``intermediates'', but
they are very artificial so that they lose the compatibility of the Jacquet functors. 
Because of the elementary case, in the case of DDR, 
one has to consider the socles of quite non-tempered parabolically induced representations. 
Therefore, it is very hard to compute $\Pi_\psi$ by M{\oe}glin's construction
in the elementary case and in the case of DDR. 
\vskip 10pt

The reason why M{\oe}glin used the above filtration would be that 
derivatives with respect to \emph{self-dual} cuspidal representations are complicated
(see cf., \cite[\S 3.3]{AM}).
The same difficulty appears when one tries to give an algorithm for the Aubert duality (see \cite[\S 7]{J-dual}). 
To overcome this difficulty, 
in the previous paper \cite{AM}, the author and M{\'i}nguez introduced 
the notions of \emph{$\Delta_\rho[0,-1]$-derivatives} and \emph{$Z_\rho[0,1]$-derivatives}. 
These derivatives can be applied to the construction of $A$-packets, i.e., 
they will make M{\oe}glin's construction of $A$-packets more simply and more computable. 
As a refinement of her construction, we will use the following filtration: 
\[
(\text{non-negative discrete diagonal restriction}) \subset (\text{of good parity}) \subset (\text{general}). 
\]
Basically, the methods are the same as above. 
In the case of the non-negative DDR, 
the socles of \emph{standard modules} are taken.
Hence one can understand these representations from the definition in terms of the Langlands classification.
Although we have to consider $\rho|\cdot|^x$-derivatives for any $x \in (1/2)\Z$ in the good parity case, 
thanks to results in \cite{AM}, one can compute them explicitly
(see cf., the proof of Theorem \ref{compatible} below). 
\vskip 10pt

To state our results more precisely, we will introduce a key notion. 
Recall in \cite[Theorem 1.5.1]{Ar} that 
if an $A$-parameter $\psi = \phi$ is \emph{tempered}, i.e., if the restriction to the second $\SL_2(\C)$ is trivial, 
then $\Pi_\phi$ consists of irreducible tempered representations, and 
the map $\Pi_\phi \rightarrow \widehat{\Sc_\phi}$ is bijective. 
When $\pi \in \Pi_\phi$ corresponds to $\eta \in \widehat{\Sc_\phi}$, we write $\pi = \pi(\phi, \eta)$. 

\begin{defi}[Definition \ref{segments}]
\begin{enumerate}
\item
An \emph{extended segment} is a triple $([A,B]_\rho, l, \eta)$,
where
\begin{itemize}
\item
$[A,B]_\rho = \{\rho|\cdot|^A, \rho|\cdot|^{A-1}, \dots, \rho|\cdot|^B \}$ is a segment, 
with an irreducible unitary cuspidal representation $\rho$ of some $\GL_d(F)$; 
\item
$l \in \Z$ with $0 \leq l \leq \half{b}$, where $b = \#[A,B]_\rho = A-B+1$; 
\item
$\eta \in \{\pm1\}$. 
\end{itemize}

\item
An \emph{extended multi-segment} for $G_n$ is 
an equivalence class of multi-sets of extended segments 
\[
\EE = \cup_{\rho}\{ ([A_i,B_i]_{\rho}, l_i, \eta_i) \}_{i \in (I_\rho,\succ)}
\]
such that 
\begin{itemize}
\item
$I_\rho$ is a totally ordered finite set with a fixed total order $\succ$ satisfying that 
if $A_i > A_j$ and $B_i > B_j$, then $i \succ j$;

\item
$A_i + B_i \geq 0$ for all $\rho$ and $i \in I_\rho$; 

\item
as a representation of $W_F \times \SL_2(\C) \times \SL_2(\C)$, 
\[
\psi = \bigoplus_\rho \bigoplus_{i \in I_\rho} \rho \boxtimes S_{a_i} \boxtimes S_{b_i} 
\]
where $(a_i, b_i) = (A_i+B_i+1, A_i-B_i+1)$,
is an $A$-parameter for $G_n$ of good parity; 

\item
a sign condition
\[
\prod_{\rho} \prod_{i \in I_\rho} (-1)^{[\half{b_i}]+l_i} \eta_i^{b_i} = 1
\]
holds. 
\end{itemize}

\item
Two extended segments $([A,B]_\rho, l, \eta)$ and $([A',B']_{\rho'}, l', \eta')$ are \emph{equivalent} 
if 
\begin{itemize}
\item
$[A,B]_\rho = [A',B']_{\rho'}$; 
\item
$l = l'$; and 
\item
$\eta = \eta'$ whenever $l = l' < \half{b}$. 
\end{itemize}
Similarly, 
$\EE = \cup_{\rho}\{ ([A_i,B_i]_{\rho}, l_i, \eta_i) \}_{i \in (I_\rho,\succ)}$ 
and 
$\EE' = \cup_{\rho}\{ ([A'_i,B'_i]_{\rho}, l'_i, \eta'_i) \}_{i \in (I_\rho,\succ)}$ 
are \emph{equivalent}
if $([A_i,B_i]_\rho, l_i, \eta_i)$ and $([A'_i,B'_i]_{\rho}, l'_i, \eta'_i)$ are equivalent for all $\rho$ and $i \in I_\rho$.

\item
The \emph{support} of $\EE$ is the multi-segment 
\[
\supp(\EE) = \cup_{\rho}\{ [A_i,B_i]_{\rho} \}_{i \in (I_\rho,\succ)}.
\]

\end{enumerate}
\end{defi}
Here, $\rho$ is identified with an irreducible bounded representation of $W_F$ by the local Langlands correspondence for $\GL_d(F)$, 
and $S_a$ is the unique irreducible algebraic representation of $\SL_2(\C)$ of dimension $a$.
\vskip 10pt

As in \S \ref{s.segment-rep}, 
one can define a representation $\pi(\EE)$ for each extended multi-segment $\EE$.
It is defined by taking derivatives from an irreducible representation with explicit Langlands data. 
By results in \cite{AM}, one can show that $\pi(\EE)$ is irreducible or zero. 
The first main theorem is a reformulation of M{\oe}glin's construction of $A$-packets. 

\begin{thm}[Theorems \ref{reform}, \ref{character}]\label{main1}
Let $\psi = \oplus_\rho (\oplus_{i \in I_\rho} \rho \boxtimes S_{a_i} \boxtimes S_{b_i})$
be an $A$-parameter for $G_n$ of good parity.  
Set $A_i = (a_i+b_i)/2-1$ and $B_i = (a_i-b_i)/2$.
Choose a total order $\succ_\psi$ on $I_\rho$ satisfying that $B_i > B_j \implies i \succ_\psi j$.
Then 
\[
\Pi_\psi = 
\left\{ \pi(\EE) \;\middle|\; \supp(\EE) = \cup_{\rho}\{ [A_i,B_i]_{\rho} \}_{i \in (I_\rho,\succ_\psi)}
\right\} \setminus \{0\}. 
\]
Moreover, for $\EE$ with $\supp(\EE) = \cup_{\rho}\{ [A_i,B_i]_{\rho} \}_{i \in (I_\rho,\succ_\psi)}$, 
one can define $\eta_\EE \in \widehat{\Sc_\psi}$ explicitly (Definition \ref{def.character}) such that 
if $\pi(\EE) \not= 0$, then
\[
\pair{\cdot, \pi(\EE)}_\psi = \eta_\EE. 
\]
\end{thm}

When $B_i \geq 0$ for all $\rho$ and $i \in I_\rho$, 
this construction is exactly the same as M{\oe}glin's original construction. 
By Theorem \ref{main1}, the remaining problems we have to consider are: 

\begin{description}
\item[Problem A]
Determine precisely when $\pi(\EE)$ is nonzero. 

\item[Problem B]
Specify $\pi(\EE)$ precisely if it is nonzero. 
\end{description}
\vskip 10pt

The second main theorem reduces these problems to the \emph{non-negative} case, i.e., 
the case where $\EE = \cup_{\rho}\{ ([A_i,B_i]_{\rho}, l_i, \eta_i) \}_{i \in (I_\rho,\succ)}$ satisfies that 
$B_i \geq 0$ for all $\rho$ and $i \in I_\rho$. 

\begin{thm}[Theorem \ref{nonzero2}]\label{main2}
Let $\EE = \cup_{\rho}\{ ([A_i,B_i]_{\rho}, l_i, \eta_i) \}_{i \in (I_\rho,\succ)}$
be an extended multi-segment for $G_n$.
Take a positive integer $t$ such that $B_i+t \geq 0$ for all $\rho$ and $i \in I_\rho$. 
Define $\EE_t$ from $\EE$ 
by replacing each $([A_i,B_i]_{\rho}, l_i, \eta_i)$ with $([A_i+t,B_i+t]_{\rho}, l_i, \eta_i)$. 
Then the representation $\pi(\EE)$ is nonzero if and only if 
$\pi(\EE_t) \not= 0$ and 
the following condition holds for all $\rho$ and $i \in I_\rho$:
\[
\tag{$\star$}
B_i+l_i \geq \left\{
\begin{aligned}
&0 \iif B_i \in \Z,\\
&\half{1} \iif B_i \not\in \Z, \eta_i \not= (-1)^{\alpha_i}, \\
&-\half{1} \iif B_i \not\in \Z, \eta_i = (-1)^{\alpha_i}, 
\end{aligned}
\right. 
\]
where we set 
\[
\alpha_i = \sum_{j \in I_\rho,\; j \prec i}(A_j+B_j+1). 
\]
Moreover, in this case, 
the Langlands data for $\pi(\EE)$ are 
the ``shift'' of those of $\pi(\EE_t)$ by $-t$. 
\end{thm}
Therefore, we will consider Problems A and B for non-negative extended multi-segments $\EE$. 
\vskip 10pt

First, we consider Problem A for non-negative extended multi-segments $\EE$ for $G_n$. 
Since our parametrization $\pi(\EE)$ is the same as M{\oe}glin's one in this case, 
we can use several Xu's results. 
In \cite{X3}, Xu gave an algorithm to determine when M{\oe}glin's original parameterizations are nonzero. 
We will apply this algorithm to $\pi(\EE)$. 
When $\EE = \cup_{\rho}\{ ([A_i,B_i]_{\rho}, l_i, \eta_i) \}_{i \in (I_\rho,\succ)}$ is non-negative, 
the total order $\succ$ on $I_\rho$ is imposed only a weaker condition: 
\[
A_i > A_j, B_i > B_j \implies i \succ j.
\]
We call such an order $\succ$ on $I_\rho$ \emph{admissible}. 
To give an algorithm, Xu prepared three necessary conditions for $\pi(\EE) \not= 0$
on $([A_i,B_i]_{\rho}, l_i, \eta_i)$ and $([A_j,B_j]_{\rho}, l_j, \eta_j)$
for two adjacent elements $i \succ j$ in $I_\rho$ (Proposition \ref{nec}). 
Roughly speaking, Xu's algorithm to determine whether $\pi(\EE) \not= 0$ 
is reformulated as the third main theorem as follows: 

\begin{thm}[Theorem \ref{nonzero}]\label{main3}
Let $\EE$ be a non-negative extended multi-segment for $G_n$.
Then $\pi(\EE) \not= 0$ 
if and only if 
the three necessary conditions in Proposition \ref{nec} are satisfied for
every pair of two elements $i,j \in I_\rho$ 
which are adjacent with respect to some admissible order $\succ'$ on $I_\rho$ for all $\rho$. 
\end{thm}

Remark that in Xu's original algorithm, 
the non-vanishing of $\pi(\EE)$ is reduced to the ones of several representations $\{\pi(\EE^*)\}$. 
They are slightly easier than $\pi(\EE)$ itself, but are representations of bigger groups than $G_n$. 
If one considers honestly, the computational complexity of this algorithm is an exponential order. 
Since Theorem \ref{main3} does not require considering another representations, 
it would reduce the computational complexity. 
Moreover, it is better for applications such as Theorems \ref{main4} below. 
\vskip 10pt

Next, we consider Problem B for non-negative extended multi-segments $\EE$ for $G_n$. 
When $\EE = \cup_{\rho}\{ ([A_i,B_i]_{\rho}, l_i, \eta_i) \}_{i \in (I_\rho,\succ)}$ (with $B_i \geq 0$) satisfies that 
\begin{itemize}
\item
$A_i - B_i = A_j - B_j$ for any $i,j \in I_\rho$; or
\item
if $i \succ j$, then $A_i > A_j$ and $B_i > B_j$,
\end{itemize}
the Langlands data for $\pi(\EE)$ were obtained by Xu (\cite[Theorems 1.1, 1.2, 1.3]{X4}). 
However, they are already quite complicated. 
To attack Problem B, it might not be good to try to give the Langlands data for $\pi(\EE)$. 
Instead of this, in this paper, we will compute certain derivatives of $\pi(\EE)$, 
which determine $\pi(\EE)$ uniquely. 
A special case is given as follows. 

\begin{thm}[Theorem \ref{derivatives} (1)]\label{main4}
Let $\EE = \cup_{\rho}\{ ([A_i,B_i]_{\rho}, l_i, \eta_i) \}_{i \in (I_\rho,\succ)}$ 
be an extended multi-segment for $G_n$ such that $\pi(\EE) \not= 0$. 
Fix $\rho$. 
Assume that $B^{\max} = \max\{B_i \;|\; i \in I_\rho\} > 0$. 
Define $\EE^+$ from $\EE$ by replacing $[A_i,B_i]_\rho$ 
with $[A_i-1,B_i-1]_\rho$ for every $i \in I_\rho$ such that $B_i = B^{\max}$. 
If $\pi(\EE^+) \not= 0$, then 
\[
D_{\rho|\cdot|^{x_t}} \circ \dots \circ D_{\rho|\cdot|^{x_1}}\left(\pi(\EE) \right) \cong \pi(\EE^+)
\]
up to a multiplicity, 
where we write
\[
\bigsqcup_{\substack{i \in I_\rho \\ B_i = B^{\max}}} [A_i,B_i]_\rho 
= \{\rho|\cdot|^{x_1}, \dots, \rho|\cdot|^{x_t}\}
\]
as multi-sets such that $x_1 \leq \dots \leq x_t$ (so that $x_1 = B^{\max}$). 
\end{thm}

However, $\pi(\EE^+)$ is often zero even if $\EE$ is non-negative. 
To use this theorem, we turn the fact that 
$\Pi_\psi \cap \Pi_{\psi'} \not= \emptyset \centernot\implies \psi \cong \psi'$ into an advantage.
Namely, for given (non-negative) $\EE$ with $\pi(\EE) \not= 0$, 
we will construct $\EE^*$ with $\pi(\EE) \cong \pi(\EE^*)$ 
such that $\EE^*$ is in the situation of Theorem \ref{main4}. 
This construction is given in Algorithm \ref{alg+}, in which a key result is Theorem \ref{union}. 
By this algorithm and Theorem \ref{main4} together with \cite[Theorem 1.3]{X4} and \cite[Theorem 7.1]{AM}, 
one can compute the Langlands data for $\pi(\EE)$. 
\vskip 10pt

Finally, we prove a formula for the Aubert dual of $\pi(\EE)$. 
In \cite{Au}, Aubert defined an involution $\pi \mapsto \hat\pi$ on $\Irr(G_n)$. 
We call $\hat\pi$ the \emph{Aubert dual} of $\pi$. 
By Xu's result \cite[Appendix A]{X2} together with \cite[\S 4.1]{BLM}, 
we know that 
\[
\{\hat\pi \;|\; \pi \in \Pi_\psi\} = \Pi_{\hat\psi}
\]
for any $A$-parameter $\psi$ (of good parity), 
where $\hat\psi$ is defined from $\psi$ by exchanging two $\SL_2(\C)$-factors. 
In Definition \ref{hatE}, 
for any extended multi-segment $\EE = \cup_{\rho}\{ ([A_i,B_i]_{\rho}, l_i, \eta_i) \}_{i \in (I_\rho,\succ)}$, 
we will define another extended multi-segment 
$\hat\EE = \cup_{\rho}\{ ([A_i,-B_i]_{\rho}, \hat{l}_i, \hat\eta_i) \}_{i \in (I_\rho,\hat{\succ})}$ explicitly.

\begin{thm}[Theorem \ref{aubert}]
If $\pi(\EE) \not= 0$, then its Aubert dual is given by 
\[
\hat\pi(\EE) \cong \pi(\hat\EE).
\]
\end{thm}

Although it is applied to irreducible representations of Arthur type only, 
this theorem is much more efficient than the algorithm given in the previous paper (\cite[Algorithm 4.1]{AM}) in this case.
To prove this theorem, 
we compare our construction of $A$-packets with M{\oe}glin's original one in a certain case. 
By the same argument, a complete comparison is also given (Theorem \ref{compare}).
\vskip 10pt

This paper is organized as follows. 
In \S \ref{s.pre}, we review representation theory of classical groups, 
including the theory of derivatives and known results on local $A$-packets. 
In \S \ref{s.extended}, we introduce the notion of extended multi-segments
and prove Theorems \ref{main1} and \ref{main2}. 
We discuss about Problems A and B (for the non-negative case) 
in \S \ref{s.nonzero} and \ref{s.deform}, respectively. 
Finally, we state and prove a formula for the Aubert duals in \S \ref{s.aubert}. 
\vskip 10pt

\noindent
{\bf Acknowledgement.}
We would like to thank Bin Xu for answering our questions
and for suggesting his result for the proof of Theorem \ref{aubert}. 
We are grateful to Alberto M{\'i}nguez for useful discussions.
Thanks are also due to the referee for the careful readings and the helpful comments.
The author was supported by JSPS KAKENHI Grant Number 19K14494. 
\par
A Sage code for computing examples of several objects in this paper
is now available at: 
\[
\text{\url{https://github.com/atobe31/Local-A-packets}}
\]
We are grateful to Alexander Hazeltine, Baiying Liu and Chi-Heng Lo for pointing out many bugs. 
\vskip 10pt

\noindent
\textbf{Notation.}
Let $F$ be a non-archimedean local field of characteristic zero.
The normalized absolute value is denoted by $|\cdot|$, 
which is also regarded as a character of $\GL_d(F)$ via composing the determinant map. 
\par

Let $G_n$ be a split special odd orthogonal group $\SO_{2n+1}(F)$ or a symplectic group $\Sp_{2n}(F)$ 
of rank $n$ over $F$. 
For a smooth representation $\Pi$ of $G_n$ or $\GL_n(F)$ of finite length, 
we write $[\Pi]$ for the semisimplification of $\Pi$. 
Similarly, we denote by $\soc(\Pi)$ the \emph{socle} of $\Pi$, 
i.e., the maximal semisimple subrepresentation of $\Pi$.
The set of equivalence classes of irreducible smooth representations of a group $G$ is denoted by $\Irr(G)$. 
\par

We will often extend the set theoretical language to multi-sets. 
Namely, we write a multi-set as $\{x, \dots, x, y, \dots, y, \ldots \}$.
When we use a multi-set, we will mention it. 

\section{Preliminary}\label{s.pre}
In this section, we review representation theory of classical groups. 

\subsection{Langlands classification}
First, we recall some notation for representations of $\GL_n(F)$. 
Let $P$ be a standard parabolic subgroup of $\GL_n(F)$ 
with Levi subgroup $M \cong \GL_{n_1}(F) \times \dots \times \GL_{n_r}(F)$. 
For representations $\tau_1, \dots, \tau_r$ of $\GL_{n_1}(F), \dots, \GL_{n_r}(F)$, 
we denote by 
\[
\tau_1 \times \dots \times \tau_r \coloneqq \Ind_P^{\GL_n(F)}(\tau_1 \boxtimes \dots \boxtimes \tau_r)
\]
the normalized parabolically induced representation. 
\par

Let $\Cusp_\unit(\GL_d(F))$ be the set of equivalence classes of 
irreducible unitary supercuspidal representations of $\GL_d(F)$, 
and $\Cusp^\bot(\GL_d(F))$ be the subset consisting of self-dual elements.
Set $\Cusp_\unit = \cup_{d \geq 1}\Cusp_\unit(\GL_d(F))$ 
and $\Cusp^\bot = \cup_{d \geq 1}\Cusp^\bot(\GL_d(F))$. 
\par

A \emph{segment} $[x,y]_\rho$ is a set of supercuspidal representations of the form 
\[
[x,y]_\rho = \{ \rho|\cdot|^x, \rho|\cdot|^{x-1}, \dots, \rho|\cdot|^y\}, 
\]
where $\rho \in \Cusp_\unit(\GL_d(F))$ and $x,y \in \R$ such that $x-y \in \Z$ and $x \geq y$.
For a segment $[x,y]_\rho$, define a \emph{Steinberg representation} $\Delta_\rho[x,y]$ 
as a unique irreducible subrepresentation of 
\[
\rho|\cdot|^x \times \dots \times \rho|\cdot|^{y}.
\]
This is an essentially discrete series representation of $\GL_{d(x-y+1)}(F)$. 
Similarly, we define $Z_\rho[y,x]$ as a unique irreducible quotient of the same induced representation.
By convention, we set $\Delta_\rho[x,x+1]= Z_\rho[x+1,x]$ 
to be the trivial representation of the trivial group $\GL_0(F)$.
\par

The Langlands classification for $\GL_n(F)$ says that 
every $\tau \in \Irr(\GL_n(F))$ is a unique irreducible subrepresentation 
of $\Delta_{\rho_1}[x_1,y_1] \times \dots \times \Delta_{\rho_r}[x_r,y_r]$, 
where $\rho_i \in \Cusp_\unit$ for $i = 1, \dots, r$ such that $x_1+y_1 \leq \dots \leq x_r+y_r$. 
In this case, we write
\[
\tau = L(\Delta_{\rho_1}[x_1,y_1], \dots, \Delta_{\rho_r}[x_r,y_r]).
\]
\par

When $(x_{i,j})_{1 \leq i \leq t, 1 \leq j \leq d}$ satisfies that $x_{i,j} = x_{1,1} - i + j$, 
the irreducible representation 
$L(\Delta_{\rho}[x_{1,1},x_{t,1}], \dots, \Delta_{\rho}[x_{1,d},x_{t,d}])$
is called a \emph{(shifted) Speh representation} and is denoted by
\[
\begin{pmatrix}
x_{1,1} & \ldots & x_{1,d} \\
\vdots & \ddots & \vdots \\
x_{t,1} & \ldots & x_{t,d}
\end{pmatrix}_\rho
\coloneqq L(\Delta_{\rho}[x_{1,1},x_{t,1}], \dots, \Delta_{\rho}[x_{1,d},x_{t,d}]).
\] 
Note that it is isomorphic to the unique irreducible subrepresentation of 
$Z_\rho[x_{1,1}, x_{1,d}] \times \dots \times Z_\rho[x_{t,1}, x_{t,d}]$.
\par

An irreducibility criterion for the product of two Speh representations was obtained by Tadi{\'c} \cite{T-irr}. 
For a more general situation, see Lapid--M{\'i}nguez \cite{LM}. 
The following is a part of \cite[Theorem 1.1]{T-irr} or \cite[Corollary 6.10]{LM}. 
\begin{thm}\label{speh}
Let $(x_{i,j})_{\rho}$ and $(y_{i',j'})_{\rho'}$ be two Speh representations, 
where $1 \leq i \leq t, 1 \leq j \leq d$ and $1 \leq i' \leq t', 1 \leq j' \leq d'$. 
If the parabolically induced representation 
\[
\begin{pmatrix}
x_{1,1} & \ldots & x_{1,d} \\
\vdots & \ddots & \vdots \\
x_{t,1} & \ldots & x_{t,d}
\end{pmatrix}_\rho 
\times 
\begin{pmatrix}
y_{1,1} & \ldots & y_{1,d'} \\
\vdots & \ddots & \vdots \\
y_{t',1} & \ldots & y_{t',d'}
\end{pmatrix}_{\rho'}
\]
is reducible, then
the segments $[x_{1,d},x_{t,1}]_\rho$ and $[y_{1,d'}, y_{t',1}]_{\rho'}$ are \emph{linked}, 
i.e., $\rho \cong \rho'$ and 
$[x_{1,d},x_{t,1}]_\rho \not\subset [y_{1,d'}, y_{t',1}]_{\rho'}$, 
$[x_{1,d},x_{t,1}]_\rho \not\supset [y_{1,d'}, y_{t',1}]_{\rho'}$
but $[x_{1,d},x_{t,1}]_\rho \cup [y_{1,d'}, y_{t',1}]_{\rho'}$ is also a segment.
\end{thm}

Next we recall some notation for representations of classical group $G_n$. 
Fix an $F$-rational Borel subgroup of $G_n$. 
Let $P$ be a standard parabolic subgroup of $G_n$ with Levi subgroup 
$M \cong \GL_{n_1}(F) \times \dots \times \GL_{n_r}(F) \times G_{n_0}$. 
For representations $\tau_1, \dots, \tau_r$ and $\pi_0$ 
of $\GL_{n_1}(F), \dots, \GL_{n_r}(F)$ and of $G_{n_0}$, respectively, 
denote by
\[
\tau_1 \times \dots \times \tau_r \rtimes \pi_0 
\coloneqq
\Ind_P^{G_n}(\tau_1 \boxtimes \dots \boxtimes \tau_r \boxtimes \pi_0)
\]
the normalized parabolically induced representation.
\par

The Langlands classification for $G_n$ says that 
every $\pi \in \Irr(G_n)$ is a unique irreducible subrepresentation 
of $\Delta_{\rho_1}[x_1,y_1] \times \dots \times \Delta_{\rho_r}[x_r,y_r] \rtimes \pi_0$, 
where 
\begin{itemize}
\item
$\rho_1, \dots, \rho_r \in \Cusp_\unit$; 
\item
$x_1+y_1 \leq \dots \leq x_r+y_r < 0$; 
\item
$\pi_0$ is an irreducible tempered representation of $G_{n_0}$.
\end{itemize}
In this case, we write
\[
\pi = L(\Delta_{\rho_1}[x_1,y_1], \dots, \Delta_{\rho_r}[x_r,y_r]; \pi_0),
\]
and call $(\Delta_{\rho_1}[x_1,y_1], \dots, \Delta_{\rho_r}[x_r,y_r]; \pi_0)$
the \emph{Langlands data} for $\pi$. 
See \cite{K} for more details. 

\subsection{Derivatives and socles}
Following Jantzen \cite{J-GL, J-temp, J-dual}, M{\'i}nguez \cite{M, LM} and 
our previous work \cite{AM}, 
we introduce the notion of \emph{$\rho|\cdot|^x$-derivatives}.
\par

For a smooth representation $\pi$ of $G_n$ of finite length, 
denote by $\Jac_{P_d}(\pi)$ its Jacquet module 
along the standard parabolic subgroup $P_d$ with Levi subgroup isomorphic to $\GL_d(F) \times G_{n-d}$. 
Fix $\rho \in \Cusp^\bot(\GL_d(F))$.
For $x \in \R$,
the \emph{$\rho|\cdot|^x$-derivative} $D_{\rho|\cdot|^x}(\pi)$ 
is a semisimple representation of $G_{n-d}$ satisfying that 
\[
[\Jac_{P_{d}}(\pi)] = \rho|\cdot|^x \boxtimes D_{\rho|\cdot|^x}(\pi) + \sum_{i} \tau_i \boxtimes \pi_i, 
\]
where $\tau_i \boxtimes \pi_i$ is an irreducible representation of $\GL_{d}(F) \times G_{n-d}$ 
such that $\tau_i \not \cong \rho|\cdot|^x$. 
We also set $D_{\rho|\cdot|^x}^{(0)}(\pi) = \pi$ and 
\[
D_{\rho|\cdot|^x}^{(k)}(\pi) = \frac{1}{k}D_{\rho|\cdot|^x} \circ D_{\rho|\cdot|^x}^{(k-1)}(\pi)
= \frac{1}{k!} \underbrace{D_{\rho|\cdot|^x} \circ \dots \circ D_{\rho|\cdot|^x}}_k (\pi).
\]
When $D_{\rho|\cdot|^x}^{(k)}(\pi) \not= 0$ but $D_{\rho|\cdot|^x}^{(k+1)}(\pi) = 0$, 
we call $D_{\rho|\cdot|^x}^{(k)}(\pi)$ the \emph{highest $\rho|\cdot|^x$-derivative} of $\pi$.
We say that $\pi$ is \emph{$\rho|\cdot|^x$-reduced} if $D_{\rho|\cdot|^x}(\pi) = 0$. 
\par

On the other hand, for a representation $\Pi$ of finite length, 
we denote by $\soc(\Pi)$ the \emph{socle} of $\Pi$, i.e., the maximal semisimple subrepresentation of $\Pi$. 
When $\Pi = (\rho|\cdot|^x)^r \rtimes \pi$, we shortly write 
\[
S_{\rho|\cdot|^x}^{(r)}(\pi) = \soc\left( (\rho|\cdot|^x)^r \rtimes \pi \right).
\]

\begin{thm}[{\cite[Lemma 3.1.3]{J-temp}}, {\cite[Propositions 3.3, 6.1, Theorem 7.1]{AM}}]\label{der}
Suppose that $x \not= 0$ so that $\rho|\cdot|^x$ is not self-dual. 
Let $\pi$ be an irreducible representation of $G_n$. 

\begin{enumerate}
\item
The highest $\rho|\cdot|^x$-derivative $D_{\rho|\cdot|^x}^{(k)}(\pi)$ is irreducible. 

\item
The socle $S_{\rho|\cdot|^x}^{(r)}(\pi)$ is irreducible for any $r \geq 0$. 

\item
They are related as
\[
\pi = S_{\rho|\cdot|^x}^{(k)}\left( D_{\rho|\cdot|^x}^{(k)}(\pi) \right)
\]
and
\[
D_{\rho|\cdot|^x}^{(k+r)} \left(S_{\rho|\cdot|^x}^{(r)}(\pi)\right) = D_{\rho|\cdot|^x}^{(k)}(\pi).
\]

\item
The Langlands data for $D_{\rho|\cdot|^x}^{(k)}(\pi)$ and $S_{\rho|\cdot|^x}^{(k)}(\pi)$
can be described from those for $\pi$ explicitly. 
\end{enumerate}
\end{thm}

When $x=0$, the $\rho$-derivative is difficult in general. 
As alternatives of $\rho$-derivative, 
following \cite[Section 3.4]{AM},
we define the \emph{$\Delta_\rho[0,-1]$-derivative $D_{\Delta_\rho[0,-1]}^{(k)}(\pi)$}
and the \emph{$Z_\rho[0,1]$-derivative $D_{Z_\rho[0,1]}^{(k)}(\pi)$} 
as semisimple representations of $G_{n-2dk}$ satisfying 
\[
[\Jac_{P_{2dk}}(\pi)] = \Delta_\rho[0,-1]^k \boxtimes D_{\Delta_\rho[0,-1]}^{(k)}(\pi)
+ Z_\rho[0,1]^k \boxtimes D_{Z_\rho[0,1]}^{(k)}(\pi)
+ \sum_{i} \tau_i \boxtimes \pi_i, 
\]
where $\tau_i \boxtimes \pi_i$ is an irreducible representation of $\GL_{2dk}(F) \times G_{n-2dk}$ 
such that $\tau_i \not \cong \Delta_{\rho}[0,-1]^k, Z_\rho[0,1]^k$.
On the other hand, we set 
\[
S_{\Delta_\rho[0,-1]}^{(r)}(\pi) = \soc( \Delta_\rho[0,-1]^r \rtimes \pi), 
\quad
S_{Z_\rho[0,1]}^{(r)}(\pi) = \soc( Z_\rho[0,1]^r \rtimes \pi). 
\]

\begin{thm}[{\cite[Proposition 3.7]{AM}}]\label{der2}
Let $\pi$ be an irreducible representation of $G_n$. 
Suppose that $\pi$ is $\rho|\cdot|^{-1}$-reduced (\resp $\rho|\cdot|^1$-reduced). 
Then the same assertions in Theorem \ref{der} (1)--(3) hold when $\rho|\cdot|^x$ is replaced with 
$\Delta_\rho[0,-1]$ (\resp $Z_\rho[0,1]$). 
\end{thm}

The following is a special case of \cite[Lemma 3.5]{AM}. 
\begin{lem}\label{compose}
Let $\pi$ be an irreducible representation of $G_n$. 
Assume that 
\begin{itemize}
\item
$\pi$ is $\rho|\cdot|^{-1}$-reduced (\resp $\rho|\cdot|^1$-reduced); 
\item
$D_\rho^{(1)}(\pi)$ is the highest $\rho$-derivative of $\pi$; 
\item
$D_{\rho|\cdot|^{-1}}^{(1)}(D_\rho^{(1)}(\pi))$ (\resp $D_{\rho|\cdot|^{1}}^{(1)}(D_\rho^{(1)}(\pi))$)
is the highest $\rho|\cdot|^{-1}$-derivative (\resp $\rho|\cdot|^{1}$-derivative) of $D_\rho^{(1)}(\pi)$. 
\end{itemize}
Then $D_{\Delta_{\rho}[0,-1]}^{(1)}(\pi) = D_{\rho|\cdot|^{-1}} \circ D_\rho(\pi)$
(\resp $D_{Z_{\rho}[0,1]}^{(1)}(\pi) = D_{\rho|\cdot|^{1}} \circ D_\rho(\pi)$). 
\end{lem}

Finally, for a segment $[x,y]_\rho$, we set 
\[
D_{\rho|\cdot|^{y}, \dots, \rho|\cdot|^{x}}(\pi) 
= D_{\rho|\cdot|^{x}} \circ \dots \circ D_{\rho|\cdot|^{y}}(\pi)
\]
and 
\[
S_{\rho|\cdot|^{x}, \dots, \rho|\cdot|^{y}}(\pi) 
= S_{\rho|\cdot|^{y}}^{(1)} \circ \dots \circ S_{\rho|\cdot|^{x}}^{(1)}(\pi).
\]
For example, if $\sigma$ is supercuspidal and $y < x < 0$, 
then 
\begin{align*}
&D_{\rho|\cdot|^{y}, \dots, \rho|\cdot|^{x}}(Z_\rho[y,x] \rtimes \sigma) 
= \sigma, \quad
S_{\rho|\cdot|^{x}, \dots, \rho|\cdot|^{y}}(\sigma) 
= L(\rho|\cdot|^y, \dots, \rho|\cdot|^x; \sigma).
\end{align*}
Similarly, we consider $D_{\rho|\cdot|^{x}, \dots, \rho|\cdot|^{y}}$ and $S_{\rho|\cdot|^{y}, \dots, \rho|\cdot|^{x}}$. 
When we consider these operators, the adjacent exponents always differ by $\pm1$.

\subsection{Terminologies of $A$-parameters}
Denote by $\widehat{G}_n$ the complex dual group of $G_n$. 
Namely, $\widehat{G}_n = \Sp_{2n}(\C)$ if $G_n = \SO_{2n+1}(F)$, and 
$\widehat{G}_n = \SO_{2n+1}(\C)$ if $G_n = \Sp_{2n}(F)$.  
Recall that an $A$-parameter for $G_n$ is a $\widehat{G}_n$-conjugacy class of admissible homomorphisms
\[
\psi \colon W_F \times \SL_2(\C) \times \SL_2(\C) \rightarrow \widehat{G}_n
\]
such that the image of the Weil group $W_F$ is bounded.
By composing with the standard representation of $\widehat{G}_n$, 
we can regard $\psi$ as a representation of $W_F \times \SL_2(\C) \times \SL_2(\C)$. 
It decomposes as 
\[
\psi = \bigoplus_\rho\left(\bigoplus_{i \in I_\rho} \rho \boxtimes S_{a_i} \boxtimes S_{b_i}\right), 
\]
where 
\begin{itemize}
\item
$\rho$ runs over $\Cusp_\unit$,
which is identified with an irreducible bounded representation of $W_F$ by the local Langlands correspondence for $\GL_d(F)$; 
\item
$S_a$ is the unique irreducible algebraic representation of $\SL_2(\C)$ of dimension $a$.
\end{itemize}
Notice that $a_i$ and $b_i$ depend on $\rho$, but we do not write it.
We write $\rho \boxtimes S_a = \rho \boxtimes S_a \boxtimes S_1$ 
and $\rho = \rho \boxtimes S_1 \boxtimes S_1$ for short. 
\par

Let $\psi$ be as above. 
We say that 
$\psi$ is \emph{of good parity} 
if $\rho \boxtimes S_{a_i} \boxtimes S_{b_i}$ is self-dual of the same type as $\psi$ for any $\rho$ and $i \in I_\rho$, 
i.e., 
\begin{itemize}
\item
$\rho \in \Cusp^\bot$ is orthogonal and $a_i+b_i \equiv 0 \bmod 2$ if $G_n = \Sp_{2n}(F)$
(\resp $a_i+b_i \equiv 1 \bmod 2$ if $G_n = \SO_{2n+1}(F)$); or 
\item
$\rho \in \Cusp^\bot$ is symplectic and $a_i+b_i \equiv 1 \bmod 2$ if $G_n = \Sp_{2n}(F)$
(\resp $a_i+b_i \equiv 0 \bmod 2$ if $G_n = \SO_{2n+1}(F)$).
\end{itemize}
Let $\Psi(G_n) \supset \Psi_\gp(G_n)$ be the sets of equivalence classes of $A$-parameters
and $A$-parameters of good parity, respectively. 
Also, we let $\Phi_\temp(G_n)$ be the subset of $\Psi(G_n)$ consisting of
\emph{tempered} $A$-parameters, i.e., 
$A$-parameters $\phi$ which are trivial on the second $\SL_2(\C)$. 
Finally, we set $\Phi_\gp(G_n) = \Psi_\gp(G_n) \cap \Phi_\temp(G_n)$.
\par

For $\psi = \oplus_\rho(\oplus_{i \in I_\rho} \rho \boxtimes S_{a_i} \boxtimes S_{b_i}) \in \Psi_\gp(G_n)$, 
define the \emph{enhanced component group} by
\[
\AA_\psi = \bigoplus_{\rho} \bigoplus_{i \in I_\rho} (\Z/2\Z) \alpha_{\rho, i}, 
\]
i.e., $\AA_\psi$ is a $(\Z/2\Z)$-vector space with a canonical basis $\alpha_{\rho, i}$ corresponding to 
$\rho \boxtimes S_{a_i} \boxtimes S_{b_i}$.  
The \emph{component group} $\Sc_\psi$ is the quotient of $\AA_\psi$ by the subgroup generated by 
\begin{itemize}
\item
$\alpha_{\rho, i} + \alpha_{\rho, j}$ 
such that $\rho \boxtimes S_{a_i} \boxtimes S_{b_i} = \rho \boxtimes S_{a_j} \boxtimes S_{b_j}$; and
\item
$z_\psi = \sum_{\rho}\sum_{i \in I_\rho} \alpha_{\rho,i}$, 
which is called the \emph{central element} of $\AA_\psi$.
\end{itemize}
Let $\widehat{\Sc_\psi} \subset \widehat{\AA_\psi}$ be the Pontryagin duals of $\Sc_\psi$ and $\AA_\psi$, 
respectively. 
When $\ep \in \widehat{\AA_\psi}$, 
we write $\ep(\rho \boxtimes S_{a_i} \boxtimes S_{b_i}) = \ep(\alpha_{\rho, i}) \in \{\pm1\}$.

\subsection{$A$-packets}\label{Apacket}
Let $\Irr_\unit(G_n)$ (\resp $\Irr_\temp(G_n)$) be the set of
equivalence classes of irreducible unitary (\resp tempered) representations of $G_n$.
To an $A$-parameter $\psi \in \Psi(G_n)$, 
Arthur \cite[Theorem 2.2.1]{Ar} associated an $A$-packet $\Pi_\psi$, which is a finite multi-set over $\Irr_\unit(G_n)$.
It is characterized by twisted endoscopic character identities. 
Moreover, if $\phi \in \Phi_\temp(G_n)$ is a tempered $A$-parameter, then
$\Pi_\phi$ is a subset of $\Irr_\temp(G_n)$ and
\[
\Irr_\temp(G_n) = \bigsqcup_{\phi \in \Phi_\temp(G_n)} \Pi_\phi
\quad \text{(disjoint union)}.
\]
\par

M{\oe}glin \cite{Moe11c} showed that $\Pi_\psi$ is multiplicity-free, i.e., a subset of $\Irr_\unit(G_n)$. 
To prove it, she constructed $\Pi_\psi$ concretely. 
The purpose of this paper is to describe $\Pi_\psi$ more explicitly. 
In general, $\Pi_{\psi_1} \cap \Pi_{\psi_2} \not= \emptyset$ even if $\psi_1 \not\cong \psi_2$. 
We will use this fact to give an algorithm to compute elements of $\Pi_\psi$.
\par

If $\psi = \oplus_\rho (\oplus_{i \in I_\rho} \rho \boxtimes S_{a_i} \boxtimes S_{b_i})$, 
set 
\[
\tau_{\psi} = \bigtimes_\rho\bigtimes_{i \in I_\rho} 
\begin{pmatrix}
\half{a_i-b_i} & \ldots & \half{a_i+b_i}-1 \\
\vdots & \ddots  & \vdots \\
-\half{a_i+b_i}+1 & \ldots & -\half{a_i-b_i} 
\end{pmatrix}_{\rho}
\]
to be a product of (unitary) Speh representations, 
which is an irreducible unitary representation of $\GL_d(F)$ with $d = \dim(\psi)$.
\par

Any $\psi \in \Psi(G_n)$ can be decomposed as 
\[
\psi = \psi_1 \oplus \psi_0 \oplus \psi_1^\vee
\]
for some representation $\psi_1$ of $W_F \times \SL_2(\C) \times \SL_2(\C)$ and some $\psi_0 \in \Psi_\gp(G_{n_0})$. 

\begin{thm}[{\cite[Theorem 6]{Moe06a}, \cite[Proposition 8.11]{X2}}]\label{irr}
For $\pi_0 \in \Pi_{\psi_0}$, 
the parabolically induced representation $\tau_{\psi_1} \rtimes \pi_0$
is irreducible and independent of the choice of $\psi_1$. 
Moreover, 
\[
\Pi_\psi = \left\{ \tau_{\psi_1} \rtimes \pi_0 \;\middle|\; \pi_0 \in \Pi_{\psi_0} \right\}.
\]
\end{thm}

Hence we may only consider $\psi \in \Psi_\gp(G_n)$. 
\par

Through this paper, we implicitly fix a Whittaker datum for $G_n$. 
Let $\psi \in \Psi_\gp(G_n)$ so that we have defined the component group $\Sc_\psi$. 
Arthur \cite[Theorem 2.2.1]{Ar} gave a map
\[
\Pi_\psi \rightarrow \widehat{\Sc_\psi},\; \pi \mapsto \pair{\cdot, \pi}_\psi. 
\]
If $\psi = \phi \in \Phi_\gp(G_n)$ is tempered, this map is bijective. 
When $\pi \in \Pi_\phi$ corresponds to $\ep \in \widehat{\Sc_\phi}$, we write $\pi = \pi(\phi, \ep)$.
\par

Now we recall some relation between $A$-parameters and derivatives. 
Let $\psi = \oplus_\rho (\oplus_{i \in I_\rho} \rho \boxtimes S_{a_i} \boxtimes S_{b_i}) \in \Psi_\gp(G_n)$. 
Define
\[
A_i = \half{a_i+b_i}-1, \quad B_i = \half{a_i-b_i}. 
\]
Note that the definition of $B_i$ is not the same as the one of M{\oe}glin and Xu.
By the compatibility of twisted endoscopic character identities and Jacquet modules 
(\cite[\S 6]{X1}, see also \cite[\S 8]{X2}), 
Xu gave the following proposition. 

\begin{prop}[{\cite[Proposition 8.3 (ii)]{X2}}]\label{k=<m}
Let $\psi = \oplus_\rho (\oplus_{i \in I_\rho} \rho \boxtimes S_{a_i} \boxtimes S_{b_i}) \in \Psi_\gp(G_n)$. 
If $\pi \in \Pi_\psi$ satisfies that $D_{\rho|\cdot|^x}^{(k)}(\pi) \not= 0$, 
then 
\[
k \leq \#\{i \in I_\rho \;|\; x = B_i\}. 
\]
\end{prop}

We call a total order $\succ_\psi$ on $I_\rho$ \emph{admissible} 
if it satisfies the following condition: 
\[
\tag{$\PP$}
\text{
For $i,j \in I_\rho$, 
if $A_i > A_j$ and $B_i > B_j$, 
then $i \succ_\psi j$.
}
\] 
More strongly, we also consider the following condition: 
\[
\tag{$\PP'$}
\text{
For $i,j \in I_\rho$, 
if $B_i > B_j$, then $i \succ_\psi j$.
}
\] 
Using these orders, we have the following theorem. 
\begin{thm}\label{compatible}
Let $\psi = \oplus_\rho (\oplus_{i \in I_\rho} \rho \boxtimes S_{a_i} \boxtimes S_{b_i}) \in \Psi_\gp(G_n)$. 
Fix an admissible order $\succ_\psi$ on $I_\rho$ for each $\rho$, 
and write $I_\rho = \{1, \dots, m\}$ with $1 \prec_\psi \dots \prec_\psi m=m_\rho$. 
Assume that $\succ_\psi$ on $I_\rho$ satisfies $(\PP')$ if $B_i < 0$ for some $i \in I_\rho$. 
Take 
$\psi' = \oplus_\rho (\oplus_{i \in I_\rho} \rho \boxtimes S_{a_i+2t_i} \boxtimes S_{b_i})$
with non-negative integers $t_1, \dots, t_m$ such that 
\[
0 \leq B_1+t_1 \leq A_1+t_1 < B_2+t_2 \leq A_2+t_2 < \dots < B_m+t_m \leq A_m+t_m. 
\]
Then 
\[
\Pi_\psi = \left\{
\circ_\rho \circ_{i \in I_\rho}
\left(
D_{\rho|\cdot|^{B_i+1}, \dots, \rho|\cdot|^{A_i+1}}
\circ \dots \circ 
D_{\rho|\cdot|^{B_i+t_i}, \dots, \rho|\cdot|^{A_i+t_i}}
\right)
(\pi') \;\middle|\;
\pi' \in \Pi_{\psi'}
\right\} \setminus\{0\}.
\]
Here, we write $\circ_{i \in I_\rho} D_i = D_m \circ \dots \circ D_1$. 
\end{thm}
\begin{proof}
For simplicity, 
write $D = \circ_\rho \circ_{i \in I_\rho}
\left(
D_{\rho|\cdot|^{B_i+1}, \dots, \rho|\cdot|^{A_i+1}}
\circ \dots \circ 
D_{\rho|\cdot|^{B_i+t_i}, \dots, \rho|\cdot|^{A_i+t_i}}
\right)$. 
When $B_i \geq 0$ for all $\rho$ and $i \in I_\rho$, 
the assertion is already known (see \cite[\S 8]{X2}). 
In general, by the compatibility of 
twisted endoscopic character identities and Jacquet modules \cite[\S 6]{X1}, 
as virtual representations, 
we have
\[
\tag{$\ast$}
\sum_{\pi \in \Pi_\psi} \pair{s_\psi, \pi}_\psi\pi 
=
\sum_{\pi' \in \Pi_{\psi'}} \pair{s_{\psi'}, \pi'}_{\psi'}D(\pi')
\]
for certain elements $s_\psi \in \Sc_\psi$ and $s_{\psi'} \in \Sc_{\psi'}$. 
When $B_i \geq 0$ or $B_i \in (1/2)\Z \setminus \Z$ for any $\rho$ and $i \in I_\rho$,
by Theorem \ref{der} and Proposition \ref{k=<m}, we see that 
\begin{enumerate}
\item
$D(\pi')$ is irreducible or zero for any $\pi' \in \Pi_{\psi'}$; 
\item
for $\pi'_1, \pi'_2 \in \Pi_{\psi'}$,  
\[
D(\pi'_1) \cong D(\pi'_2) \not= 0 \implies \pi'_1 \cong \pi'_2.
\]
\end{enumerate}
On the other hand, since the $\rho$-derivative $D_{\rho}$ is not injective in the sense of (2), 
there might be a cancellation in the right hand side of $(\ast)$. 
However, since $A_i \geq 0$ for all $\rho$ and $i \in I_\rho$, 
if $D_\rho$ appears in the definition of $D$, 
it appears as $D_{\rho, \rho|\cdot|^1} = D_{\rho|\cdot|^1} \circ D_\rho$.
By Lemma \ref{compose}
together with the condition $(\PP')$ and Proposition \ref{k=<m}, 
we may replace $D_{\rho, \rho|\cdot|^1}$ with $D_{Z_\rho[0,1]}$.
By Theorem \ref{der2}, we have the same conclusions as (1) and (2), 
and hence there is no cancellation in $(\ast)$.
\end{proof}

\section{Extended multi-segments and their associated representations}\label{s.extended}
To describe $A$-packets, we introduce the notion of extended multi-segments, 
and define associated representations. 

\subsection{Extended segments}
In this subsection, we define extended (multi-)segments.

\begin{defi}\label{segments}
\begin{enumerate}
\item
An \emph{extended segment} is a triple $([A,B]_\rho, l, \eta)$,
where
\begin{itemize}
\item
$[A,B]_\rho = \{\rho|\cdot|^A, \dots, \rho|\cdot|^B \}$ is a segment; 
\item
$l \in \Z$ with $0 \leq l \leq \half{b}$, where $b = \#[A,B]_\rho = A-B+1$; 
\item
$\eta \in \{\pm1\}$. 
\end{itemize}

\item
An \emph{extended multi-segment} for $G_n$ is an equivalence class of multi-sets of extended segments 
\[
\EE = \cup_{\rho}\{ ([A_i,B_i]_{\rho}, l_i, \eta_i) \}_{i \in (I_\rho,\succ)}
\]
such that 
\begin{itemize}
\item
$I_\rho$ is a totally ordered finite set with a fixed admissible order $\succ$, 
i.e., a total order satisfying that
if $A_i > A_j$ and $B_i > B_j$, then $i \succ j$;

\item
$A_i + B_i \geq 0$ for all $\rho$ and $i \in I_\rho$; 

\item
as a representation of $W_F \times \SL_2(\C) \times \SL_2(\C)$, 
\[
\psi = \bigoplus_\rho \bigoplus_{i \in I_\rho} \rho \boxtimes S_{a_i} \boxtimes S_{b_i} \in \Psi_\gp(G_n), 
\]
where $a_i = A_i+B_i+1$ and $b_i = A_i-B_i+1$; 

\item
a sign condition
\[
\prod_{\rho} \prod_{i \in I_\rho} (-1)^{[\half{b_i}]+l_i} \eta_i^{b_i} = 1
\]
holds. 
\end{itemize}

\item
Two extended segments $([A,B]_\rho, l, \eta)$ and $([A',B']_{\rho'}, l', \eta')$ are \emph{equivalent} 
if 
\begin{itemize}
\item
$[A,B]_\rho = [A',B']_{\rho'}$; 
\item
$l = l'$; and
\item
$\eta = \eta'$ whenever $l = l' < \half{b}$. 
\end{itemize}
Similarly, $\EE = \cup_{\rho}\{ ([A_i,B_i]_{\rho}, l_i, \eta_i) \}_{i \in (I_\rho,\succ)}$ 
and $\EE' = \cup_{\rho}\{ ([A'_i,B'_i]_{\rho}, l'_i, \eta'_i) \}_{i \in (I_\rho,\succ)}$ are \emph{equivalent}
if $([A_i,B_i]_\rho, l_i, \eta_i)$ and $([A'_i,B'_i]_{\rho}, l'_i, \eta'_i)$ are equivalent for all $\rho$ and $i \in I_\rho$.

\item
The \emph{support} of $\EE$ is the multi-segment 
\[
\supp(\EE) = \cup_{\rho}\{ [A_i,B_i]_{\rho} \}_{i \in (I_\rho,\succ)}.
\]

\end{enumerate}
\end{defi}

Let $\EE = \cup_{\rho}\EE_\rho$ be an extended multi-segment
with $\EE_\rho = \{ ([A_i,B_i]_{\rho}, l_i, \eta_i) \}_{i \in (I_\rho,\succ)}$. 
We regard $\EE_\rho$ as the following symbol:
When $\EE_\rho = \{([A,B]_\rho, l, \eta)\}$ is a singleton, we write
\[
\EE_\rho = 
\left(
\begin{array}{rcl}
\underbrace{\overset{B}{\lhd} \lhd \cdots \overset{B+l-1}{\lhd}}_l 
&
\overset{B+l}{\odot} \odot \cdots \odot \overset{A-l} \odot 
&
\underbrace{\overset{A-l+1}{\rhd} \cdots \rhd \overset{A}{\rhd}}_l
\end{array}
\right)_\rho,
\] 
where $\odot$ is replaced with $\oplus$ and $\ominus$ alternately, 
starting with $\oplus$ if $\eta = +1$ (\resp $\ominus$ if $\eta = -1$).
In general, we put each symbol vertically. 
For the meaning of ``vertically'', see the following example.

\begin{ex}\label{ex1}
When $\EE_\rho = \{([A_i,B_i]_\rho,l_i, \eta_i)\}_{i \in \{1 \prec 2 \prec 3 \prec 4\}}$
with 
\begin{itemize}
\item
$[A_1,B_1] = [3,1]$, $[A_2,B_2] = [5,2]$, $[A_3,B_3] = [6,3]$ and $[A_4,B_4] = [5,4]$; 
\item
$(l_1,l_2,l_3,l_4) = (0,1,2,0)$ and $(\eta_1, \eta_2, \eta_4) = (-1,-1,-1)$, 
\end{itemize}
the symbol is 
\[
\EE_\rho = 
\bordermatrix{
& 1 & 2 & 3 & 4 & 5 & 6\cr
& \ominus & \oplus & \ominus &&& \cr
& & \lhd & \ominus & \oplus & \rhd & \cr
& & & \lhd & \lhd & \rhd & \rhd \cr
& & & & \ominus & \oplus & 
}_\rho.
\]
This symbol does not depend on $\eta_3$.
\end{ex}

As in this example, 
the symbol is determined by the equivalence class of extended multi-segments. 
The number of $\ominus$ appearing in $\EE_\rho$ is odd 
if and only if $\prod_{i \in I_\rho} (-1)^{[\half{b_i}]+l_i} \eta_i^{b_i} = -1$. 
Hence the sign condition in Definition \ref{segments} (2) means that 
$\ominus$ appears with even times among all symbols. 
\par

\begin{rem}
The definition of extended (multi-)segments 
is derived from M{\oe}glin's original parametrization of local $A$-packets 
\cite{Moe06a, Moe06b, Moe09, Moe10, Moe11c}. 
More precisely, instead of $([A,B]_\rho, l, \eta)$, 
she considered $([A, |B|]_\rho, \zeta, l ,\eta)$ with $\zeta \in \{\pm1\}$ such that $\zeta|B| = B$.
Taking the absolute value of $B$ makes her construction inaccessible. 
(See e.g., \cite[Section 6]{X2}.)
\end{rem}

\subsection{Associated representations}\label{s.segment-rep}
Let $\EE = \cup_\rho \{([A_i,B_i]_{\rho}, l_i, \eta_i)\}_{i \in (I_\rho,\succ)}$ be an extended multi-segment for $G_n$. 
Assume that the admissible order $\succ$ on $I_\rho$ satisfies $(\PP')$ if $B_i < 0$ for some $i \in I_\rho$. 
We will define a representation $\pi(\EE)$ as follows. 
First, we assume that 
\begin{itemize}
\item
for $i,j \in I_\rho$, if $i \succ j$, then $B_i > A_j$; 
\item
$B_i \geq 0$ for any $i \in I_\rho$.  
\end{itemize}
In this case, we define 
\begin{align*}
\pi(\EE) = \soc \left(
\bigtimes_\rho \bigtimes_{i \in I_\rho}
\begin{pmatrix}
B_i & \ldots & B_i + l_i -1\\
\vdots & \ddots & \vdots \\
-A_i  & \ldots & -(A_i-l_i+1)
\end{pmatrix}_\rho
\rtimes
\pi(\phi, \ep)
\right)
\end{align*}
with 
\[
\phi = \bigoplus_\rho \bigoplus_{i \in I_\rho} 
\rho \boxtimes \left( S_{2(B_i+l_i)+1} \oplus \dots \oplus S_{2(A_i-l_i)+1} \right)
\]
and $\ep(\rho \boxtimes S_{2(B_i+l_i+k)+1}) = (-1)^k \eta_i$ for $0 \leq k \leq b_i - 2l_i - 1$.
Note that 
\begin{itemize}
\item
the shifted Speh representations appearing in the above induced representation
are commutative to each other by Theorem \ref{speh}; 
\item
the parabolically induced representation is isomorphic to 
a subrepresentation of a certain standard module 
so that $\pi(\EE)$ is an irreducible representation. 
\end{itemize}
In general, 
take non-negative integers $t_1, \dots, t_m$ such that 
$\EE' = \cup_\rho \{([A_i+t_i,B_i+t_i]_{\rho}, l_i, \eta_i)\}_{i \in (I_\rho,\succ)}$ satisfies the above conditions, 
and define 
\[
\pi(\EE) = 
\circ_\rho \circ_{i \in I_\rho}
\left(
D_{\rho|\cdot|^{B_i+1}, \dots, \rho|\cdot|^{A_i+1}}
\circ \dots \circ 
D_{\rho|\cdot|^{B_i+t_i}, \dots, \rho|\cdot|^{A_i+t_i}}
\right)
(\pi(\EE')).
\]
Here, if $I_\rho = \{1,\dots,m\}$ with $1 \prec \dots \prec m$, 
then $\circ_{i \in I_\rho}D_i$ means that $D_m \circ \dots \circ D_1$. 
This definition does not depend on the choice of $t_i$. 
By Theorems \ref{der} and \ref{der2}, we see that $\pi(\EE)$ is irreducible or zero
(see the proof of Theorem \ref{compatible}). 
For examples, see \S \ref{ex.sc} below. 
\par

The following is a reformulation of M{\oe}glin's construction of $A$-packets. 
\begin{thm}\label{reform}
Let $\psi = \oplus_\rho (\oplus_{i \in I_\rho} \rho \boxtimes S_{a_i} \boxtimes S_{b_i}) \in \Psi_\gp(G_n)$. 
Set $A_i = (a_i+b_i)/2-1$ and $B_i = (a_i-b_i)/2$.
Choose an admissible order $\succ_\psi$ on $I_\rho$ for each $\rho$, 
which satisfies $(\PP')$ if $B_i < 0$ for some $i \in I_\rho$. 
Then 
\[
\bigoplus_{\pi \in \Pi_\psi} \pi \cong \bigoplus_{\EE} \pi(\EE), 
\]
where $\EE$ runs over all extended multi-segments 
with $\supp(\EE) = \cup_{\rho}\{ [A_i,B_i]_{\rho} \}_{i \in (I_\rho,\succ_\psi)}$.
\end{thm}
\begin{proof}
When $\psi$ satisfies that
\begin{itemize}
\item
for $i,j \in I_\rho$, if $i \succ_\psi j$, then $B_i > A_j$; 
\item
$B_i \geq 0$ for any $i \in I_\rho$,
\end{itemize}
this is just M{\oe}glin's construction (see \cite[\S 7]{X2}).
In general, the assertion follows from Theorem \ref{compatible} and the definition of $\pi(\EE)$. 
\end{proof}

\subsection{Characters of component groups}
Let $\psi \in \Psi_\gp(G_n)$. 
Recall that Arthur \cite[Theorem 2.2.1]{Ar} established the $A$-packet $\Pi_\psi$
together with a map
\[
\Pi_\psi \rightarrow \widehat{\Sc_\psi}, \; \pi \mapsto \pair{\cdot, \pi}_\psi. 
\]
This map plays an important role for global applications. 
In this subsection, we describe this map. 
\par

The following is a reformulation of \cite[Definition 5.2]{X2}. 
\begin{defi}\label{def.character}
Let $\psi = \oplus_\rho (\oplus_{i \in I_\rho} \rho \boxtimes S_{a_i} \boxtimes S_{b_i}) \in \Psi_\gp(G_n)$. 
With $A_i = (a_i+b_i)/2-1$ and $B_i = (a_i-b_i)/2$, 
we choose an admissible order $\succ$ on $I_\rho$, 
which satisfies $(\PP')$ if $B_i < 0$ for some $i \in I_\rho$. 

\begin{enumerate}
\item
For $i \in I_\rho$, define $Z(\psi)_{\rho,i}$ by the set of $j \in I_\rho$ such that 
$\#[A_i, B_i]_\rho \not\equiv \#[A_j,B_j]_\rho \bmod 2$ and 
\[
\left\{
\begin{aligned}
j \prec i \implies 
\half{A_j+B_j} > \half{A_i+B_i},\; \#[A_j,B_j]_\rho > \#[A_i,B_i]_\rho, \\
j \succ i \implies 
\half{A_j+B_j} < \half{A_i+B_i},\; \#[A_j,B_j]_\rho < \#[A_i,B_i]_\rho. 
\end{aligned}
\right.
\]

\item
For an extended multi-segment $\EE = \cup_{\rho}\{ ([A_i,B_i]_{\rho}, l_i, \eta_i) \}_{i \in (I_\rho,\succ)}$ for $G_n$, 
define $\eta_\EE \in \widehat{\Sc_{\psi}}$ by 
\[
\eta_\EE(\rho \boxtimes S_{a_i} \boxtimes S_{b_i}) = (-1)^{\# Z(\psi)_{\rho, i}+[\half{b_i}]+l_i} \eta_i^{b_i}.
\]
\end{enumerate}
\end{defi}

\begin{thm}\label{character}
If $\pi(\EE) \not= 0$ so that $\pi(\EE) \in \Pi_\psi$, 
we have
\[
\pair{\cdot, \pi(\EE)}_\psi = \eta_\EE.
\]
\end{thm}
\begin{proof}
When $B_i \geq 0$ for all $\rho$ and $i \in I_\rho$, 
the assertion is \cite[Propositions 5.7, 8.2]{X2}. 
Remark that the character of \cite[Definition 8.1]{X2} is trivial in this case.
The general case follows from the compatibility of standard endoscopic character identities and Jacquet modules 
(\cite[\S 6]{X1}, see also the proof of \cite[Theorem 7.5]{X2}). 
\end{proof}

\subsection{Reduction to the non-negative case}
In this subsection, we reduce problems for $\pi(\EE)$ to the non-negative case. 
Let $\EE = \cup_\rho \{([A_i,B_i]_{\rho}, l_i, \eta_i)\}_{i \in (I_\rho,\succ)}$ be an extended multi-segment for $G_n$. 
We assume that the fixed admissible order $\succ$ satisfies that: 
\[
\tag{$\PP'$}
\text{
For $i,j \in I_\rho$, 
if $B_i > B_j$, then $i \succ j$.
}
\] 
Take a non-negative integer $t$ such that $t+B_i \geq 0$ for any $\rho$ and $i \in I_\rho$. 
Define $\EE_t$ from $\EE$ by 
replacing $[A_i,B_i]_{\rho}$ with $[A_i+t,B_i+t]_{\rho}$ for any $\rho$ and $i \in I_\rho$.

\begin{thm}\label{nonzero2}
The representation $\pi(\EE)$ is nonzero if and only if 
$\pi(\EE_t) \not= 0$ and 
the following condition holds for any $\rho$ and $i \in I_\rho$:
\[
\tag{$\star$}
B_i+l_i \geq \left\{
\begin{aligned}
&0 \iif B_i \in \Z,\\
&\half{1} \iif B_i \not\in \Z, \eta_i \not= (-1)^{\alpha_i}, \\
&-\half{1} \iif B_i \not\in \Z, \eta_i = (-1)^{\alpha_i}, 
\end{aligned}
\right. 
\]
where we set 
\[
\alpha_i = \sum_{j \in I_\rho,\; j \prec i}(A_j+B_j+1). 
\]
Moreover, in this case, if $\pi(\EE_t) \cong L(\Delta_{\rho_1}[x_1,-y_1], \dots, \Delta_{\rho_r}[x_r,-y_r]; \pi(\phi, \ep))$ 
with $\phi = \bigoplus_{j=1}^s \rho'_j \boxtimes S_{a_j}$, 
then 
\begin{itemize}
\item
$x_i+y_i+1 \geq 2t$ for any $1 \leq i \leq r$; 
\item
$a_j \geq 2t$ for any $1 \leq j \leq s$,
\end{itemize}
and 
\[
\pi(\EE) \cong L(\Delta_{\rho_1}[x_1-t,-(y_1-t)], \dots, \Delta_{\rho_r}[x_r-t,-(y_r-t)]; \pi(\phi_{-t}, \ep_{-t}))
\]
where $\phi_{-t} = \bigoplus_{j=1}^s \rho'_j \boxtimes S_{a_j-2t}$ 
and $\ep_{-t}(\rho'_j \boxtimes S_{a_j-2t}) = \ep(\rho'_j \boxtimes S_{a_j})$.
\end{thm}
\begin{proof}
We prove the assertion by induction on $\sum_\rho(\#I_\rho -1)$.
Write $I_\rho = \{1, \dots, m\}$ with $1 \prec \dots \prec m$. 
First of all, if $0 \leq B_1 \leq A_1 < B_2 \leq A_2 < \dots < B_m \leq A_m$ for any $\rho$, 
then the assertion is obvious from the definition of $\pi(\EE)$. 
\par

Now we consider the general case. 
Take a positive integer $t'$, 
and define $\EE'$ (\resp $\EE'_t$) from $\EE$ (\resp $\EE_t$) 
by replacing $[A_m,B_m]_\rho$ with $[A_m+t', B_m+t']_\rho$ 
(\resp $[A_m+t,B_m+t]_\rho$ with $[A_m+t+t', B_m+t+t']_\rho$). 
When $t' \gg 0$, 
the calculation of the definition of $\pi(\EE')$
is the same if we replace $[A_m+t', B_m+t']_\rho$ with $[A_m+t', B_m+t']_{\rho'}$ for some $\rho'$, 
i.e., 
we can replace $I_\rho$ with the union of $I_\rho = \{1, \dots, m-1\}$ and $I_{\rho'} = \{m\}$. 
Therefore, we may apply the induction hypothesis to $\EE'$ and $\EE'_t$. 
Moreover, we have
\begin{align}
\label{E}
\pi(\EE) &= 
D_{\rho|\cdot|^{B_m+1}, \dots, \rho|\cdot|^{A_m+1}}
\circ \dots \circ
D_{\rho|\cdot|^{B_m+t'}, \dots, \rho|\cdot|^{A_m+t'}}(\pi(\EE')), \\
\label{E'}
\pi(\EE_t) &= 
D_{\rho|\cdot|^{B_m+t+1}, \dots, \rho|\cdot|^{A_m+t+1}}
\circ \dots \circ
D_{\rho|\cdot|^{B_m+t+t'}, \dots, \rho|\cdot|^{A_m+t+t'}}(\pi(\EE'_t)).
\end{align}
\par

We show that if $\pi(\EE) \not= 0$, then the condition $(\star)$ holds for $i = m$. 
To do this, we may assume that $B_m < 0$. 
We note that $B_1 \leq \dots \leq B_m$ by the condition $(\PP')$.
First, we prove that $\pi(\EE) \not= 0 \implies B_m+l_m \geq -1/2$. 
Suppose that $\pi(\EE) \not= 0$ but $B_m+l_m < -1/2$. 
When $t'$ is big enough, by definition, we have
\[
\pi(\EE') \hookrightarrow 
\begin{pmatrix}
B_m+t' & \ldots & B_m+t'+l_m-1 \\
\vdots & \ddots & \vdots \\
-(A_m+t') & \ldots & -(A_m+t'-l_m+1)
\end{pmatrix}_\rho
\rtimes \pi(\EE''), 
\]
where $\EE''$ is defined from $\EE'$ by replacing 
$([A_m+t',B_m+t']_\rho, l_m, \eta_m)$ by $([A_m+t'-l_m,B_m+t'+l_m]_\rho, 0, \eta_m)$.
Hence if $\pi(\EE) \not= 0$, then (using \cite[Lemma 5.6]{X1}), we see that
\begin{align}
\label{D}
D_{\rho|\cdot|^{B_m+1+l_m}} \circ D_{\rho|\cdot|^{B_m+2+l_m}} \circ \dots \circ D_{\rho|\cdot|^{B_m+t'+l_m}}(\pi(\EE''))
\not= 0.
\end{align}
Note that $\#[A_m+t'-l_m,B_m+t'+l_m]_\rho > 0$ since $A_m+B_m \geq 0$. 
We may redefine $\EE''$ by splitting 
$([A_m+t'-l_m,B_m+t'+l_m]_\rho, 0, \eta_m)$ into 
\[
\{ ([B_m+t'+l_m,B_m+t'+l_m]_\rho, 0, \eta_m), ([A_m+t'-l_m,B_m+t'+l_m+1]_\rho, 0, -\eta_m) \}.
\] 
Using the twisted endoscopy, we transfer \eqref{D} to a general linear group. 
By the compatibility of twisted endoscopic character identities and Jacquet modules \cite[\S 6]{X1}
(cf., see Proposition \ref{k=<m} and Theorem \ref{compatible}), 
the Steinberg representation $\Delta_\rho[B_m+t'+l_m, -(B_m+t'+l_m)]$ should be embedded into 
\begin{align*}
\rho|\cdot|^{B_m+t'+l_m} \times \dots \times \rho|\cdot|^{B_m+1+l_m}
\times \tau_0 \times 
\rho|\cdot|^{-(B_m+1+l_m)} \times \dots \times \rho|\cdot|^{-(B_m+t'+l_m)}
\end{align*}
for some nonzero representation $\tau_0$.
However, in this case, $\tau_0$ is a representation of $\GL_{d(2(B_m+l_m)+1)}(F)$
(where $\rho \in \Cusp^\bot(\GL_d(F))$). 
Therefore, we must have $2(B_m+l_m)+1 \geq 0$, which is a contradiction.
\par

When $B_m +l_m = -1/2$ and $\eta_m \not= (-1)^{\alpha_i}$, 
one can see that 
\[
D_{\rho|\cdot|^{1/2}} \circ D_{\rho|\cdot|^{3/2}} \circ \dots \circ
D_{\rho|\cdot|^{B_m+t'+l_m}}(\pi(\EE'')) = 0. 
\]
This comes from the special (formal) understanding of \cite[Theorem 5.3]{AM} for $x=1/2$
(see also \cite[Theorem 3.3]{J-dual}). 
Or, it is also understood by using Theorem \ref{change} below. 
In particular, if $\pi(\EE) \not= 0$, then the condition $(\star)$ must hold for $i = m$. 
\par

Under the condition $(\star)$ on $B_m+l_m$, 
we will show that 
the computations of the derivatives in the equations \eqref{E} and \eqref{E'} are exactly the same 
(up to the shift by $t$). 
Let $D_{\rho|\cdot|^{z_s}} \circ \dots \circ D_{\rho|\cdot|^{z_1}}$ and 
$D_{\rho|\cdot|^{z_s+t}} \circ \dots \circ D_{\rho|\cdot|^{z_1+t}}$ 
denote the compositions of derivatives appearing in \eqref{E} and \eqref{E'}, 
respectively. 
Fix $0 \leq s' \leq s$, and set 
\[
\pi = D_{\rho|\cdot|^{z_{s'}}} \circ \dots \circ D_{\rho|\cdot|^{z_{1}}}(\pi(\EE')), 
\quad
\pi_t = D_{\rho|\cdot|^{z_{s'}+t}} \circ \dots \circ D_{\rho|\cdot|^{z_{1}+t}}(\pi(\EE_t')). 
\]
By induction on $s'$, we claim that if 
$\pi_t = L(\Delta_{\rho_1}[x_1,-y_1], \dots, \Delta_{\rho_r}[x_r,-y_r]; \pi(\phi, \ep))$, 
then 
$\pi = L(\Delta_{\rho_1}[x_1-t,-(y_1-t)], \dots, \Delta_{\rho_r}[x_r-t,-(y_r-t)]; \pi(\phi_{-t}, \ep_{-t}))$. 
Namely, we compare $D_{\rho|\cdot|^x}(\pi)$ and $D_{\rho|\cdot|^{x+t}}(\pi_t)$.  
\par

When $x \leq 0$, by the condition $(\star)$ on $B_m+l_m$, 
we may assume that $\rho_1|\cdot|^{x_1} \cong \rho|\cdot|^{x+t}$ 
so that $D_{\rho|\cdot|^x}(\pi)$ and $D_{\rho|\cdot|^{x+t}}(\pi_t)$ are both nonzero. 
However, by Proposition \ref{k=<m}, they are both irreducible (up to multiplicities). 
Moreover, they are given by replacing $\Delta_{\rho_1}[x_1,-y_1]$ and $\Delta_{\rho_1}[x_1-t,-(y_1-t)]$
with $\Delta_{\rho_1}[x_1-1,-y_1]$ and $\Delta_{\rho_1}[x_1-t-1,-(y_1-t)]$, respectively. 
It shows our claim. 
\par

Hence we may assume that $x > 0$. 
In this case, to compute $D_{\rho|\cdot|^{x+t}}(\pi_t)$, 
according to \cite[Theorem 7.1]{AM}, we consider 
\begin{align*}
A_{\rho|\cdot|^{x+t}} &= \{i \in \{1, \dots, r\}\;|\; \rho_i \cong \rho,\; x_i = x+t\}, \\
A_{\rho|\cdot|^{x+t-1}} &= \{i \in \{1, \dots, r\}\;|\; \rho_i \cong \rho,\; x_i = x+t-1,\; y_i \not= x+t\}, \\
B_{\rho|\cdot|^{x+t}} &= \{i \in \{1, \dots, r\}\;|\; \rho_i \cong \rho,\; y_i = x+t, \; x_i \not= x+t-1\}, \\
B_{\rho|\cdot|^{x+t-1}} &= \{i \in \{1, \dots, r\}\;|\; \rho_i \cong \rho,\; y_i = x+t-1\}
\end{align*}
and relevant matching functions 
$f \colon A_{\rho|\cdot|^{x+t-1}}^0 \rightarrow A_{\rho|\cdot|^{x+t}}^0$
and 
$g \colon B_{\rho|\cdot|^{x+t-1}}^0 \rightarrow B_{\rho|\cdot|^{x+t}}^0$. 
See \cite[\S 6.1]{AM} for the definition of matching functions. 
When we compute $D_{\rho|\cdot|^x}(\pi)$, 
we need to consider similar (totally ordered) sets 
$A_{\rho|\cdot|^{x}}, A_{\rho|\cdot|^{x-1}}, B_{\rho|\cdot|^{x}}, B_{\rho|\cdot|^{x-1}}$ 
and matching functions. 
However, since $\Delta_\rho[-x-1,-x] = \Delta_\rho[-x,-(x-1)] = \1_{\GL_0(F)}$, 
we have to remove them in the definitions of $B_{\rho|\cdot|^x}$ and $B_{\rho|\cdot|^{x-1}}$. 
It follows from the definition of $\pi(\EE_t)$ that 
the multiplicity of $\Delta_\rho[-x+t-1, -(x+t)]$ 
in the multi-set $\{\Delta_{\rho_1}[x_1, -y_1], \dots, \Delta_{\rho_r}[x_r, -y_r]\}$
is greater than or equal to that of $\Delta_\rho[-x+t,-(x+t-1)]$.
If $\Delta_\rho[-x+t-1, -(x+t)]$ is in $\{\Delta_{\rho_1}[x_1, -y_1], \dots, \Delta_{\rho_r}[x_r, -y_r]\}$, 
the corresponding index is the maximal element in $B_{\rho|\cdot|^{x+t}}$. 
Moreover, by the definition of the best matching function (see \cite[\S 6.1, 7.1]{AM}), 
it is the image of an index in $B_{\rho|\cdot|^{x+t-1}}$ corresponding to $\Delta_\rho[-x+t,-(x+t-1)]$ via $f$. 
Therefore, the complement $B_{\rho|\cdot|^x} \setminus B_{\rho|\cdot|^x}^0$ is equal to 
$B_{\rho|\cdot|^{x+t}} \setminus B_{\rho|\cdot|^{x+t}}^0$. 
Finally, if $x = 1/2$, then by the condition $(\star)$, 
the multiplicity of $\rho \boxtimes S_{2(x+t)-1}$ in $\phi$ is at most one. 
Therefore, by \cite[Theorem 7.1]{AM}, 
the computations of $D_{\rho|\cdot|^x}(\pi)$ and $D_{\rho|\cdot|^{x+t}}(\pi_t)$ 
are exactly the same (up to the shift by $t$). 
\par

Finally, by the proof of the claim, 
we see that $(x_i-t)+(y_i-t)+1 \geq 0$ and $a_j-2t \geq 0$. 
This completes the proof of Theorem \ref{nonzero2}. 
\end{proof}

In \S \ref{s.nonzero} and \S \ref{s.deform} below, 
we will consider $\pi(\EE)$ for $\EE = \cup_\rho \{([A_i,B_i]_{\rho}, l_i, \eta_i)\}_{i \in (I_\rho,\succ)}$
with $B_i \geq 0$ for all $\rho$ and $i \in I_\rho$. 

\subsection{Examples of $A$-packets}\label{ex.sc}
In this subsection, we set $\rho = \1_{\GL_1(F)}$ and we drop $\rho$ from the notation. 
When $\phi = \rho \boxtimes (S_{2x_1+1} \oplus \dots \oplus S_{2x_r+1})$ 
with $\ep(\rho \boxtimes S_{2x_i+1}) = \epsilon_i$, 
we write $\pi(\phi,\ep) = \pi(x_1^{\epsilon_1}, \dots, x_r^{\epsilon_r})$.
\par

\begin{ex}\label{ex-super}
Let us compute the $A$-packet $\Pi_\psi$ for
\[
\psi = \1 \boxtimes S_6 + \1 \boxtimes S_2 + S_4 \boxtimes \1 \in \Psi_\gp(\SO_{13}(F)). 
\]
It is an \emph{elementary} $A$-parameter (see \cite[\S 6]{X2}). 
By Theorem \ref{nonzero2} together with the sign condition in Definition \ref{segments} (2), we see that
$\Pi_\psi$ has at most 4 irreducible representations and they are associated to 
\begin{align*}
\EE_1 = 
\bordermatrix{
& -\half{5} & -\half{3} & -\half{1} & \half{1} & \half{3} & \half{5}\cr
& \lhd & \lhd & \lhd & \rhd & \rhd & \rhd \cr
&  &  & \lhd & \rhd &  &  \cr
&  &  &  &  & \oplus &  \cr
}, 
&\quad
\EE_2 = 
\bordermatrix{
& -\half{5} & -\half{3} & -\half{1} & \half{1} & \half{3} & \half{5}\cr
& \lhd & \lhd & \lhd & \rhd & \rhd & \rhd \cr
&  &  & \ominus & \oplus &  &  \cr
&  &  &  &  & \ominus &  \cr
}, 
\\
\EE_3 = 
\bordermatrix{
& -\half{5} & -\half{3} & -\half{1} & \half{1} & \half{3} & \half{5}\cr
& \lhd & \lhd & \oplus & \ominus & \rhd & \rhd \cr
&  &  & \lhd & \rhd &  &  \cr
&  &  &  &  & \ominus &  \cr
}, 
&\quad
\EE_4 = 
\bordermatrix{
& -\half{5} & -\half{3} & -\half{1} & \half{1} & \half{3} & \half{5}\cr
& \lhd & \lhd & \oplus & \ominus & \rhd & \rhd \cr
&  &  & \ominus & \oplus &  &  \cr
&  &  &  &  & \oplus &  \cr
}.
\end{align*}
We compute $\pi(\EE_i)$ for $i = 1,2,3,4$.
Taking $(t_1,t_2,t_3) = (0,4,4)$, we consider $\EE'_i$ as in \S \ref{s.segment-rep}. 
Then $\pi(\EE_i) \cong D_3 \circ D_2(\pi(\EE'_i))$ with 
\begin{align*}
D_2 &= 
D_{|\cdot|^{1/2}, |\cdot|^{3/2}} \circ D_{|\cdot|^{3/2}, |\cdot|^{5/2}} \circ 
D_{|\cdot|^{5/2}, |\cdot|^{7/2}} \circ D_{|\cdot|^{7/2}, |\cdot|^{9/2}}, \\
D_3 &= D_{|\cdot|^{5/2}} \circ D_{|\cdot|^{7/2}} \circ D_{|\cdot|^{9/2}} \circ D_{|\cdot|^{11/2}}. 
\end{align*}

\begin{enumerate}
\item
For $i=1$, by \cite[Theorem 7.1]{AM}, we have 
\begin{align*}
\pi(\EE_1) 
&\cong 
D_3 \circ D_2\left(
L(|\cdot|^{-\half{5}}, |\cdot|^{-\half{3}}, |\cdot|^{-\half{1}}, \Delta[7/2,-9/2]; \pi((11/2)^+))
\right)
\\&\cong L(|\cdot|^{-\half{5}}, |\cdot|^{-\half{3}}, |\cdot|^{-\half{1}}, |\cdot|^{-\half{1}}; \pi((3/2)^+)).
\end{align*}

\item
For $i=2$, by \cite[Theorem 7.1]{AM}, we have 
\begin{align*}
\pi(\EE_2) 
&\cong 
D_3 \circ D_2\left(
L(|\cdot|^{-\half{5}}, |\cdot|^{-\half{3}}, |\cdot|^{-\half{1}}; \pi((7/2)^-, (9/2)^+, (11/2)^-))
\right)
\\&\cong 
D_3 \circ D_{|\cdot|^{3/2}} \circ D_{|\cdot|^{1/2}}
\left(
L(|\cdot|^{-\half{5}}, |\cdot|^{-\half{3}}, |\cdot|^{-\half{1}}; \pi((1/2)^-, (3/2)^+, (11/2)^-))
\right)
\\&\cong 
D_3 \circ D_{|\cdot|^{3/2}}
\left(
L(|\cdot|^{-\half{5}}, |\cdot|^{-\half{3}}; \pi((1/2)^-, (3/2)^+, (11/2)^-))
\right)
\\&\cong 
D_3
\left(
L(|\cdot|^{-\half{5}}; \pi((1/2)^-, (3/2)^+, (11/2)^-))
\right)
\\&\cong 
D_{|\cdot|^{5/2}}
\left(
L(|\cdot|^{-\half{5}}; \pi((1/2)^-, (3/2)^+, (5/2)^-))
\right)
\\&\cong 
\pi((1/2)^-, (3/2)^+, (5/2)^-).
\end{align*}
In particular, $\pi(\EE_2)$ is a supercuspidal representation by \cite[Theorem 3.3]{X1}. 

\item
For $i=3$, by \cite[Theorem 7.1]{AM}, we have 
\begin{align*}
\pi(\EE_3) 
&\cong 
D_3 \circ D_2\left(
L(|\cdot|^{-\half{5}}, |\cdot|^{-\half{3}}, \Delta[7/2,-9/2]; \pi((1/2)^-, (11/2)^-))
\right)
\\&\cong 
D_3 \circ D_{|\cdot|^{3/2}}
\left(
L(|\cdot|^{-\half{5}}, |\cdot|^{-\half{3}}, \Delta[-1/2,-3/2]; \pi((1/2)^-, (11/2)^-))
\right)
\\&\cong 
D_3
\left(
L(|\cdot|^{-\half{5}}, \Delta[-1/2,-3/2]; \pi((1/2)^-, (11/2)^-))
\right)
\\&\cong 
L(|\cdot|^{-\half{5}}, \Delta[-1/2,-3/2]; \pi((1/2)^-, (3/2)^-)). 
\end{align*}

\item
For $i=4$, by \cite[Theorem 7.1]{AM}, we have 
\begin{align*}
\pi(\EE_4) 
&\cong 
D_3 \circ D_2\left(
L(|\cdot|^{-\half{5}}, |\cdot|^{-\half{3}}; \pi((1/2)^-, (7/2)^-, (9/2)^+, (11/2)^+))
\right)
\\&\cong 
D_3 \circ D_{|\cdot|^{3/2}} \circ D_{|\cdot|^{1/2}}
\left(
L(|\cdot|^{-\half{5}}, |\cdot|^{-\half{3}}; \pi((1/2)^-, (1/2)^-, (3/2)^+, (11/2)^+))
\right)
\\&\cong 
D_3 \circ D_{|\cdot|^{3/2}} 
\left(
L(|\cdot|^{-\half{5}}, |\cdot|^{-\half{3}}, |\cdot|^{-\half{1}}; \pi((3/2)^+, (11/2)^+))
\right)
\\&\cong 
D_3 
\left(
L(|\cdot|^{-\half{5}}, |\cdot|^{-\half{3}}, |\cdot|^{-\half{1}}; \pi((1/2)^+, (11/2)^+))
\right)
\\&\cong 
L(|\cdot|^{-\half{5}}, |\cdot|^{-\half{3}}, |\cdot|^{-\half{1}}; \pi((1/2)^+, (3/2)^+)).
\end{align*}
\end{enumerate}
Therefore, $\Pi_\psi$ consists of 4 irreducible representations
\begin{align*}
\pi(\EE_1) &= L(|\cdot|^{-\half{5}}, |\cdot|^{-\half{3}}, |\cdot|^{-\half{1}}, |\cdot|^{-\half{1}}; \pi((3/2)^+)),\\
\pi(\EE_2) &= \pi((1/2)^-, (3/2)^+, (5/2)^-), \\
\pi(\EE_3) &= L(|\cdot|^{-\half{5}}, \Delta[-1/2, -3/2]; \pi((1/2)^-, (3/2)^-)), \\
\pi(\EE_4) &= L(|\cdot|^{-\half{5}}, |\cdot|^{-\half{3}}, |\cdot|^{-\half{1}}; \pi((1/2)^+, (3/2)^+)). 
\end{align*}
Remark that: 
\begin{itemize}
\item
By Arthur's general result \cite[Proposition 7.4.1]{Ar}, we already know that $\pi(\EE_1) \in \Pi_\psi$.
\item
By M{\oe}glin's original construction, one can easily see that $|\Pi_\psi| = 4$ and $\pi(\EE_2) \in \Pi_\psi$
(see \cite[\S 6]{X2}). 
One might also conclude that $\pi(\EE_3), \pi(\EE_4) \in \Pi_\psi$, 
but it would be much harder. 
\end{itemize}
\end{ex}

\begin{ex}
Let us consider $\psi = S_4 \boxtimes S_6 + S_3 \boxtimes S_3 \in \Psi_\gp(\Sp_{32}(F))$. 
By Theorems \ref{nonzero2} and \ref{nonzero} below, 
we can see that $\#\Pi_\psi = 7$. 
The extended multi-segments $\EE$ with $\pi(\EE) \not= 0$ and the characters $\eta_\EE$ are listed as follows. 
Here, $\eta \in \widehat{\Sc_\psi}$ is identified with 
$(\eta(S_4 \boxtimes S_6), \eta(S_3 \boxtimes S_3)) \in \{\pm1\}^2$.
\begin{align*}
\EE_1 &= 
\bordermatrix{
& -1 & 0 & 1 & 2 & 3 & 4 \cr
& \lhd & \lhd & \lhd & \rhd & \rhd & \rhd \cr
&  & \lhd & \oplus & \rhd &  &  \cr
}, 
&
\eta_{\EE_1} &= (-,-), \\
\EE_2 &= 
\bordermatrix{
& -1 & 0 & 1 & 2 & 3 & 4 \cr
& \lhd & \lhd & \lhd & \rhd & \rhd & \rhd \cr
&  & \ominus & \oplus & \ominus &  &  \cr
}, 
&
\eta_{\EE_2} &= (-,-), \\
\EE_3 &= 
\bordermatrix{
& -1 & 0 & 1 & 2 & 3 & 4 \cr
& \lhd & \lhd & \oplus & \ominus & \rhd & \rhd \cr
&  & \lhd & \ominus & \rhd &  &  \cr
}, 
&
\eta_{\EE_3} &= (+,+), \\
\EE_4 &= 
\bordermatrix{
& -1 & 0 & 1 & 2 & 3 & 4 \cr
& \lhd & \lhd & \ominus & \oplus & \rhd & \rhd \cr
&  & \oplus & \ominus & \oplus &  &  \cr
}, 
&
\eta_{\EE_4} &= (+,+), \\
\EE_5 &= 
\bordermatrix{
& -1 & 0 & 1 & 2 & 3 & 4 \cr
& \lhd & \lhd & \ominus & \oplus & \rhd & \rhd \cr
&  & \lhd & \ominus & \rhd &  &  \cr
}, 
&
\eta_{\EE_5} &= (+,+), \\
\EE_6 &= 
\bordermatrix{
& -1 & 0 & 1 & 2 & 3 & 4 \cr
& \lhd & \oplus & \ominus & \oplus & \ominus & \rhd \cr
&  & \ominus & \oplus & \ominus &  &  \cr
}, 
&
\eta_{\EE_6} &= (-,-), \\
\EE_7 &= 
\bordermatrix{
& -1 & 0 & 1 & 2 & 3 & 4 \cr
& \lhd & \ominus & \oplus & \ominus & \oplus & \rhd \cr
&  & \lhd & \oplus & \rhd &  &  \cr
}, 
&
\eta_{\EE_7} &= (-,-).
\end{align*}
The associated representations are listed as follows. 
\begin{align*}
\pi(\EE_1) &\cong L(\Delta[-1,-4], \Delta[0,-3], \Delta[0,-2]; \pi(1^-,1^-,2^+)), \\
\pi(\EE_2) &\cong L(\Delta[-1,-2], \Delta[0,-1]; \pi(0^+,1^+,2^-,3^+,4^-)), \\
\pi(\EE_3) &\cong L(\Delta[-1,-4], \Delta[0,-3], \Delta[0,-2], \Delta[1,-2]; \pi(1^+)), \\
\pi(\EE_4) &\cong L(\Delta[-1,-4], \Delta[0,-2]; \pi(0^-,1^+,1^+,2^-,3^+)), \\
\pi(\EE_5) &\cong L(\Delta[-1,-2], \Delta[0,-4], \Delta[0,-1]; \pi(1^-,2^+,3^-)), \\
\pi(\EE_6) &\cong L(\Delta[-1,-4], \Delta[0,-3], \Delta[1,-2]; \pi(0^+,1^-,2^-)), \\
\pi(\EE_7) &\cong L(\Delta[-1,-2], \Delta[0,-4], \Delta[1,-3]; \pi(0^-,1^+,2^-)).
\end{align*}
\end{ex}

\section{A non-vanishing criterion}\label{s.nonzero}
We fix an extended multi-segment $\EE = \cup_\rho \{([A_i,B_i]_{\rho}, l_i, \eta_i)\}_{i \in (I_\rho,\succ)}$ for $G_n$
such that $B_i \geq 0$ for all $\rho$ and $i \in I_\rho$. 
In this section, we discuss about the non-vanishing of $\pi(\EE)$. 

\subsection{Necessary conditions}
In \cite{X3}, Xu established an algorithm to determine whether $\pi(\EE) \not= 0$. 
To do this, he gave three necessary conditions for $\pi(\EE) \not= 0$. 
Recall that $I_\rho$ is a totally ordered set with a fixed admissible order $\succ$. 
For $[A_i,B_i]_\rho$, we set $b_i = \#[A_i,B_i] = A_i-B_i+1$.

\begin{prop}[{\cite[Lemmas 5.5, 5.6, 5.7]{X3}}]\label{nec}
Let $\EE = \cup_\rho \{([A_i,B_i]_{\rho}, l_i, \eta_i)\}_{i \in (I_\rho,\succ)}$ 
be an extended multi-segment for $G_n$ such that $B_i \geq 0$ for all $\rho$ and $i \in I_\rho$. 
Let $k \succ k-1$ be two adjacent elements in $I_\rho$.
Suppose that $\pi(\EE) \not= 0$. 

\begin{enumerate}
\item
If $A_k \geq A_{k-1}$ and $B_k \geq B_{k-1}$, then 
\[
\left\{
\begin{aligned}
\eta_{k} = (-1)^{A_{k-1}-B_{k-1}}\eta_{k-1} 
&\implies A_k - l_k \geq A_{k-1} - l_{k-1}, \quad B_k + l_k \geq B_{k-1} + l_{k-1}, \\
\eta_{k} \not= (-1)^{A_{k-1}-B_{k-1}}\eta_{k-1} 
&\implies B_k + l_k > A_{k-1} - l_{k-1}.
\end{aligned}
\right.
\]
In particular, if $[A_k,B_k]_\rho = [A_{k-1},B_{k-1}]_\rho$, 
then $\eta_{k} = (-1)^{A_{k-1}-B_{k-1}}\eta_{k-1}$ and $l_k = l_{k-1}$. 

\item
If $[A_{k-1},B_{k-1}]_\rho \subset [A_k,B_k]_\rho$, 
then 
\[
\left\{
\begin{aligned}
\eta_{k} = (-1)^{A_{k-1}-B_{k-1}}\eta_{k-1} 
&\implies 0 \leq l_{k} -l_{k-1} \leq b_{k}-b_{k-1}, \\
\eta_{k} \not= (-1)^{A_{k-1}-B_{k-1}}\eta_{k-1} 
&\implies l_k + l_{k-1} \geq b_{k-1}.
\end{aligned}
\right.
\]

\item
If $[A_{k-1},B_{k-1}]_\rho \supset [A_k,B_k]_\rho$, 
then 
\[
\left\{
\begin{aligned}
\eta_{k} = (-1)^{A_{k-1}-B_{k-1}}\eta_{k-1} 
&\implies 0 \leq l_{k-1} -l_k \leq b_{k-1}-b_k, \\
\eta_{k} \not= (-1)^{A_{k-1}-B_{k-1}}\eta_{k-1} 
&\implies l_k + l_{k-1} \geq b_k.
\end{aligned}
\right.
\]
\end{enumerate}
\end{prop}

In a particular case, the condition in Proposition \ref{nec} (1) is sufficient. 
\begin{prop}[{\cite[Theorem A.3]{X4}}]\label{ladder}
Let $\EE = \cup_\rho \{([A_i,B_i]_{\rho}, l_i, \eta_i)\}_{i \in (I_\rho,\succ)}$ 
be an extended multi-segment for $G_n$ such that $B_i \geq 0$ for all $\rho$ and $i \in I_\rho$.
Suppose that for any $\rho$ and any two adjacent elements $k \succ k-1$ of $I_\rho$, 
we have $A_k \geq A_{k-1}$ and $B_k \geq B_{k-1}$. 
Then $\pi(\EE) \not= 0$ if and only if 
the condition in Proposition \ref{nec} (1) holds for all adjacent elements $k \succ k-1$.
\end{prop}

\subsection{Changing of admissible orders}\label{sec.change}
Now suppose that $\EE = \cup_\rho \{([A_i,B_i]_{\rho}, l_i, \eta_i)\}_{i \in (I_\rho,\succ)}$ satisfies the three necessary conditions in Proposition \ref{nec}.
In general, there are many choices of admissible orders on $I_\rho$, and there is no canonical choice. 
We recall the behavior of $(l_i, \eta_i)_{i \in I_\rho}$ under changing admissible orders. 
\par

For a positive integer $b$, let $(\Z/b\Z)/\{\pm1\}$ be the quotient of $\Z/b\Z$ by the multiplication by $\pm1$. 
Then we can identify $\{l \in \Z \;|\; 0 \leq l \leq \half{b}\}$ with $(\Z/b\Z)/\{\pm1\}$.
In particular, we regard $l_i$ as an element in $(\Z/b_i\Z)/\{\pm1\}$, where $b_i = \#[A_i,B_i] = A_i-B_i+1$. 
\par

Let $\succ'$ be the order on $I_\rho$ which is given from $\succ$ by changing $k-1 \succ' k$. 
Suppose that $\succ'$ is also admissible. 
This is equivalent to 
$[A_{k-1},B_{k-1}]_\rho \subset [A_k,B_k]_\rho$ or $[A_{k-1},B_{k-1}]_\rho \supset [A_k,B_k]_\rho$.
Following \cite[Section 6.1]{X3}, 
we will define $(l_i', \eta_i')$ for $i \in I_\rho$.
\par

When $[A_{k-1},B_{k-1}]_\rho \subset [A_k,B_k]_\rho$, 
we set $l_i'$ and $\eta_i'$ for $i \in I_\rho$ as follows: 
\begin{itemize}
\item
If $i \not= k,k-1$, then $l_i' = l_i$ and $\eta_i' = \eta_i$; 

\item
$l_{k-1}' = l_{k-1}$ and $\eta_{k-1}' = (-1)^{A_k-B_k}\eta_{k-1}$; 

\item
$l_k' = l_k + \epsilon (b_{k-1}-2l_{k-1})$ in $(\Z/b_k\Z)/\{\pm1\}$, 
where $\epsilon = (-1)^{A_{k-1}-B_{k-1}}\eta_{k-1}\eta_k \in \{\pm1\}$; 

\item
if $\eta_k = (-1)^{A_{k-1}-B_{k-1}}\eta_{k-1}$ and $b_k-2l_k < 2(b_{k-1}-2l_{k-1})$, 
then $\eta_k' = (-1)^{A_{k-1}-B_{k-1}}\eta_k$; 
and otherwise $\eta_k' = (-1)^{A_{k-1}-B_{k-1}-1}\eta_k$. 
\end{itemize}
\par

Similarly, when $[A_{k-1},B_{k-1}]_\rho \supset [A_k,B_k]_\rho$, 
we set $l_i'$ and $\eta_i'$ for $i \in I_\rho$ as follows: 
\begin{itemize}
\item
If $i \not= k,k-1$, then $l_i' = l_i$ and $\eta_i' = \eta_i$; 

\item
$l_{k}' = l_{k}$ and $\eta_k' = (-1)^{A_{k-1}-B_{k-1}}\eta_k$; 

\item
$l_{k-1}' = l_{k-1} + \epsilon (b_{k}-2l_{k})$ in $(\Z/b_{k-1}\Z)/\{\pm1\}$, 
where $\epsilon = (-1)^{A_{k-1}-B_{k-1}}\eta_{k-1}\eta_k \in \{\pm1\}$; 

\item
if $\eta_k = (-1)^{A_{k-1}-B_{k-1}}\eta_{k-1}$ and $b_{k-1}-2l_{k-1} < 2(b_k-2l_k)$, 
then $\eta_{k-1}' = (-1)^{A_k-B_k}\eta_{k-1}$; 
and otherwise $\eta_{k-1}' = (-1)^{A_k-B_k-1}\eta_{k-1}$. 
\end{itemize}

Set $\EE' = \cup_\rho\{([A_i,B_i]_\rho, l_i', \eta_i')\}_{i \in (I_\rho, \succ')}$. 
One can check that the necessary conditions in Proposition \ref{nec} hold 
for $\EE'$ with respect to $k-1 \succ' k$.

\begin{thm}[{\cite[Theorem 6.1]{X3}}]\label{change}
Suppose that 
$[A_{k-1},B_{k-1}]_\rho \subset [A_k,B_k]_\rho$ or $[A_{k-1},B_{k-1}]_\rho \supset [A_k,B_k]_\rho$. 
With the above notation, we have $\pi(\EE) \cong \pi(\EE')$.
\end{thm}

We describe the relation $\EE_\rho \leftrightarrow \EE'_\rho$ 
in the case where $[A_{k-1},B_{k-1}]_\rho \supset [A_k,B_k]_\rho$. 
\begin{enumerate}
\item[(1)]
If $\eta_k = (-1)^{A_{k-1}-B_{k-1}}\eta_{k-1}$ and $b_{k-1}-2l_{k-1} \geq 2(b_k-2l_k)$, 
with $\alpha = b_k-2l_k$, we have
\begin{align*}
&\left(
\begin{array}{rrclcll}
\overset{B_{k-1}}{\lhd} & \cdots \cdots \lhd 
&\overbrace{\odot \cdots \odot}^\alpha & \odot \cdots \odot & \overbrace{\odot \cdots \odot}^\alpha &
\rhd \cdots \cdots & \overset{A_{k-1}}{\rhd}\cr
&\underset{B_{k}}{\lhd} \cdots \lhd &\underbrace{\odot \cdots \odot}_\alpha& 
\rhd \cdots \underset{A_{k}}{\rhd}
\end{array}
\right)_\rho 
\\&\leftrightarrow
\left(
\begin{array}{rrrllll}
&\overset{B_{k}}{\lhd} \cdots \lhd &\overbrace{\odot \cdots \odot}^\alpha& 
\rhd \cdots \overset{A_{k}}{\rhd}\cr
\underset{B_{k-1}}{\lhd} & \cdots \cdots \lhd 
&\underbrace{\lhd \cdots \lhd}_\alpha & \odot \cdots \odot &
\underbrace{\rhd \cdots \rhd}_\alpha &
\rhd  \cdots \cdots & \underset{A_{k-1}}{\rhd}
\end{array}
\right)_\rho. 
\end{align*}

\item[(2a)]
If $\eta_k = (-1)^{A_{k-1}-B_{k-1}}\eta_{k-1}$ and $b_{k-1}-2l_{k-1} < 2(b_k-2l_k)$, 
and if $\alpha = (b_{k-1}-2l_{k-1})-(b_k-2l_k) \geq 0$, then noting that $b_{k-1}-2l_{k-1} > 2\alpha$, we have
\begin{align*}
&\left(
\begin{array}{rrlrlll}
\overset{B_{k-1}}{\lhd} & \cdots \cdots \lhd 
&\overbrace{\odot \cdots \odot}^\alpha & \odot \cdots \odot & \overbrace{\odot \cdots \odot}^\alpha 
&\rhd \cdots \cdots & \overset{A_{k-1}}{\rhd}\cr
&\underset{B_{k}}{\lhd} \cdots \lhd & \odot \cdots \cdots &\cdots \cdots \odot 
&\rhd \cdots \underset{A_{k}}{\rhd}
\end{array}
\right)_\rho
\\&\leftrightarrow
\left(
\begin{array}{rrlrlll}
&\overset{B_{k}}{\lhd} \cdots \lhd & \odot \cdots \cdots &\cdots \cdots \odot 
&\rhd \cdots \overset{A_{k}}{\rhd}\cr
\underset{B_{k-1}}{\lhd} & \cdots \cdots \lhd 
&\underbrace{\lhd \cdots \lhd}_\alpha & \odot \cdots \odot & \underbrace{\rhd \cdots \rhd}_\alpha 
&\rhd \cdots \cdots & \underset{A_{k-1}}{\rhd}\cr
\end{array}
\right)_\rho. 
\end{align*}

\item[(2b)]
If $\eta_k = (-1)^{A_{k-1}-B_{k-1}}\eta_{k-1}$ and $b_{k-1}-2l_{k-1} < 2(b_k-2l_k)$, 
and if $\alpha = (b_k-2l_k)-(b_{k-1}-2l_{k-1}) \geq 0$, 
then noting that $\alpha \leq l_{k-1}$ by Proposition \ref{nec}, 
we have
\begin{align*}
&\left(
\begin{array}{rrrlrll}
\overset{B_{k-1}}{\lhd} &\cdots \cdots \lhd 
&\overbrace{\lhd \cdots \lhd}^\alpha &\odot \cdots \odot&\overbrace{\rhd \cdots \rhd}^\alpha
& \rhd \cdots \cdots &\overset{A_{k-1}}{\rhd}\cr
&&\underset{B_{k}}{\lhd} \cdots \lhd & \odot \cdots \cdots & \cdots \cdot \odot& \rhd \cdots \underset{A_{k}}{\rhd}
\end{array}
\right)_\rho
\\&\leftrightarrow
\left(
\begin{array}{rrrlrll}
&&\overset{B_{k}}{\lhd} \cdots \lhd & \odot \cdots \cdots & \cdots \cdot \odot &\rhd \cdots \overset{A_{k}}{\rhd}\cr
\underset{B_{k-1}}{\lhd} & \cdots \cdots \lhd 
&\underbrace{\odot \cdots \odot}_\alpha & \odot \cdots \odot &\underbrace{\odot \cdots \odot}_\alpha 
&\rhd \cdots \cdots & \underset{A_{k-1}}{\rhd}
\end{array}
\begin{aligned}
\end{aligned}
\right)_\rho. 
\end{align*}

\item[(3)]
If $\eta_k \not= (-1)^{A_{k-1}-B_{k-1}}\eta_{k-1}$, 
then noting that $\alpha = b_{k}-2l_k \leq l_{k-1}$ by Proposition \ref{nec},
we have
\begin{align*}
&\left(
\begin{array}{rrrrlll}
\overset{B_{k-1}}{\lhd}& \cdots \lhd & \overbrace{\lhd \cdots \lhd}^\alpha 
&\odot \cdots \odot& \overbrace{\rhd \cdots \rhd}^\alpha & \rhd\cdots &\overset{A_{k-1}}{\rhd}\cr
&&&\underset{B_{k}}{\lhd} \cdots \lhd & \underbrace{\odot \cdots \odot}_\alpha & \rhd \cdots \underset{A_{k}}{\rhd}
\end{array}
\right)_\rho
\\&\leftrightarrow
\left(
\begin{array}{rrrrlll}
&&&\overset{B_{k}}{\lhd} \cdots \lhd & \overbrace{\odot \cdots \odot}^\alpha & \rhd \cdots \overset{A_{k}}{\rhd}\cr
\underset{B_{k-1}}{\lhd}& \cdots \lhd & \underbrace{\odot \cdots \odot}_\alpha 
&\odot \cdots \odot& \underbrace{\odot \cdots \odot}_\alpha & \rhd\cdots &\underset{A_{k-1}}{\rhd}
\end{array}
\right)_\rho.
\end{align*}
\end{enumerate}

\subsection{Non-vanishing criterion}
Let $\EE = \cup_\rho \{([A_i,B_i]_{\rho}, l_i, \eta_i)\}_{i \in (I_\rho,\succ)}$ 
be an extended multi-segment for $G_n$ such that $B_i \geq 0$ for all $\rho$ and $i \in I_\rho$. 
Define $I_\rho^{2,\adj}$ to be the set of triples $(i,j,\succ')$, 
where $\succ'$ is an admissible order on $I_\rho$, 
and $i \succ' j$ are two adjacent elements in $I_\rho$ with respect to $\succ'$. 
Now we can reformulate Xu's algorithm \cite[\S 8]{X3} for $\pi(\EE) \not= 0$ as follows. 

\begin{thm}\label{nonzero}
Let $\EE = \cup_\rho \{([A_i,B_i]_{\rho}, l_i, \eta_i)\}_{i \in (I_\rho,\succ)}$ 
be an extended multi-segment for $G_n$ such that $B_i \geq 0$ for all $\rho$ and $i \in I_\rho$.
Then the representation $\pi(\EE)$ is nonzero 
if and only if for every $(i,j,\succ') \in I_\rho^{2,\adj}$, 
the three necessary conditions in Proposition \ref{nec} are satisfied for $\EE'$ with respect to $i \succ' j$, 
where $\EE' = \cup_\rho \{([A_i,B_i]_{\rho}, l'_i, \eta'_i)\}_{i \in (I_\rho,\succ')}$ is such that $\pi(\EE) \cong \pi(\EE')$. 
\end{thm}
\begin{proof}
The only if part is Proposition \ref{nec}.
We will prove the if part by induction on $\sum_\rho (\#I_\rho-1)$.
Suppose that the three necessary conditions in Proposition \ref{nec} are satisfied
with respect to every $(i,j,\succ') \in I_\rho^{2,\adj}$.
The case where $|I_\rho| \leq 2$ for any $\rho$ follows from Propositions \ref{nec} and \ref{ladder}
using ``Pull'' (\cite[Proposition 7.1]{X3}) if necessary.
Hence we may assume that $|I_\rho| \geq 3$ for a fixed $\rho$. 
To prove $\pi(\EE) \not= 0$, 
we apply Xu's algorithm \cite[\S 8]{X3}. 
\par

First, we assume that
an element $m \in I_\rho$ is maximal  for every admissible order $\succ'$. 
We may assume that the original admissible order $\succ$ satisfies that 
for $i,j \in I_\rho \setminus \{m\}$, 
\[
i \succ j \implies 
\text{$B_i > B_j$, or $B_i = B_j$ and $A_i \geq A_j$}.
\]
Write $I_\rho = \{1, \dots, m\}$ such that $1 \prec \dots \prec m$. 
In this situation, we can use ``Expand'' (\cite[Proposition 7.4]{X3}).
Let $t = B_m-B_{m-1}$. (Note that we assume that $m \geq 3$.)
Define $\EE^*$ from $\EE$ by replacing $([A_m,B_m]_\rho,l_m,\eta_m)$ with $([A_m+t,B_m-t,l_m+t,\eta_m])$.
Then \cite[Proposition 7.4]{X3} says that
\[
\pi(\EE) \not= 0
\iff
\pi(\EE^*) \not= 0.
\]
\par

We claim that $\EE^*$ satisfies the three necessary conditions
in Proposition \ref{nec} with respect to every $(i,j,\succ') \in I_\rho^{2,\adj}$. 
It is non-trivial only when $i = m$ or $j = m$. 
We may assume that $i = m$ and hence $j \not= m$. 
In this case, by changing $\succ$ suitably, 
we may assume that 
$B_{k-1} \leq B_{k}$ and $A_{k-1} \leq A_{k}$ for any $j < k \leq m$. 
Consider another extended multi-segment 
\[
\EE' = \{([A_k,B_k]_\rho, l_k,\eta_k)\}_{k \in \{j \prec j+1 \prec \dots \prec m\}} \cup \EE_0
\]
for $G_{n'}$, where $\EE_0$ is an ``easy'' auxiliary data. 
Then we can apply Proposition \ref{ladder}, 
and hence the conditions for $(m,j,\succ')$ follow from Proposition \ref{nec}. 
\par

Therefore, we may replace $\pi(\EE)$ with $\pi(\EE^*)$.
In other words, we may assume that 
$I_\rho = \{1, \dots, m\}$ with $1 \prec \dots \prec m$ such that the order $\succ'$ on $I_\rho$ given by 
\[
1 \prec' \dots \prec' m-2 \prec' m \prec' m-1
\]
is also admissible.
Moreover, we may assume that $A_m = \max\{A_i \;|\; i \in I_\rho\}$ and 
\[
A_{m-1} = \max\{A_i \;|\; i \in I_\rho \setminus\{m\},\; [A_i,B_i]_\rho \subset [A_m,B_m]_\rho \}.
\]
In this situation, we can use ``Pull'' (\cite[Propositions 7.1, 7.3]{X3}).
For $t \gg 0$, define 
\begin{itemize}
\item
$\EE^{\sharp}$ from $\EE$ by replacing 
$([A_m,B_m]_\rho, l_m, \eta_m)$ (\resp $([A_{m-1},B_{m-1}]_\rho, l_{m-1}, \eta_{m-1})$)
with $([A_m+t,B_m+t]_\rho, l_m, \eta_m)$ (\resp $([A_{m-1}+t,B_{m-1}+t]_\rho, l_{m-1}, \eta_{m-1})$); 
\item
$\EE^{\flat}$ from $\EE$ by replacing 
$([A_m,B_m]_\rho, l_m, \eta_m)$ with $([A_m+t,B_m+t]_\rho, l_m, \eta_m)$; 
\item
$\EE^{\natural}$ from $\EE'$ by replacing 
$([A_{m-1},B_{m-1}]_\rho, l'_{m-1}, \eta'_{m-1})$ with $([A_{m-1}+t,B_{m-1}+t]_\rho, l'_{m-1}, \eta'_{m-1})$, 
where $\EE' = \cup_\rho \{([A_i,B_i]_{\rho}, l_i', \eta_i')\}_{i \in (I_\rho,\succ')}$ is such that $\pi(\EE) \cong \pi(\EE')$. 
\end{itemize}
Then \cite[Propositions 7.1, 7.3]{X3} say that $\pi(\EE) \not= 0$ if and only if 
the three representations 
\[
\pi(\EE^{\sharp}), \quad
\pi(\EE^{\flat}), \quad
\pi(\EE^{\natural})
\]
are all nonzero. 
However, 
when $t \gg 0$, 
the calculation of the definition of $\pi(\EE^{\sharp})$
is the same if we replace 
$\{([A_{m-1}+t,B_{m-1}+t]_\rho, l_{m-1}, \eta_{m-1}), ([A_m+t,B_m+t]_\rho, l_m, \eta_m)\}$
with 
$\{([A_{m-1}+t,B_{m-1}+t]_{\rho'}, l_{m-1}, \eta_{m-1}), ([A_m+t,B_m+t]_{\rho'}, l_m, \eta_m)\}$
for some $\rho'$, i.e., 
we can replace $I_\rho$ with the union of $I_\rho = \{1, \dots, m-2\}$ and $I_{\rho'}=\{m-1,m\}$. 
Hence by the induction hypothesis, we have $\pi(\EE^{\sharp}) \not= 0$. 
Similarly, using the partition $I_\rho = \{1, \dots, m-2,m-1\} \sqcup \{m\}$ 
(\resp $I_\rho = \{1, \dots, m-2,m\} \sqcup \{m-1\}$),
by the induction hypothesis, we have $\pi(\EE^{\flat}) \not=0$ (\resp $\pi(\EE^{\natural}) \not= 0$). 
Therefore, we conclude that $\pi(\EE) \not= 0$, as desired.
\end{proof}

\subsection{Xu's example}
As in \cite[Example B.1]{X3}, let us determine the cardinality $\#\Pi_\psi$ for 
\[
\psi = \rho \boxtimes (S_{51}\boxtimes S_{31} \oplus S_{31} \boxtimes S_{45} \oplus S_{13} \boxtimes S_5).
\]
Set $(a_i,b_i) = (13,5)$, $(a_j,b_j) = (31,45)$ and $(a_k,b_k) = (51,31)$
so that $[A_i,B_i]_\rho = [8,4]_\rho$, $[A_j,B_j]_\rho = [37,-7]_\rho$ and $[A_k,B_k]_\rho = [40,10]_\rho$.
There are exactly two admissible orders $j \prec i \prec k$ and $i \prec' j \prec' k$, 
and the first one satisfies the condition $(\PP')$.
Let 
\begin{align*}
\EE &= \{([A_i,B_i]_\rho, l_i, \eta_i), ([A_j,B_j]_\rho, l_j, \eta_j), ([A_k,B_k]_\rho, l_k, \eta_k)\}, 
\\
\EE' &= \{([A_i,B_i]_\rho, l'_i, \eta'_i), ([A_j,B_j]_\rho, l'_j, \eta'_j), ([A_k,B_k]_\rho, l'_k, \eta'_k)\}, 
\end{align*}
where the admissible order $\succ$ (\resp $\succ'$) is used for $\EE$ (\resp $\EE'$), 
be such that $\pi(\EE) \cong \pi(\EE')$.
Note that
\[
\EE = 
\left(
\begin{aligned}
\underbrace{\overset{-7}{\lhd} \cdots\cdots\cdots\cdots \lhd}_{l_j} 
\overset{-7+l_j}{\odot} \cdots\cdots\cdots\cdots &\overset{37-l_j}{\odot} 
\underbrace{\rhd \cdots\cdots\cdots\cdots \overset{37}{\rhd}}_{l_j}\\
\underbrace{\overset{4}{\lhd} \cdots \lhd}_{l_i} 
\odot \cdots \odot
\underbrace{\rhd \cdots \overset{8}{\rhd}}_{l_i} \phantom{\odot}&\\
&\underbrace{\overset{10}{\lhd} \cdots\cdots \lhd}_{l_k} 
\odot \cdots\cdots \odot
\underbrace{\rhd \cdots\cdots \overset{40}{\rhd}}_{l_k}
\end{aligned}
\right)_\rho.
\]
By Theorems \ref{nonzero2} and \ref{nonzero}, $\pi(\EE)$ is nonzero if and only if 
\begin{itemize}
\item
$l_j \geq 7$; 
\item
the condition in Proposition \ref{nec} (3) holds for $i \succ j$; 
\item
the condition in Proposition \ref{nec} (1) holds for $k \succ i$; 
\item
the condition in Proposition \ref{nec} (1) holds for $k \succ' j$.
\end{itemize}
However, when $l_j \geq 7 \geq \#[A_i,B_i] = 5$, 
the condition in Proposition \ref{nec} (3) holds for $i \succ j$. 
Similarly, since $B_k > A_i$, the condition in Proposition \ref{nec} (1) for $k \succ i$ is trivial.
\par

Now we determine when the condition in Proposition \ref{nec} (1) holds for $k \succ' j$ 
(under assuming $l_j \geq 7$).
Note that $(\eta_i,\eta_j,\eta_k)$ is determined by equations between $\eta_i, \eta_j, \eta_k$ 
since the sign condition $\eta_i\eta_j\eta_k = (-1)^{l_i+l_j+l_k}$ is required. 
\par

To exchange the first and second lines, we separate several cases. 
\begin{enumerate}
\item
Suppose that $\eta_i = \eta_j$ and $45-2l_j \geq 2(5-2l_i)$, i.e., $7 \leq l_j \leq 17+2l_i$. 
By Theorem \ref{change}, we have
$(l_j', \eta_j') = (l_j+5-2l_i, -\eta_j)$, i.e.,
\[
\EE' = 
\left(
\begin{aligned}
\underbrace{\overset{4}{\lhd} \cdots \lhd}_{l_i} 
\odot \cdots \odot
\underbrace{\rhd \cdots \overset{8}{\rhd}}_{l_i} \phantom{\odot}&\\
\underbrace{\overset{-7}{\lhd} \cdots\cdots\cdots\cdots\cdots \lhd}_{l_j+5-2l_i} 
\overset{l_j-2-2l_i}{\odot} \cdots\cdots &\overset{32-l_j+2l_i}{\odot} 
\underbrace{\rhd \cdots\cdots\cdots\cdots\cdots \overset{37}{\rhd}}_{l_j+5-2l_i}\\
&\underbrace{\overset{10}{\lhd} \cdots\cdots \lhd}_{l_k} 
\odot \cdots\cdots \odot
\underbrace{\rhd \cdots\cdots \overset{40}{\rhd}}_{l_k}
\end{aligned}
\right)_\rho, 
\]
where the last $\odot$ in the second line is $-\eta_j$.
Therefore, in this case, the condition in Proposition \ref{nec} (1) for $k \succ' j$ says that
\begin{itemize}
\item
if $\eta_k \not= \eta_j$, then $l_j-2-2l_i \leq 10+l_k$ and $32-l_j+2l_i \leq 40-l_k$, 
i.e., $\max\{0,l_j-2l_i-12\} \leq l_k \leq \min\{15, 8+l_j-2l_i\}$; 
\item
if $\eta_k = \eta_j$, then $33-l_j+2l_i \leq 10+l_k$, 
i.e., $\max\{0,23-l_j+2l_i\} \leq l_k \leq 15$.
\end{itemize}
In particular, for fixed $\eta_i = \eta_j$ and $(l_i,l_j)$ with $7 \leq l_j \leq 17+2l_i$, 
\[
\#\{ (l_k,\eta_k) \;|\; \pi(\EE) \not= 0\}
=
\left\{
\begin{aligned}
&9+l_j-2l_i \iif 7 \leq l_j \leq 12+2l_i, \\
&21 \iif 12+2l_i \leq l_j \leq 17+2l_i.
\end{aligned}
\right. 
\]

\item
Suppose that $\eta_i = \eta_j$ and $45-2l_j < 2(5-2l_i)$, i.e., $17+2l_i < l_j \leq 15$. 
By Theorem \ref{change}, we have $(l_j', \eta_j') = (40+2l_i-l_j, \eta_j)$, i.e.,
\[
\EE' = 
\left(
\begin{aligned}
\underbrace{\overset{4}{\lhd} \cdots \lhd}_{l_i} 
\odot \cdots \odot
\underbrace{\rhd \cdots \overset{8}{\rhd}}_{l_i} \phantom{\odot}&\\
\underbrace{\overset{-7}{\lhd} \cdots\cdots\cdots\cdots\cdots \lhd}_{40+2l_i-l_j} 
\overset{33+2l_i-l_j}{\odot} \cdots\cdots &\overset{-3-2l_1+l_2}{\odot} 
\underbrace{\rhd \cdots\cdots\cdots\cdots\cdots \overset{37}{\rhd}}_{40+2l_i-l_j}\\
&\underbrace{\overset{10}{\lhd} \cdots\cdots \lhd}_{l_k} 
\odot \cdots\cdots \odot
\underbrace{\rhd \cdots\cdots \overset{40}{\rhd}}_{l_k}
\end{aligned}
\right)_\rho, 
\]
where the last $\odot$ in the second line is $\eta_j$.
Therefore, in this case, the condition in Proposition \ref{nec} (1) for $k \succ' j$ says that
\begin{itemize}
\item
if $\eta_k = \eta_j$, then $33+2l_i-l_j \leq 10+l_k$ and $-3-2l_i+l_j \leq 40-l_k$, 
i.e., $\max\{0,23+2l_i-l_j\} \leq l_k \leq \min\{15, 43+2l_i-l_j\}$; 
\item
if $\eta_k \not= \eta_j$, then $-2-2l_i+l_j \leq 10+l_k$, 
i.e., $\max\{0,-12-2l_i+l_j\} \leq l_k \leq 15$.
\end{itemize}
Since we assume that $17 < l_j-2l_i \leq 15$, 
it is equivalent that 
\begin{itemize}
\item
if $\eta_k = \eta_j$, then $23+2l_i-l_j \leq l_k \leq 15$; 
\item
if $\eta_k \not= \eta_j$, then $-12-2l_i+l_j \leq l_k \leq 15$. 
\end{itemize}
In particular, for fixed $\eta_i = \eta_j$ and $(l_i,l_j)$ with $17+2l_i < l_j \leq 22$, we have
\[
\#\{ (l_k,\eta_k) \;|\; \pi(\EE) \not= 0\} = 21. 
\]

\item
Suppose that $\eta_i \not= \eta_j$.
By Theorem \ref{change}, we have $(l_j', \eta_j') = (l_j+2l_i-5, \eta_i)$, i.e.,
\[
\EE' = 
\left(
\begin{aligned}
\underbrace{\overset{4}{\lhd} \cdots \lhd}_{l_i} 
\odot \cdots \odot
\underbrace{\rhd \cdots \overset{8}{\rhd}}_{l_i} \phantom{\odot}&\\
\underbrace{\overset{-7}{\lhd} \cdots\cdots\cdots\cdots\cdots \lhd}_{l_j+2l_i-5} 
\overset{l_j+2l_i-12}{\odot} \cdots\cdots &\overset{42-2l_i-l_j}{\odot} 
\underbrace{\rhd \cdots\cdots\cdots\cdots\cdots \overset{37}{\rhd}}_{l_j+2l_i-5}\\
&\underbrace{\overset{10}{\lhd} \cdots\cdots \lhd}_{l_k} 
\odot \cdots\cdots \odot
\underbrace{\rhd \cdots\cdots \overset{40}{\rhd}}_{l_k}
\end{aligned}
\right)_\rho, 
\]
where the last $\odot$ in the second line is $\eta_i$.
Therefore, in this case, the condition in Proposition \ref{nec} (1) for $k \succ' j$ says that
\begin{itemize}
\item
if $\eta_k = \eta_i$, then $2l_i+l_j-12 \leq 10+l_k$ and $42-2l_i-l_j \leq 40-l_k$, 
i.e., $\max\{0,2l_i+l_j-22\} \leq l_k \leq \min\{15, 2l_i+l_j-2\}$; 
\item
if $\eta_k \not= \eta_i$, then $43-2l_i-l_j \leq 10+l_k$, 
i.e., $\max\{0,33-2l_i-l_j\} \leq l_k \leq 15$.
\end{itemize}
In particular, for fixed $\eta_i \not= \eta_j$ and $(l_i,l_j)$, we have
\[
\#\{ (l_k,\eta_k) \;|\; \pi(\EE) \not= 0\} =
\left\{
\begin{aligned}
&2l_i+l_j-1 \iif 7 \leq l_j \leq 22-2l_i, \\
&21 \iif 22-2l_i \leq l_j \leq 22.
\end{aligned}
\right. 
\]
\end{enumerate}
\par

In conclusion, 
\begin{align*}
\#\Pi_\psi &= \sum_{l_i=0}^{2} \left(
\sum_{l_j=7}^{12+2l_i}(9+l_j-2l_i) + 21(10-2l_i) + \sum_{l_j=7}^{22-2l_i}(2l_i+l_j-1) + 21\cdot 2l_i
\right)
\\&= \sum_{l_i=0}^{2} \left(
2(16-2l_i)(6+2l_i) 
+\sum_{l_j=7}^{12+2l_i}(l_j-7)+\sum_{l_j=7}^{22-2l_i}(l_j-7)
+ 210
\right)
\\&= \sum_{l_i=0}^{2} ( -4l_i^2 + 20l_i + 537)
= 537 + 553 + 561 = 1651.
\end{align*}
This is consistent with the conclusion of \cite[Example B.1]{X3}.

\section{Deformations and algorithms for derivatives}\label{s.deform}
Let $\EE = \cup_\rho \{([A_i,B_i]_{\rho}, l_i, \eta_i)\}_{i \in (I_\rho,\succ)}$ 
be an extended multi-segment for $G_n$ such that $\pi(\EE) \not= 0$. 
In this section, we specify $\pi(\EE)$. 
If $A_i \geq A_j$ and $B_i \geq B_j \geq 0$ whenever $i \succ j$, 
an explicit formula for the Langlands data for $\pi(\EE)$ was already given by Xu \cite[Theorem 1.3]{X4}. 
However, this case is enough complicated to write down the detail. 
\par

Instead of giving the Langlands data for $\pi(\EE)$ explicitly,
we will give algorithms to compute certain derivatives. 
Together with \cite[Theorem 7.1]{AM} and a special case of \cite[Theorem 1.3]{X4}, 
these algorithms would determine the Langlands data for $\pi(\EE)$.

\subsection{Derivatives of certain representations}
In this subsection, we describe derivatives of $\pi(\EE)$ in certain special situations.
Suppose that the admissible order $\succ$ on $I_\rho$ satisfies $(\PP')$, i.e., $B_i > B_j \implies i \succ j$. 
\par

Fix $\rho$. 
We set 
\begin{align*}
B^{\max} = \max\{B_i \;|\; i \in I_\rho\}, 
&\quad
I_\rho^{\max} = \{i \in I_\rho \;|\; B_i = B^{\max}\}, \\
B^{\min} = \min\{B_i \;|\; i \in I_\rho\}, 
&\quad
I_\rho^{\min} = \{i \in I_\rho \;|\; B_i = B^{\min}\}.
\end{align*}
Write 
\begin{align*}
\bigsqcup_{i \in I_\rho^{\max}} [A_i,B_i]_\rho &= \{\rho|\cdot|^{x_1}, \dots, \rho|\cdot|^{x_t}\}, \\
\bigsqcup_{i \in I_\rho^{\min}} [B_i,-A_i]_\rho &= \{\rho|\cdot|^{y_1}, \dots, \rho|\cdot|^{y_s}\}
\end{align*}
as multi-sets such that $x_1 \leq \dots \leq x_t$ and $y_1 \geq \dots \geq y_s$. 
(Here, the left hand sides are regarded as disjoint unions.)
\par

\begin{thm}\label{derivatives}
Notation is as above. 
\begin{enumerate}
\item
Suppose that $B^{\max} > 0$. 
Define $\EE^+$ from $\EE$ by replacing $([A_i,B_i]_\rho,l_i,\eta_i)$ 
with $([A_i-1,B_i-1]_\rho,l_i,\eta_i)$ for any $i \in I_\rho^{\max}$. 
If $\pi(\EE^+) \not= 0$, then
\[
D_{\rho|\cdot|^{x_t}} \circ \dots \circ D_{\rho|\cdot|^{x_1}}\left(\pi(\EE) \right) 
\cong \pi(\EE^+)
\]
up to a multiplicity. 
In particular, 
\[
\pi(\EE) \cong 
S_{\rho|\cdot|^{x_1}} \circ \dots \circ S_{\rho|\cdot|^{x_t}}\left( \pi(\EE^+) \right). 
\]

\item
Assume that $l_i \geq 1$ for any $i \in I_\rho^{\min}$. 
Define $\EE^-$ from $\EE$ by replacing $([A_i,B_i]_\rho,l_i,\eta_i)$ with 
$([A_i-1,B_i+1]_\rho, l_i-1,\eta_i)$ for any $i \in I_\rho^{\min}$.
If $\pi(\EE^-) \not= 0$, then
\[
D_{\rho|\cdot|^{y_s}} \circ \dots \circ D_{\rho|\cdot|^{y_1}}\left(\pi(\EE) \right) 
\cong \pi(\EE^-)
\]
up to a multiplicity. 
In particular, 
\[
\pi(\EE) \cong 
S_{\rho|\cdot|^{y_1}} \circ \dots \circ S_{\rho|\cdot|^{y_s}}\left( \pi(\EE^-) \right). 
\]
\end{enumerate}
\end{thm}

\begin{proof}
We show (1). 
Write $I_\rho^{\max} = \{1,2, \dots, m\}$ with $1 \prec \dots \prec m$. 
Take integers $0< t_1 < \dots < t_m$ and define $\EE_1$ from $\EE$ by replacing 
$[A_i, B_i]_\rho$ with $[A_i+t_i, B_i+t_i]_\rho$ for $i \in I_\rho^{\max}$. 
By choosing $0< t_1 < \dots < t_m$ appropriately, 
we may assume that 
\begin{align*}
\pi(\EE) &= 
\circ_{i \in I_\rho^{\max}}\left(
D_{\rho|\cdot|^{B_i+1}, \dots, \rho|\cdot|^{A_i+1}} \circ \dots \circ D_{\rho|\cdot|^{B_i+t_i}, \dots, \rho|\cdot|^{A_i+t_i}}
\right) (\pi(\EE_1)), \\
\pi(\EE^+) &= 
\circ_{i \in I_\rho^{\max}}\left(
D_{\rho|\cdot|^{B_i}, \dots, \rho|\cdot|^{A_i}} \circ \dots \circ D_{\rho|\cdot|^{B_i+t_i}, \dots, \rho|\cdot|^{A_i+t_i}}
\right) (\pi(\EE_1)). 
\end{align*}
Since $\pi(\EE^+) \not= 0$, 
by \cite[Proposition 8.5]{X2}, we have
\begin{align*}
\pi(\EE_1) 
\hookrightarrow 
\begin{pmatrix}
B_1+t_1 & \ldots & A_1+t_1 \\
\vdots & \ddots & \vdots \\
B_1 & \ldots & A_1
\end{pmatrix}_\rho
\times \dots \times 
\begin{pmatrix}
B_m+t_m & \ldots & A_m+t_m \\
\vdots & \ddots & \vdots \\
B_m & \ldots & A_m
\end{pmatrix}_\rho
\rtimes \pi(\EE^+).
\end{align*}
Since $B_1 = \dots = B_m = B^{\max}$, by Theorem \ref{speh}, 
the two Speh representations $Z_\rho[B_j,A_j]$ and 
$\begin{pmatrix}
B_i+t_i & \ldots & A_i+t_i \\
\vdots & \ddots & \vdots \\
B_i & \ldots & A_i
\end{pmatrix}_\rho$
are commutative for any $i,j \in I_\rho^{\max}$.
Therefore, we have
\[
\pi(\EE) 
\hookrightarrow 
Z_\rho[B_1, A_1] \times \dots \times Z_\rho[B_m,A_m] \rtimes \pi(\EE^+).
\]
Since $\sqcup_{i=1}^m[A_i,B_i]_\rho = \{\rho|\cdot|^{x_1}, \dots, \rho|\cdot|^{x_t}\}$ as multi-sets, 
we have 
\[
D_{\rho|\cdot|^{x_t}} \circ \dots \circ D_{\rho|\cdot|^{x_1}}\left(\pi(\EE) \right) 
\cong \pi(\EE^+)
\]
up to a multiplicity. 
This proves (1). 
\par

We show (2). 
Write $I_\rho \setminus I_\rho^{\min} = \{1, \dots, m\}$ with $1 \prec \dots \prec m$. 
Take integers $0< t_1 < \dots < t_m$ and define $\EE_2$ (\resp $\EE_2^-$) from $\EE$ (\resp $\EE^-$) by replacing 
$[A_i, B_i]_\rho$ with $[A_i+t_i, B_i+t_i]_\rho$ for $i \in I_\rho \setminus I_\rho^{\min}$. 
By choosing $0< t_1 < \dots < t_m$ appropriately, 
we may assume that 
\begin{align*}
\pi(\EE) &= 
\circ_{i \in I_\rho \setminus I_\rho^{\min}}\left(
D_{\rho|\cdot|^{B_i+1}, \dots, \rho|\cdot|^{A_i+1}} \circ \dots \circ D_{\rho|\cdot|^{B_i+t_i}, \dots, \rho|\cdot|^{A_i+t_i}}
\right) (\pi(\EE_2)), \\
\pi(\EE^-) &= 
\circ_{i \in I_\rho \setminus I_\rho^{\min}}\left(
D_{\rho|\cdot|^{B_i+1}, \dots, \rho|\cdot|^{A_i+1}} \circ \dots \circ D_{\rho|\cdot|^{B_i+t_i}, \dots, \rho|\cdot|^{A_i+t_i}}
\right) (\pi(\EE_2^-)), 
\end{align*}
and that 
\[
\pi(\EE_2) \hookrightarrow 
\left( \bigtimes_{j \in I_\rho^{\min}}\Delta_\rho[B_j, -A_j] \right)
\rtimes \pi(\EE_2^-). 
\]
For this inclusion, use \cite[Theorem 1.3]{X4}. 
Since $\pi(\EE^-) \not= 0$, 
by \cite[Proposition 8.5]{X2}, we have
\begin{align*}
\pi(\EE_2^-) 
\hookrightarrow 
\begin{pmatrix}
B_1+t_1 & \ldots & A_1+t_1 \\
\vdots & \ddots & \vdots \\
B_1+1 & \ldots & A_1+1
\end{pmatrix}_\rho
\times \dots \times 
\begin{pmatrix}
B_m+t_m & \ldots & A_m+t_m \\
\vdots & \ddots & \vdots \\
B_m+1 & \ldots & A_m+1
\end{pmatrix}_\rho
\rtimes \pi(\EE^-).
\end{align*}
Since $B_1, \dots, B_m > B^{\min}$, 
by Theorem \ref{speh}, 
the two Speh representations $\Delta_\rho[B_j,-A_j]$ and 
$\begin{pmatrix}
B_i+t_i & \ldots & A_i+t_i \\
\vdots & \ddots & \vdots \\
B_i+1 & \ldots & A_i+1
\end{pmatrix}_\rho$
are commutative for any $j \in I_\rho^{\min}$ and $i \in I_\rho \setminus I_\rho^{\min}$.
Therefore, we have
\[
\pi(\EE) \hookrightarrow 
\left( \bigtimes_{j \in I_\rho^{\min}}\Delta_\rho[B_j, -A_j] \right)
\rtimes \pi(\EE^-). 
\]
Since $\sqcup_{j \in I_\rho^{\min}}[B_j,-A_j]_\rho = \{\rho|\cdot|^{y_1}, \dots, \rho|\cdot|^{y_s}\}$ as multi-sets, 
we have 
\[
D_{\rho|\cdot|^{y_s}} \circ \dots \circ D_{\rho|\cdot|^{y_1}}\left(\pi(\EE) \right) 
\cong \pi(\EE^-)
\]
up to a multiplicity. 
This proves (2). 
\end{proof}

\subsection{Deformation}\label{sec.deform}
To describe derivatives of $\pi(\EE)$ in general, 
in Algorithms \ref{alg+} and \ref{alg-},
we will construct $\EE^*$ from $\EE$ which is in the situation of Theorem \ref{derivatives}. 
The key idea to do this is to study the condition in Proposition \ref{nec} (1) more deeply. 
\par

First of all, 
we assume that 
\[
\EE = \{([A_{k-1}, B_{k-1}]_\rho, l_{k-1}, \eta_{k-1}), ([A_k,B_k]_\rho, l_k, \eta_k)\}
\]
consists of two extended segments with $k \succ k-1$. 
Recall that $b_{k-1} = \#[A_{k-1},B_{k-1}] = A_{k-1}-B_{k-1}+1$.
Suppose that $A_k > A_{k-1}$, $B_k > B_{k-1} \geq 0$ and that one of the following holds: 
\begin{enumerate}
\item
$\eta_{k} = (-1)^{A_{k-1}-B_{k-1}}\eta_{k-1}$ and $A_k - l_k = A_{k-1} - l_{k-1}$; 
\item
$\eta_{k} = (-1)^{A_{k-1}-B_{k-1}}\eta_{k-1}$ and $B_k + l_k = B_{k-1} + l_{k-1}$; 
\item
$\eta_{k} \not= (-1)^{A_{k-1}-B_{k-1}}\eta_{k-1}$ and $B_k + l_k = A_{k-1} - l_{k-1} + 1$.
\end{enumerate}
Note that (1) implies $l_k > l_{k-1}$, whereas (2) implies $l_k < l_{k-1}$.
In particular, (1) and (2) cannot occur simultaneously. 
We will define $\EE' = \{([A'_{k-1}, B'_{k-1}]_\rho, l'_{k-1}, \eta'_{k-1}), ([A'_k,B'_k]_\rho, l'_k, \eta'_k)\}$
so that $[A_{k-1}',B_{k-1}']_\rho = [A_{k-1},B_{k-1}]_\rho \cup [A_k,B_k]_\rho$ 
and $[A_{k}',B_{k}']_\rho = [A_{k-1},B_{k-1}]_\rho \cap [A_k,B_k]_\rho$ as sets (not multi-sets), 
and so that $\pi(\EE) \cong \pi(\EE')$. 

\begin{thm}\label{union}
Set $[A_{k-1}',B_{k-1}']_\rho = [A_{k},B_{k-1}]_\rho$ 
and $[A_{k}',B_{k}']_\rho = [A_{k-1},B_{k}]_\rho$. 
Define $l_i'$ and $\eta_i'$ for $i \in \{k-1, k\}$ as follows. 

\begin{enumerate}
\item
When $\eta_{k} = (-1)^{A_{k-1}-B_{k-1}}\eta_{k-1}$ and $A_k - l_k = A_{k-1} - l_{k-1}$, 
set 
\[
(l'_{k-1}, l'_k, \eta'_{k-1}, \eta'_k)
= \left(l_{k-1}, l_k-(A_k-A_{k-1}), \eta_{k-1}, (-1)^{A_k-A_{k-1}}\eta_k\right).
\]
Note that $l_k-(A_k-A_{k-1}) = l_{k-1}$.

\item
When $\eta_{k} = (-1)^{A_{k-1}-B_{k-1}}\eta_{k-1}$ and $B_k + l_k = B_{k-1} + l_{k-1}$, 
set 
\[
(l'_{k-1}, l'_k, \eta'_{k-1}, \eta'_k)
= \left\{
\begin{aligned}
&\left(l_{k-1}+(A_k-A_{k-1}), l_k, \eta_{k-1}, (-1)^{A_k-A_{k-1}}\eta_k\right), \\
&\left(b_{k-1}-l_{k-1}, l_k, -\eta_{k-1}, (-1)^{A_k-A_{k-1}}\eta_k\right) 
\end{aligned}
\right. 
\]
according to $b_{k-1} - 2l_{k-1} \geq A_k-A_{k-1}$ or not. 
Note that $b_{k-1}+(A_k-A_{k-1}) = A_k-B_{k-1}+1$ so that if $b_{k-1}-2l_{k-1} = A_k-A_{k-1}$, 
then $\EE'$ does not depend on $\eta_{k-1}'$. 

\item
When $\eta_{k} \not= (-1)^{A_{k-1}-B_{k-1}}\eta_{k-1}$ and $B_k + l_k = A_{k-1} - l_{k-1} + 1$, 
set 
\[
(l'_{k-1}, l'_k, \eta'_{k-1}, \eta'_k)
= \left\{
\begin{aligned}
&\left(l_{k-1}, l_k, \eta_{k-1}, (-1)^{A_k-A_{k-1}}\eta_k\right) \iif l_k \leq l_{k-1}, \\
&\left(l_{k-1}, l_{k-1}, \eta_{k-1}, (-1)^{A_k-A_{k-1}-1}\eta_k\right) \iif l_k > l_{k-1}.
\end{aligned}
\right. 
\]
Note that if $l_k = l_{k-1}$, then $\EE'$ does not depend on $\eta_{k}'$. 
\end{enumerate}
Then we have $\pi(\EE) \cong \pi(\EE')$. 
\end{thm}
\begin{proof}
For simplicity, we write $k = 2$.
We will prove the assertion by using the symbol of $\EE = \EE_\rho$. 
It might be a formal proof, but it is justified by the explicit formula for $\pi(\EE)$ (\cite[Theorem 1.3]{X4}).
\par

In the case (1), $\EE = \EE_\rho$ is of the form
\[
\left(
\begin{array}{rrllcll}
\overset{B_1}{\lhd} & \cdots \lhd & \odot \cdots & \cdots \cdots &\overset{A_1-l_1}{\odot} 
&\rhd \cdots &\overset{A_1}{\rhd}\cr
&\underset{B_2}{\lhd} \cdots & \cdots \lhd &
\odot \cdots &\underset{A_2-l_2}{\odot} 
&\rhd \cdots & \cdots \underset{A_2}{\rhd}
\end{array}
\right)_\rho 
\]
such that the sign of the first $\odot$ in the first line (\resp the second line) is $\eta_1$ (\resp $\eta_2$). 
Taking the socle 
\[
S_{\rho|\cdot|^{A_2}, \dots \rho|\cdot|^{A_2-l_1+1}}
\circ \dots \circ 
S_{\rho|\cdot|^{A_1+2}, \dots \rho|\cdot|^{A_1-l_1+3}} 
\circ 
S_{\rho|\cdot|^{A_1+1}, \dots \rho|\cdot|^{A_1-l_1+2}}, 
\]
it becomes 
\[
\left(
\begin{array}{rlllclll}
\overset{B_1}{\lhd} & \cdots \lhd  \odot & \cdots \cdots & \cdots \cdots &\overset{A_1-l_1}{\odot} 
&&\rhd \cdots \overset{A_2}{\rhd}\cr
&\underset{B_2}{\lhd} \cdots \lhd & \underbrace{\lhd \cdots \lhd}_{A_2-A_1} &
\odot \cdots &\underset{A_2-l_2}{\odot} 
&\underbrace{\rhd \cdots \rhd}_{A_2-A_1} & \rhd\cdots \underset{A_2}{\rhd}
\end{array}
\right)_\rho, 
\]
whose representation is isomorphic to the one of
\[
\left(
\begin{array}{rlllclll}
\overset{B_1}{\lhd} & \cdots \lhd  \odot & \cdots \cdots & \cdots \cdots &\overset{A_1-l_1}{\odot} 
&\overbrace{\odot \cdots \odot}^{A_2-A_1}&\rhd \cdots \overset{A_2}{\rhd}\cr
&\underset{B_2}{\lhd} \cdots \lhd & \underbrace{\odot \cdots \odot}_{A_2-A_1} &
\odot \cdots &\underset{A_2-l_2}{\odot} 
&& \rhd\cdots \underset{A_2}{\rhd}
\end{array}
\right)_\rho, 
\]
where the first $\odot$ in the second line is equal to the last one in the first line
which is $(-1)^{A_2-B_1}\eta_1 = (-1)^{A_2-A_1}\eta_2$.
By Taking the derivative 
\[
D_{\rho|\cdot|^{A_1-l_2+2}, \dots, \rho|\cdot|^{A_1+1}} \circ \dots \circ 
D_{\rho|\cdot|^{A_2-l_1+1}, \dots, \rho|\cdot|^{A_2}}, 
\]
with $(l_1', l_2') = (l_1,l_1)$ and $(\eta_1', \eta_2') = (\eta_1, (-1)^{A_2-A_1}\eta_2)$, 
we see that $\pi(\EE) \cong \pi(\EE')$. 
This proves (1). 
\par

In the case (2), $\EE = \EE_\rho$ is of the form
\[
\left(
\begin{array}{rrccrl}
\overset{B_1}{\lhd} & \cdots \lhd 
&\overset{B_1+l_1}{\odot} & \cdots \odot
&\rhd \cdots & \overset{A_1}{\rhd}\cr
&\underset{B_2}{\lhd} \cdots \lhd 
&\underset{B_2+l_2}{\odot}& \cdots & \cdots \odot 
&\rhd \cdots \underset{A_2}{\rhd}
\end{array}
\right)_\rho 
\]
such that the sign of the first $\odot$ in the first line (\resp the second line) is $\eta_1$ (\resp $\eta_2$). 
We will compute the socle
\[
S_{\rho|\cdot|^{A_1+1}, \dots, \rho|\cdot|^{B_1+l_1+1}} 
\]
by separating the several cases when $b_1 - 2l_1 > 0$. 
We write $b_1 -2l_1 = 2\alpha + \delta$ for $\alpha \in \Z$ and $\delta \in \{0,1\}$.

\begin{enumerate}
\item[(a)]
When $\delta = 0$ and $\alpha > 0$, we write
\[
\left(
\begin{array}{rrclrl}
\overset{B_1}{\lhd} & \cdots \lhd 
&\overbrace{\odot \cdots \odot}^\alpha & \overbrace{\odot \cdots \odot}^\alpha
&\rhd \cdots & \overset{A_1}{\rhd}\cr
&\underset{B_2}{\lhd} \cdots \lhd 
&\underbrace{\odot \cdots \odot}_\alpha & \odot \cdots & \cdots \odot 
&\rhd \cdots \underset{A_2}{\rhd}
\end{array}
\right)_\rho. 
\]
Using \cite[Theorem 7.1]{AM}, by taking the socle, it becomes
\[
\left(
\begin{array}{rrllrl}
\overset{B_1}{\lhd} & \cdots \lhd 
&\phantom{\odot}\overbrace{\odot \cdots \odot}^\alpha & \overbrace{\odot \cdots \odot}^\alpha
&\rhd \cdots & \overset{A_1+1}{\rhd}\cr
&\underset{B_2}{\lhd} \cdots \lhd 
&\underbrace{\odot \cdots \odot}_\alpha \odot & \cdots \cdots & \cdots \odot 
&\rhd \cdots \underset{A_2}{\rhd}
\end{array}
\right)_\rho,
\]
whose representation is isomorphic to the one of
\[
\left(
\begin{array}{rrllrl}
\overset{B_1}{\lhd} & \cdots \lhd 
&\lhd \overbrace{\odot \cdots \phantom{\cdots} \odot}^\alpha & \overbrace{\odot \cdots \odot}^{\alpha-1} \rhd
&\rhd \cdots & \overset{A_1+1}{\rhd}\cr
&\underset{B_2}{\lhd} \cdots \lhd 
&\phantom{\odot}\underbrace{\odot \cdots \odot}_{\alpha-1}\; \odot & \cdots \cdots & \cdots \odot 
&\rhd \cdots \underset{A_2}{\rhd}
\end{array}
\right)_\rho.
\]

\item[(b)]
When $\delta = 1$ and $\alpha > 0$, we write
\[
\left(
\begin{array}{rrcclrl}
\overset{B_1}{\lhd} & \cdots \lhd 
&\overbrace{\odot \cdots \odot}^\alpha & \odot & \overbrace{\odot \cdots \odot}^\alpha
&\rhd \cdots & \overset{A_1}{\rhd}\cr
&\underset{B_2}{\lhd} \cdots \lhd 
&\underbrace{\odot \cdots \odot}_\alpha & \odot & \odot \cdots & \cdots \odot 
&\rhd \cdots \underset{A_2}{\rhd}
\end{array}
\right)_\rho. 
\]
Using \cite[Theorem 7.1]{AM}, by taking the socle, it becomes
\[
\left(
\begin{array}{rrcclrl}
\overset{B_1}{\lhd} & \cdots \lhd 
&\phantom{\odot}\overbrace{\odot \cdots \odot}^\alpha & \odot & \overbrace{\odot \cdots \odot}^\alpha
&\rhd \cdots & \overset{A_1+1}{\rhd}\cr
&\underset{B_2}{\lhd} \cdots \lhd 
&\underbrace{\odot \cdots \odot}_\alpha \odot & \odot & \cdots \cdots & \cdots \odot 
&\rhd \cdots \underset{A_2}{\rhd}
\end{array}
\right)_\rho. 
\]
Here, we used the following fact (or its sign change). 
The socle map $S_{\rho|\cdot|^x}$ gives 
\[
\pi
\bordermatrix{
& x-1 & x \cr
& \ominus &   \cr
& \ominus & \oplus 
}_\rho
\mapsto 
\pi
\bordermatrix{
& x-1 & x \cr
& & \ominus  \cr
& \ominus & \oplus 
}_\rho
\cong 
\pi
\bordermatrix{
& x-1 & x \cr
& \lhd & \rhd  \cr
& & \oplus 
}_\rho. 
\]
The representation associated to the last symbol is isomorphic to the one of
\[
\left(
\begin{array}{rrcclrl}
\overset{B_1}{\lhd} & \cdots \lhd 
&\lhd \overbrace{\odot \cdots  \phantom{\odot}  \odot}^\alpha & \odot & \overbrace{\odot \cdots \odot}^{\alpha-1} \rhd
&\rhd \cdots & \overset{A_1+1}{\rhd}\cr
&\underset{B_2}{\lhd} \cdots \lhd 
&\phantom{\odot}\underbrace{\odot \cdots \odot}_{\alpha-1} \odot & \odot & \cdots \cdots & \cdots \odot 
&\rhd \cdots \underset{A_2}{\rhd}
\end{array}
\right)_\rho. 
\]

\item[(c)]
When $\delta = 1$ and $\alpha = 0$, we write
\[
\left(
\begin{array}{rrcrl}
\overset{B_1}{\lhd} & \cdots \lhd 
& \overset{B_1+l_1}{\odot} 
&\rhd \cdots & \overset{A_1}{\rhd}\cr
&\underset{B_2}{\lhd} \cdots \lhd 
& \underset{B_2+l_2}{\odot} & \cdots \odot 
&\rhd \cdots \underset{A_2}{\rhd}
\end{array}
\right)_\rho. 
\]
Using \cite[Theorem 7.1]{AM}, by taking the socle, it becomes
\[
\left(
\begin{array}{rrcrl}
\overset{B_1}{\lhd} & \cdots \lhd 
& \lhd \overset{B_1+l_1+1}{\rhd} 
&\rhd \cdots & \overset{A_1+1}{\rhd}\cr
&\underset{B_2}{\lhd} \cdots \lhd 
&\phantom{\odot} \underset{B_2+l_2+1}{\odot} & \cdots \odot 
&\rhd \cdots \underset{A_2}{\rhd}
\end{array}
\right)_\rho. 
\]
Here, we used the same fact as in (b) above. 
\end{enumerate}
In conclusion, 
if $b_1-2l_1 > 0$, i.e., if $\odot$ appears in the first line, after taking the socle, 
the number of $\odot$ in the first line becomes $b_1 - 2l_1 - 1$, 
and $(l_1,l_2,\eta_1,\eta_2)$ becomes $(l_1+1,l_2, \eta_1, -\eta_2)$. 
In particular, if $b_1-2l_1 \geq (A_2-A_1)$, then by repeating these operations $(A_2-A_1)$ times, 
with $(l_1',l_2') = (l_1+A_2-A_1,l_2)$ and $(\eta_1', \eta_2') = (\eta_1, (-1)^{A_2-A_1}\eta_2)$, 
we have $\pi(\EE) \cong \pi(\EE')$. 
\par

From now, we assume that $b_1 - 2l_1 < (A_2-A_1)$. 
Then after the $(b_1-2l_1)$-th step, we have the following symbol
\[
\left(
\begin{array}{rrclrl}
\overset{B_1}{\lhd} & \cdots \lhd 
&\overbrace{\lhd \cdots \lhd}^{b_1-2l_1} &\overbrace{\rhd \cdots \rhd}^{b_1-2l_1}
&\rhd \cdots & \overset{A_1+b_1-2l_1}{\rhd}\cr
&\underset{B_2}{\lhd} \cdots \lhd 
&& \odot \cdots & \cdots \odot 
&\rhd \cdots \underset{A_2}{\rhd}
\end{array}
\right)_\rho, 
\]
where the first $\odot$ in the second line is $(-1)^{b_1}\eta_2 = -\eta_1$.
By taking the appropriate socle, it becomes 
\[
\left(
\begin{array}{rrclrrl}
\overset{B_1}{\lhd} & \cdots \lhd 
&\overbrace{\lhd \cdots \lhd}^{b_1-2l_1} &&\overbrace{\rhd \cdots \rhd}^{b_1-2l_1}
&\rhd \cdots \overset{A_2}{\rhd}\cr
&\underset{B_2}{\lhd} \cdots \lhd 
&& \odot \cdots & \cdots \odot 
&\rhd \cdots \underset{A_2}{\rhd}
\end{array}
\right)_\rho, 
\]
whose representation is isomorphic to the one of
\[
\left(
\begin{array}{rrclrrl}
\overset{B_1}{\lhd} & \cdots \lhd 
&\overbrace{\lhd \cdots \lhd}^{b_1-2l_1} & \odot \cdots \odot&\overbrace{\rhd \cdots \rhd}^{b_1-2l_1}
&\rhd \cdots \overset{A_2}{\rhd}\cr
&\underset{B_2}{\lhd} \cdots \lhd 
&&& \odot \cdots \odot 
&\rhd \cdots \underset{A_2}{\rhd}
\end{array}
\right)_\rho, 
\]
where the first $\odot$ in the first line (\resp the second line) is $-\eta_1$ (\resp $(-1)^{A_2-A_1} \eta_2$).
Therefore, with $(l_1', l_2') = (b_1-l_1,l_2)$ and $(\eta_1', \eta_2') = (-\eta_1, (-1)^{A_2-A_1}\eta_2)$, 
we have $\pi(\EE) \cong \pi(\EE')$. 
This proves (2). 
\par

In the case (3), $\EE = \EE_\rho$ is of the form
\[
\left(
\begin{array}{rrrccl}
\overset{B_1}{\lhd} \cdots &\lhd & \odot \cdots \odot
&\overset{A_1-l_1+1}{\rhd} & \cdots 
& \overset{A_1}{\rhd}\cr
&\underset{B_2}{\lhd} &\cdots \lhd 
&\underset{B_2+l_2}{\odot}& \cdots \odot 
&\rhd \cdots \underset{A_2}{\rhd}
\end{array}
\right)_\rho 
\]
such that the sign of the first $\odot$ in the first line (\resp the second line) is $\eta_1$ (\resp $\eta_2$). 
We take the socle
\[
S_{\rho|\cdot|^{A_2}, \dots \rho|\cdot|^{A_2-l_1+1}}
\circ \dots \circ 
S_{\rho|\cdot|^{A_1+2}, \dots \rho|\cdot|^{A_1-l_1+3}}
\circ
S_{\rho|\cdot|^{A_1+1}, \dots \rho|\cdot|^{A_1-l_1+2}}.  
\]
If $l_1 \geq l_2$, 
it becomes 
\[
\left(
\begin{array}{rrrclr}
\overset{B_1}{\lhd} \cdots &\lhd & \odot \cdots \odot
&& \rhd  \cdots
& \cdots \overset{A_2}{\rhd}\cr
&\underset{B_2}{\lhd} &\cdots \lhd 
&\odot \cdots \odot& \odot \cdots \odot 
&\rhd \cdots \underset{A_2}{\rhd}
\end{array}
\right)_\rho, 
\]
whose representation is isomorphic to the one of
\[
\left(
\begin{array}{rrrclr}
\overset{B_1}{\lhd} \cdots &\lhd & \odot \cdots 
&\cdots \odot & \rhd  \cdots
& \cdots \overset{A_2}{\rhd}\cr
&\underset{B_2}{\lhd} &\cdots \lhd 
&& \odot \cdots \odot 
&\rhd \cdots \underset{A_2}{\rhd}
\end{array}
\right)_\rho.
\]
Hence, with $(l_1', l_2') = (l_1,l_2)$ and $(\eta_1', \eta_2') = (\eta_1, (-1)^{A_2-A_1}\eta_2)$, 
we have $\pi(\EE) \cong \pi(\EE')$. 
If $l_1 < l_2$, it becomes 
\[
\left(
\begin{array}{rrrcrrr}
\overset{B_1}{\lhd} \cdots &\lhd \odot \cdots &\cdots \odot
&&
&\rhd \cdots \overset{A_2}{\rhd}\cr
&\underset{B_2}{\lhd} \cdots \lhd &\underbrace{\lhd \cdots \lhd}_{l_2-l_1}
& \odot \cdots \odot
&\underbrace{\rhd \cdots \rhd}_{l_2-l_1}&\rhd \cdots \underset{A_2}{\rhd}
\end{array}
\right)_\rho, 
\]
whose representation is isomorphic to the one of
\[
\left(
\begin{array}{rrrcrrr}
\overset{B_1}{\lhd} \cdots &\lhd \odot \cdots &\cdots \odot
& \odot \cdots \odot & \overbrace{\odot \cdots \odot}^{l_2-l_1}
&\overset{A_2-l_1+1}{\rhd} \cdots \overset{A_2}{\rhd}\cr
&\underset{B_2}{\lhd} \cdots \lhd &\underbrace{\odot \cdots \odot}_{l_2-l_1}
&&&\underset{A_2-l_1+1}{\rhd} \cdots \underset{A_2}{\rhd}
\end{array}
\right)_\rho, 
\]
where the first $\odot$ in the second line is $(-1)^{A_2-B_1}\eta_1 = -(-1)^{A_2-A_1}\eta_2$. 
Hence, with $(l_1',l_2') = (l_1,l_1)$ and $(\eta_1', \eta_2') = (\eta_1, -(-1)^{A_2-A_1}\eta_2)$, 
we have $\pi(\EE) \cong \pi(\EE')$. 
This proves (3). 
\end{proof}

The same also holds in general. 
\begin{cor}\label{union2}
Let $\EE = \cup_\rho \{([A_i,B_i]_{\rho}, l_i, \eta_i)\}_{i \in (I_\rho,\succ)}$ be an extended multi-segment for $G_n$. 
Assume that two adjacent elements $k \succ k-1$ of $I_\rho$ satisfy one of the conditions (1), (2) or (3). 
Define $\EE'$ from $\EE$ by 
replacing $([A_k,B_k]_\rho, l_k, \eta_k)$ (\resp $([A_{k-1},B_{k-1}]_\rho, l_{k-1}, \eta_{k-1})$)
with $([A'_k,B'_k]_\rho, l'_k, \eta'_k)$ (\resp $([A'_{k-1},B'_{k-1}]_\rho, l'_{k-1}, \eta'_{k-1})$)
as in Theorem \ref{union}. 
Then $\pi(\EE) \cong \pi(\EE')$. 
\end{cor}
\begin{proof}
By Theorem \ref{nonzero2}, we may assume that $B_i \geq 0$ for all $i \in I_\rho$. 
Furthermore, we can assume that $B_i$ is big enough for any $i \in I_\rho$ with $i \succ k$. 
For positive integers $t_k > t_{k-1} \gg 0$, 
we define $\EE_{1}$ from $\EE$ by replacing 
$([A_k,B_k]_\rho, l_k, \eta_k)$ (\resp $([A_{k-1},B_{k-1}]_\rho, l_{k-1}, \eta_{k-1})$)
with 
$([A_k+t_k,B_k+t_k]_\rho, l_k, \eta_k)$ (\resp $([A_{k-1}+t_{k-1},B_{k-1}+t_{k-1}]_\rho, l_{k-1}, \eta_{k-1})$). 
Define $\EE'_1$ from $\EE'$ similarly. 
Then we may also assume that 
\begin{align*}
\pi(\EE) = \circ_{i \in \{k-1,k\}} \left(
D_{\rho|\cdot|^{B_i+1}, \dots, \rho|\cdot|^{A_i+1}} 
\circ \dots \circ 
D_{\rho|\cdot|^{B_i+t_i}, \dots, \rho|\cdot|^{A_i+t_i}}
\right)
(\pi(\EE_1)), \\
\pi(\EE') = \circ_{i \in \{k-1,k\}} \left(
D_{\rho|\cdot|^{B'_i+1}, \dots, \rho|\cdot|^{A'_i+1}} 
\circ \dots \circ 
D_{\rho|\cdot|^{B'_i+t_i}, \dots, \rho|\cdot|^{A'_i+t_i}}
\right)
(\pi(\EE'_1)).
\end{align*}
Consider
\begin{align*}
\pi_0 &= 
D_{\rho|\cdot|^{B_k+t_{k-1}+1}, \dots, \rho|\cdot|^{A_k+t_{k-1}+1}}
\circ \dots \circ
D_{\rho|\cdot|^{B_k+t_k}, \dots, \rho|\cdot|^{A_k+t_k}}  
(\pi(\EE_1)), \\ 
\pi_0' &= 
D_{\rho|\cdot|^{B'_k+t_{k-1}+1}, \dots, \rho|\cdot|^{A'_{k-1}+t_{k-1}+1}}
 \circ \dots \circ 
D_{\rho|\cdot|^{B'_k+t_k}, \dots, \rho|\cdot|^{A'_{k-1}+t_k}}
(\pi(\EE_1')). 
\end{align*}
If $\pi(\EE)$ (\resp $\pi(\EE')$) is nonzero, 
then $\pi_0$ (\resp $\pi_0'$) is also nonzero (and irreducible) by Theorem \ref{nonzero} 
(at least when $t_{k-1} \gg 0$).
Moreover, by (the same argument as) Theorem \ref{union}, we see that $\pi_0 \cong \pi_0'$ 
(so that both of them are nonzero). 
\par

Now assume that $\pi(\EE)$ is nonzero. 
Then by \cite[Proposition 8.5]{X2}, we see that 
\begin{align*}
\pi(\EE_1) \hookrightarrow 
\begin{pmatrix}
B_{k-1}+t_{k-1} & \ldots & A_{k-1}+t_{k-1} \\
\vdots & \ddots & \vdots \\
 B_{k-1}+1 & \ldots & A_{k-1}+1
\end{pmatrix}_\rho
\times 
\begin{pmatrix}
B_{k}+t_{k} & \ldots & A_{k}+t_{k} \\
\vdots & \ddots & \vdots \\
 B_{k}+1 & \ldots & A_{k}+1
\end{pmatrix}_\rho
\rtimes \pi(\EE).
\end{align*}
Since 
\begin{align*}
&
\begin{pmatrix}
B_{k}+t_{k} & \ldots & A_{k}+t_{k} \\
\vdots & \ddots & \vdots \\
 B_{k}+1 & \ldots & A_{k}+1
\end{pmatrix}_\rho
\\&\hookrightarrow 
\begin{pmatrix}
B_{k}+t_{k} & \ldots & A_{k}+t_{k} \\
\vdots & \ddots & \vdots \\
 B_{k}+t_{k-1}+1 & \ldots & A_{k}+t_{k-1}+1
\end{pmatrix}_\rho
\times 
\begin{pmatrix}
B_{k}+t_{k-1} & \ldots & A_{k}+t_{k-1} \\
\vdots & \ddots & \vdots \\
 B_{k}+1 & \ldots & A_{k}+1
\end{pmatrix}_\rho
\end{align*}
and since $\pi_0 \not= 0$, by \cite[Proposition 6.15]{LM}, 
the above embedding of $\pi(\EE_1)$ factors through 
\begin{align*}
&\pi(\EE_1) 
\hookrightarrow 
\\&
\soc\left(
\begin{pmatrix}
B_{k}+t_{k} & \ldots & A_{k}+t_{k} \\
\vdots & \ddots & \vdots \\
 B_{k}+t_{k-1}+1 & \ldots & A_{k}+t_{k-1}+1
\end{pmatrix}_\rho
\times
\begin{pmatrix}
B_{k-1}+t_{k-1} & \ldots & A_{k-1}+t_{k-1} \\
\vdots & \ddots & \vdots \\
 B_{k-1}+1 & \ldots & A_{k-1}+1
\end{pmatrix}_\rho
\right)
\\&\times 
\begin{pmatrix}
B_{k}+t_{k-1} & \ldots & A_{k}+t_{k-1} \\
\vdots & \ddots & \vdots \\
 B_{k}+1 & \ldots & A_{k}+1
\end{pmatrix}_\rho
\rtimes \pi(\EE).
\end{align*}
This implies that 
\[
\pi_0 
\hookrightarrow  
\begin{pmatrix}
B_{k-1}+t_{k-1} & \ldots & A_{k-1}+t_{k-1} \\
\vdots & \ddots & \vdots \\
 B_{k-1}+1 & \ldots & A_{k-1}+1
\end{pmatrix}_\rho
\times 
\begin{pmatrix}
B_{k}+t_{k-1} & \ldots & A_{k}+t_{k-1} \\
\vdots & \ddots & \vdots \\
 B_{k}+1 & \ldots & A_{k}+1
\end{pmatrix}_\rho
\rtimes \pi(\EE).
\]
Hence
\begin{align*}
\pi(\EE_1') 
&\hookrightarrow 
\begin{pmatrix}
B_k+t_k & \ldots & A_{k-1}+t_k \\
\vdots & \ddots & \vdots \\
B_k + t_{k-1}+1 & \ldots & A_{k-1} + t_{k-1} + 1
\end{pmatrix}_\rho
\rtimes \pi_0
\\&\hookrightarrow 
\begin{pmatrix}
B_k+t_k & \ldots & A_{k-1}+t_k \\
\vdots & \ddots & \vdots \\
B_k + t_{k-1}+1 & \ldots & A_{k-1} + t_{k-1} + 1
\end{pmatrix}_\rho
\times 
\begin{pmatrix}
B_{k-1}+t_{k-1} & \ldots & A_{k-1}+t_{k-1} \\
\vdots & \ddots & \vdots \\
 B_{k-1}+1 & \ldots & A_{k-1}+1
\end{pmatrix}_\rho
\\&\quad
\times 
\begin{pmatrix}
B_{k}+t_{k-1} & \ldots & A_{k}+t_{k-1} \\
\vdots & \ddots & \vdots \\
B_{k}+1 & \ldots & A_{k}+1
\end{pmatrix}_\rho
\rtimes \pi(\EE).
\end{align*}
Since $\pi_0' \not= 0$, 
on the first two Speh representations in the last inclusion, 
$\pi(\EE_1')$ factors though
\[
\soc\left(
\begin{pmatrix}
B_k+t_k & \ldots & A_{k-1}+t_k \\
\vdots & \ddots & \vdots \\
B_k + t_{k-1}+1 & \ldots & A_{k-1} + t_{k-1} + 1
\end{pmatrix}_\rho 
\times 
\begin{pmatrix}
B_{k-1}+t_{k-1} & \ldots & A_{k-1}+t_{k-1} \\
\vdots & \ddots & \vdots \\
 B_{k-1}+1 & \ldots & A_{k-1}+1
\end{pmatrix}_\rho
\right). 
\]
This is also a subrepresentation of 
\[
\begin{pmatrix}
B_{k-1}+t_{k-1} & \ldots & B_k+t_{k-1}-1 \\
\vdots & \ddots & \vdots \\
B_{k-1}+1 & \ldots & B_k
\end{pmatrix}_\rho
\times 
\begin{pmatrix}
B_k+t_k & \ldots & A_{k-1}+t_k \\
\vdots & \ddots & \vdots \\
B_k+1 & \ldots & A_{k-1}+1
\end{pmatrix}_\rho.
\]
Note that two shifted Speh representations
\[
\begin{pmatrix}
B_{k}+t_{k} & \ldots & A_{k-1}+t_{k} \\
\vdots & \ddots & \vdots \\
B_{k}+1 & \ldots & A_{k-1}+1
\end{pmatrix}_\rho
\quad \text{and} \quad
\begin{pmatrix}
B_{k}+t_{k-1} & \ldots & A_{k}+t_{k-1} \\
\vdots & \ddots & \vdots \\
B_{k}+1 & \ldots & A_{k}+1
\end{pmatrix}_\rho
\]
are commutative by Theorem \ref{speh}.
Therefore, 
\begin{align*}
\pi(\EE_1') 
&\hookrightarrow 
\begin{pmatrix}
B_{k-1}+t_{k-1} & \ldots & B_k+t_{k-1}-1 \\
\vdots & \ddots & \vdots \\
B_{k-1}+1 & \ldots & B_k
\end{pmatrix}_\rho
\times 
\begin{pmatrix}
B_{k}+t_{k-1} & \ldots & A_{k}+t_{k-1} \\
\vdots & \ddots & \vdots \\
B_{k}+1 & \ldots & A_{k}+1
\end{pmatrix}_\rho
\\&\quad
\times 
\begin{pmatrix}
B_k+t_k & \ldots & A_{k-1}+t_k \\
\vdots & \ddots & \vdots \\
B_k+1 & \ldots & A_{k-1}+1
\end{pmatrix}_\rho
\rtimes \pi(\EE).
\end{align*}
This together with \cite[Lemma 5.6]{X1} implies that 
\begin{align*}
\pi(\EE') 
&= \circ_{i \in \{k-1,k\}}\left(
D_{\rho|\cdot|^{B'_i+1}, \dots, \rho|\cdot|^{A'_i+1}} 
\circ \dots \circ 
D_{\rho|\cdot|^{B'_i+t_i}, \dots, \rho|\cdot|^{A'_i+t_i}}
\right) (\pi(\EE'_1))
= \pi(\EE).
\end{align*}
By the same argument, if $\pi(\EE') \not= 0$, 
then we can show that $\pi(\EE) \cong \pi(\EE')$. 
This completes the proof.
\end{proof}

The image of Theorem \ref{union} is determined as follows. 
\begin{prop}\label{image}
Let $\EE' = \{([A'_{k-1}, B'_{k-1}]_\rho, l'_{k-1}, \eta'_{k-1}), ([A'_k,B'_k]_\rho, l'_k, \eta'_k)\}$ with $\pi(\EE') \not= 0$. 
It is obtained from some $\EE$ by Theorem \ref{union} 
if and only if $A'_k < A'_{k-1}$, $B'_{k-1} < B'_k$ and one of the following holds: 
\begin{itemize}
\item
$l'_{k-1} = 0$ and $\#[A'_{k-1}, B'_{k-1}]_\rho > 1$; 
\item
$l'_{k-1} > 0$ and $\pi(\EE'^-) = 0$, 
where 
\[
\EE'^- = \{([A'_{k-1}-1, B'_{k-1}+1]_\rho, l'_{k-1}-1, \eta'_{k-1}), ([A'_k,B'_k]_\rho, l'_k, \eta'_k)\}.
\]
\end{itemize}
\end{prop}
\begin{proof}
The first case is clear.
Hence we may assume that $l'_{k-1} > 0$.
If $\EE'$ is the image of $\EE$, but if $\pi(\EE'^-) \not= 0$, 
we can apply \cite[Theorem 1.3]{X4} to $\EE$ and Theorem \ref{derivatives} (2) to $\EE'$. 
These theorems give a contradiction.
Conversely, if $\pi(\EE'^-) = 0$, by a case-by-case argument (or by the definition of $\pi(\EE')$), 
one can construct $\EE$ with $\EE \mapsto \EE'$.
\end{proof}

\subsection{Algorithms for derivatives}
Let $\EE = \cup_\rho \{([A_i,B_i]_{\rho}, l_i, \eta_i)\}_{i \in (I_\rho,\succ)}$ 
be an extended multi-segment for $G_n$ such that $\pi(\EE) \not= 0$. 
Now we give an algorithm to compute certain derivatives (or socles) of $\pi(\EE)$. 
\par

Recall that $I_\rho^{2,\adj}$ is the set of triples $(i,j,\succ')$, 
where $\succ'$ is an admissible order on $I_\rho$, 
and $i \succ' j$ are two adjacent elements in $I_\rho$ with respect to $\succ'$. 
\par

We have an algorithm to use Theorem \ref{derivatives} (1). 

\begin{alg}\label{alg+}
With the notation as above, 
we proceed the following:

\begin{description}
\item[Step 1]
Define $B^{\max} = \max\{B_i \;|\; i \in I_\rho\}$ and $I_\rho^{\max} = \{ i \in I_\rho \;|\; B_i = B^{\max}\}$. 
If $B^{\max} > 0$, set $i$ to be the minimal element in $I_\rho^{\max}$, and go to Step 2. 
Otherwise, the procedure is ended.

\item[Step 2]
If there exist an element $j \in I_\rho \setminus I_\rho^{\max}$ and an admissible order $\succ'$ such that 
\begin{itemize}
\item
$(i,j,\succ') \in I_\rho^{2,\adj}$; 
\item
$A_i > A_j$ and $B_i > B_j$; and 
\item
$\EE'$ is in the situation of \S \ref{sec.deform} with respect to $i \succ' j$, 
where $\EE' = \cup_\rho \{([A_{i'},B_{i'}]_{\rho}, l'_{i'}, \eta'_{i'})\}_{i' \in (I_\rho,\succ')}$ 
is such that $\pi(\EE) \cong \pi(\EE')$, 
which is obtained in \S \ref{sec.change},
\end{itemize}
then using Theorem \ref{union}, 
replace $\EE = \cup_\rho \{([A_i,B_i]_{\rho}, l_i, \eta_i)\}_{i \in (I_\rho,\succ)}$ so that 
\[
[A_i,B_i]_\rho \leadsto [A_i,B_i]_\rho \cap [A_j,B_j]_\rho, 
\quad
[A_j,B_j]_\rho \leadsto [A_i,B_i]_\rho \cup [A_j,B_j]_\rho, 
\]
and so that $\pi(\EE)$ does not change. 
If the new $[A_i,B_i]_\rho$ is empty, 
we understand that $I_\rho^{\max}$ is replaced with $I_\rho^{\max} \setminus \{i\}$. 
Go to Step 3. 

\item[Step 3]
When Theorem \ref{union} was applied in Step 2, go back to Step 1. 
Otherwise, 
\begin{itemize}
\item
if $i$ is the maximal element, go to Step 4;
\item
otherwise, replace $i \in I_\rho^{\max}$ with the next minimal element, and go back to Step 2.
\end{itemize}

\item[Step 4]
Using Theorem \ref{change}, change the admissible order on $I_\rho$ 
such that $i \succ i' \iff A_i > A_{i'}$ for any $i,i' \in I_\rho^{\max}$. 
If $B^{\max} = 1/2$, and if $l_i = 0$ and $\eta_i \not= (-1)^{\alpha_i}$ for the minimal element $i \in I_\rho^{\max}$, 
where $\alpha_i = \sum_{k \in I_\rho, k \prec i}(A_k+B_k+1)$, 
do the following: 
\begin{itemize}
\item
Formally replace $I_\rho$ with $I_\rho \sqcup \{i_0\}$ with $i_0$ minimal in $I_\rho$ and 
$([A_{i_0}, B_{i_0}], l_{i_0}, \eta_{i_0}) = ([-1/2,-1/2],0,+1)$; 

\item
using Theorem \ref{change}, change the admissible order on $I_\rho$ such that 
$j \prec i_0 \prec i$ for any $j \in I_\rho \setminus I_\rho^{\max}$; 

\item
replace $([A_i,B_i], l_i, \eta_i)$ (where $B_i=1/2$ and $l_i = 0$) with $([A_i,-1/2], 0, -\eta_i)$, 
and remove $i_0$ (\resp $i$) from $I_\rho$ (\resp $I_\rho^{\max}$); 

\item
go back to Step 4. 
\end{itemize}
Otherwise, the procedure is ended.

\end{description}

Let $\EE^* = \cup_\rho \{([A_i^*,B_i^*]_{\rho}, l_i^*, \eta_i^*)\}_{i \in (I_\rho^*,\succ)}$ 
be the resulting extended multi-segment so that $\pi(\EE) \cong \pi(\EE^*)$. 
Then $\EE^*$ satisfies the assumptions in Theorem \ref{derivatives} (1) 
if $\max\{B_i^* \;|\; i \in I_\rho^*\} > 0$. 
\end{alg}
\begin{proof}
First of all, by induction on 
\[
t = \sum_{i \in I_\rho^{\max}}(A_i-B_i+1), 
\]
we see that this algorithm stops for any $\EE$. 
\par

Since $\pi(\EE^*) \not= 0$, 
by Theorem \ref{nonzero}, 
we see that 
$\EE^*$ satisfies the three necessary conditions in Proposition \ref{nec} 
with respect to all $(i,j,\succ') \in I_\rho^{*,2,\adj}$. 
However, by the construction of $\EE^*$, 
we see that $\EE^+$ also satisfies the same conditions, 
where $\EE^* \mapsto \EE^+$ is given in Theorem \ref{derivatives} (1). 
Moreover, by Step 4, $\EE^+$ satisfies the condition ($\star$) in Theorem \ref{nonzero2}. 
Therefore $\pi(\EE^+) \not= 0$ by Theorems \ref{nonzero2} and \ref{nonzero}. 
\end{proof}

We also have an algorithm to use Theorem \ref{derivatives} (2). 
\begin{alg}\label{alg-}
With the notation as above, 
we proceed the following:

\begin{description}
\item[Step 1]
Define $B^{\min} = \min\{B_j \;|\; j \in I_\rho\}$ and $I_\rho^{\min} = \{ j \in I_\rho \;|\; B_j = B^{\min}\}$. 
If $B^{\min} < 0$, set $j$ to be the maximal element in $I_\rho^{\min}$, and go to Step 2. 
Otherwise, the procedure is ended.

\item[Step 2]
If there exist an element $i \in I_\rho \setminus I_\rho^{\min}$ and an admissible order $\succ'$ such that 
\begin{itemize}
\item
$(i,j,\succ') \in I_\rho^{2,\adj}$; 
\item
$A_j > A_i$ and $B_i > B_j$; and 
\item
$\EE'$ is in the image of Corollary \ref{union2} with respect to $i \succ' j$, 
where $\EE' = \cup_\rho \{([A_{i'},B_{i'}]_{\rho}, l'_{i'}, \eta'_{i'})\}_{i' \in (I_\rho,\succ')}$ 
is such that $\pi(\EE) \cong \pi(\EE')$, 
which is obtained in \S \ref{sec.change},
\end{itemize}
then using the converse of Theorem \ref{union}, 
replace $\EE = \cup_\rho \{([A_i,B_i]_{\rho}, l_i, \eta_i)\}_{i \in (I_\rho,\succ)}$ so that 
\[
[A_i,B_i]_\rho \leadsto [A_j,B_i]_\rho, 
\quad
[A_j,B_j]_\rho \leadsto [A_i,B_j]_\rho, 
\]
and so that $\pi(\EE)$ does not change. 
If the new $[A_j,B_j]_\rho$ satisfies that $A_j+B_j < 0$, 
we understand that $I_\rho^{\min}$ is replaced with $I_\rho^{\min} \setminus \{j\}$. 
Go to Step 3. 

\item[Step 3]
When the converse of Theorem \ref{union} was applied in Step 2, go back to Step 1. 
Otherwise, 
\begin{itemize}
\item
if $j$ is the minimal element, go to Step 4;
\item
otherwise, replace $j \in I_\rho^{\min}$ with the next maximal element, and go back to Step 2.
\end{itemize}

\item[Step 4]
Using Theorem \ref{change}, change the admissible order on $I_\rho$ 
such that $j \succ j' \iff A_j < A_{j'}$ for any $j,j' \in I_\rho^{\min}$. 
If $B^{\min} = -1/2$, and if $l_j = 0$ (so that $\eta_j = (-1)^{\alpha_j}$) for the maximal element $j \in I_\rho^{\min}$, 
do the following: 
\begin{itemize}
\item
Formally 
replace $([A_j,B_j],l_j,\eta_j)$ (where $B_j = -1/2$ and $l_j = 0$) with $([A_j,1/2], 0, -\eta_j)$, 
and set $([A_{j_0}, B_{j_0}], l_{j_0}, \eta_{j_0}) = ([-1/2,-1/2],0,\eta_j)$; 

\item
replace $I_\rho$ (\resp $I_\rho^{\min}$) 
with $I_\rho \sqcup \{j_0\}$ (\resp $I_\rho^{\min} \cup\{j_0\} \setminus\{j\}$)  
where $j_0$ the maximal element of the new $I_\rho^{\min}$; 

\item
using Theorem \ref{change}, change the admissible order on $I_\rho$ such that 
$j_0$ is the minimal element in $I_\rho$
(in which case $([A_{j_0}, B_{j_0}], l_{j_0}, \eta_{j_0}) = ([-1/2,-1/2],0,+1)$); 

\item
remove $j_0$ from $I_\rho$; 

\item
go back to Step 4. 
\end{itemize}
Otherwise, the procedure is ended.

\end{description}

Let $\EE^* = \cup_\rho \{([A_i^*,B_i^*]_{\rho}, l_i^*, \eta_i^*)\}_{i \in (I_\rho^*,\succ)}$ 
be the resulting extended multi-segment so that $\pi(\EE) \cong \pi(\EE^*)$. 
Then $\EE^*$ satisfies the assumptions in Theorem \ref{derivatives} (2) 
if $\min\{B_i^* \;|\; i \in I_\rho^*\} < 0$. 
\end{alg}

Using Algorithms \ref{alg+} and \ref{alg-}, 
we can compute derivatives of $\pi(\EE)$ in general
if $D_{\rho|\cdot|^x}(\pi(\EE)) \not= 0$ for some $\rho \in \Cusp^\bot$ and $x \not= 0$.
Otherwise, by applying both Algorithms \ref{alg+} and \ref{alg-}, 
we have 
$\EE^* = \cup_\rho \{([A_i^*,B_i^*]_{\rho}, l_i^*, \eta_i^*)\}_{i \in (I_\rho^*,\succ)}$ 
with $\pi(\EE) \cong \pi(\EE^*)$
such that $B_i^* = 0$ for all $\rho$ and $i \in I_\rho$.
In this case, $\pi(\EE^*)$ is explicitly given by \cite[Theorem 1.3]{X4}.

\subsection{Examples for derivatives}
In this subsection, we set $\rho = \1_{\GL_1(F)}$ and we drop $\rho$ from the notation. 
\par

First, let us consider $\EE = \EE_\rho$ with 
\[
\EE = 
\bordermatrix{
& -\half{5} & -\half{3} & -\half{1} & \half{1} & \half{3} & \half{5}\cr
& \lhd & \lhd & \oplus & \ominus & \rhd & \rhd \cr
&  &  & \lhd & \rhd &  &  \cr
&  &  &  &  & \ominus &  \cr
}.
\]
This is $\EE_3$ in Example \ref{ex-super}.
We compute some derivatives using Algorithm \ref{alg-}.
By Theorems \ref{nonzero2}, \ref{nonzero} and \ref{derivatives} (2), 
we have 
\[
D_{|\cdot|^{-\half{5}}}(\pi(\EE)) 
\cong \pi\bordermatrix{
& -\half{3} & -\half{1} & \half{1} & \half{3} \cr
& \lhd & \oplus & \ominus & \rhd \cr
&  & \lhd & \rhd & \cr
&  &  &  & \ominus \cr
}.
\]
By using the converse of Theorem \ref{union} for the first and second lines, 
we have
\[
D_{|\cdot|^{-\half{5}}}(\pi(\EE)) 
\cong \pi\bordermatrix{
& -\half{3} & -\half{1} & \half{1} & \half{3} \cr
& \lhd & \oplus & \rhd & \cr
&  & \lhd & \ominus & \rhd \cr
&  &  &  & \ominus \cr
}
\cong
\pi\bordermatrix{
& -\half{1} & \half{1} & \half{3} \cr
& \lhd & \ominus & \rhd \cr
&  &  & \ominus \cr
}.
\]
Then by Theorem \ref{derivatives} (2), we have
\[
D_{|\cdot|^{-\half{3}}} \circ D_{|\cdot|^{-\half{1}}} \circ D_{|\cdot|^{-\half{5}}}(\pi(\EE)) 
= 
\pi\bordermatrix{
& \half{1} & \half{3} \cr
& \ominus & \cr
&  & \ominus \cr
} = \pi((1/2)^-,(3/2)^-).
\]
By taking socles, we conclude that
\[
\pi(\EE) \cong L(|\cdot|^{-\half{5}}, \Delta[-1/2,-3/2]; \pi((1/2)^-,(3/2)^-)), 
\]
which is consistent with Example \ref{ex-super}. 
\par

Next, let us consider $\EE = \EE_\rho$ with 
\[
\EE = 
\bordermatrix{
& -1/2 & 1/2 & 3/2 & 5/2 \cr
& \oplus & \ominus & & \cr
& & \lhd & \rhd & \cr
& & \lhd & \oplus & \rhd \cr
& & & \oplus & \ominus \cr
}. 
\]
Note that $\pi(\EE)$ is a representation of $\SO_{31}(F)$. 
We use Algorithm \ref{alg+} mainly and Algorithm \ref{alg-} secondarily. 
Using Theorems \ref{change}, \ref{union} and \ref{derivatives}, we have
\begin{align*}
\pi(\EE) &\cong 
\pi
\bordermatrix{
& -1/2 & 1/2 & 3/2 & 5/2 \cr
&\oplus & \ominus & & \cr
&& \lhd & \rhd & \cr
&& & \oplus & \ominus \cr
&& \ominus & \oplus & \ominus 
}
\cong
\pi
\bordermatrix{
& -1/2 & 1/2 & 3/2 & 5/2 \cr
&\oplus & \ominus & & \cr
&& \lhd & \oplus & \rhd \cr
&& & \ominus & \cr
&& \ominus & \oplus & \ominus 
}
\\&\xmapsto{D_{|\cdot|^{3/2}}}
\pi
\bordermatrix{
& -1/2 & 1/2 & 3/2 & 5/2 \cr
&\oplus & \ominus & & \cr
&& \lhd & \oplus & \rhd \cr
&& \ominus & & \cr
&& \ominus & \oplus & \ominus 
}
\cong 
\pi
\bordermatrix{
& -1/2 & 1/2 & 3/2 & 5/2 \cr
&\oplus & \ominus & \oplus & \ominus\cr
&& \ominus & & \cr
&& \ominus & & \cr
&& \ominus & \oplus & \ominus
}
\\&\xmapsto{D_{|\cdot|^{5/2}} \circ D_{|\cdot|^{3/2}} \circ D_{|\cdot|^{1/2}}^{(3)}}
\pi
\bordermatrix{
& -1/2 & 1/2 & 3/2 & 5/2 \cr
& \oplus & \ominus & \oplus & \ominus \cr
& \ominus & & \cr
& \ominus & & \cr
& \ominus & \oplus & \ominus \cr
}
\\&\cong 
\pi
\bordermatrix{
& -1/2 & 1/2 & 3/2 & 5/2 \cr
& \oplus & & \cr
& \oplus & & \cr
& \oplus & \ominus & \oplus \cr
& \lhd & \oplus & \ominus & \rhd \cr
}
\cong 
\pi
\bordermatrix{
& -1/2 & 1/2 & 3/2 & 5/2 \cr
& & \ominus & \oplus \cr
& \lhd & \oplus & \ominus & \rhd \cr
}
\\&
\xmapsto{D_{|\cdot|^{-5/2}} \circ D_{|\cdot|^{-3/2}} \circ D_{|\cdot|^{-1/2}}}
\pi
\bordermatrix{
& 1/2 & 3/2  \cr
& \ominus & \oplus  \cr
& \oplus & \ominus \cr
}
\cong 
\pi
\bordermatrix{
& -1/2 & 1/2 & 3/2 \cr
& \oplus & \ominus & \oplus \cr
& & \oplus & \ominus \cr
}
\\&\xmapsto{D_{|\cdot|^{3/2}} \circ D_{|\cdot|^{1/2}}}
\pi
\bordermatrix{
& -1/2 & 1/2 & 3/2 \cr
& \oplus & \ominus & \oplus \cr
& \oplus & \ominus \cr
}
\cong 
\pi
\bordermatrix{
& -1/2 & 1/2 & 3/2 \cr
& \lhd & \ominus & \rhd \cr
& & \ominus \cr
}
\\&\xmapsto{D_{|\cdot|^{-3/2}} \circ D_{|\cdot|^{-1/2}}}
\pi
\bordermatrix{
& 1/2 \cr
& \ominus \cr
& \ominus \cr
}
\cong 
\pi
\bordermatrix{
& -1/2 & 1/2 \cr
& \oplus & \ominus \cr
& & \ominus \cr
}
\\&\xmapsto{D_{|\cdot|^{1/2}}}
\pi
\bordermatrix{
& -1/2 & 1/2 \cr
& \oplus & \ominus \cr
& \ominus \cr
}
\cong 
\pi
\bordermatrix{
& -1/2 & 1/2 \cr
& \lhd & \rhd \cr
}
\xmapsto{D_{|\cdot|^{-1/2}}}
\1_{\SO_1(F)}.
\end{align*}

\section{A formula for the Aubert duality}\label{s.aubert}
In \cite{Au}, Aubert defined an involution on $\Irr(G_n)$, 
which is a generalization of the Zelevinsky involution given in \cite{Z}. 
In this section, we prove a formula for the Aubert dual of $\pi(\EE)$. 

\subsection{Definition and algorithm}
Aubert \cite{Au} showed that 
for any irreducible representation $\pi$ of $G_n$, 
there exists a sign $\epsilon \in \{\pm1\}$ such that 
the virtual representation 
\[
\hat\pi = \epsilon \sum_{P} (-1)^{\dim A_M}[\Ind_P^{G_n}(\Jac_P(\pi))]
\]
is again an irreducible representation, 
where $P = MN$ runs over all standard parabolic subgroups of $G_n$, 
and $A_M$ is the maximal split torus of the center of $M$. 
We call $\hat\pi$ the \emph{Aubert dual} of $\pi$. 
\par

In the previous work with M{\'i}nguez, 
we gave an algorithm \cite[Algorithm 4.1]{AM} to compute $\hat\pi$ for any $\pi \in \Irr(G_n)$. 
It says that by using the compatibilities of the Aubert duality with derivatives
\[
\begin{aligned}
\hat\pi &\cong S_{\rho|\cdot|^{-x}}^{(k)} \left( D_{\rho|\cdot|^{x}}^{(k)}(\pi)^{\widehat{\;}} \right)
&&\text{for $x \not= 0$; and} \\
\hat\pi &\cong S_{Z_\rho[0,1]}^{(k)} \left( D_{\Delta_\rho[0,-1]}^{(k)}(\pi)^{\widehat{\;}} \right)
&&\text{if $\pi$ is $\rho|\cdot|^{-1}$-reduced}, 
\end{aligned}
\]
the computation of $\hat\pi$ can be reduced to 
the one of the Aubert dual of an easier representation (\cite[Proposition 5.4]{AM}). 

\subsection{Statement}\label{s.conj}
It is expected that the Aubert duality would preserve the unitarity. 
For $\psi = \oplus_\rho (\oplus_{i \in I_\rho} \rho \boxtimes S_{a_i} \boxtimes S_{b_i}) \in \Psi_\gp(G_n)$, 
define $\hat\psi = \oplus_\rho (\oplus_{i \in I_\rho} \rho \boxtimes S_{b_i} \boxtimes S_{a_i})$.
By the compatibility of the Aubert duality with parabolic inductions, and 
by the result on the Zelevinsky duals of Speh (ladder) representations (\cite[\S 4.1]{BLM}), 
we see that the Zelevinsky dual of $\tau_{\psi}$ is isomorphic to $\tau_{\hat\psi}$.
In particular, by the compatibility of twisted endoscopic character identities and the Aubert duality 
(\cite[\S A]{X2}), we have 
\[
\{\hat\pi \;|\; \pi \in \Pi_\psi\} = \Pi_{\hat\psi}. 
\]
This equation would also follow from M{\oe}glin's original construction of $\Pi_\psi$ (see \S \ref{s.moeglin} below). 
\par

Now we define an involution $\EE \mapsto \hat\EE$ on the set of extended multi-segments for $G_n$. 
\begin{defi}\label{hatE}
Let $\EE = \cup_\rho \{([A_i,B_i]_{\rho}, l_i, \eta_i)\}_{i \in (I_\rho,\succ)}$ be an extended multi-segment for $G_n$. 
We assume that the fixed admissible order $\succ$ on $I_\rho$ satisfies $(\PP')$, i.e., $B_i > B_j \implies i \succ j$. 
We define 
\[
\hat\EE = \cup_\rho \{([A_i,-B_i]_{\rho}, \hat{l}_i, \hat{\eta}_i)\}_{i \in (I_\rho,\hat{\succ})}
\]
as follows:
\begin{itemize}
\item
The total order $\hat{\succ}$ on $I_\rho$ is given by $i \hat{\succ} j \iff i \prec j$. 

\item
The pair $(\hat{l}_i, \hat{\eta}_i)$ is given by 
\begin{align*}
\hat{l}_i &= 
\left\{
\begin{aligned}
&l_i + B_i \iif B_i \in \Z, \\
&l_i+B_i+\half{1}(-1)^{\alpha_i}\eta_i \iif B_i \not\in \Z,
\end{aligned}
\right. 
\\
\hat\eta_i &= 
\left\{
\begin{aligned}
&(-1)^{\alpha_i+\beta_i}\eta_i \iif B_i \in \Z, \\
&(-1)^{\alpha_i+\beta_i+1}\eta_i \iif B_i \not\in \Z,
\end{aligned}
\right. 
\end{align*}
where 
\[
\alpha_i = \sum_{j \in I_\rho, j \prec i}a_j, \quad
\beta_i = \sum_{j \in I_\rho, j \succ i}b_j
\]
with $a_j = A_j+B_j+1$ and $b_j = A_j-B_j+1$. 
Here, when $B_i \not\in \Z$ and $l_i = b_i/2$, we regard $\eta_i = (-1)^{\alpha_i+1}$.
\end{itemize}
\end{defi}

Note that $\hat{\succ}$ is an admissible order on $I_\rho$ since we assume that $\succ$ satisfies $(\PP')$. 
Also, it is easy to check that the correspondence $\EE \mapsto \hat\EE$ is an involution. 

\begin{thm}\label{aubert}
If $\pi(\EE) \not= 0$, then its Aubert dual is given by 
\[
\hat\pi(\EE) \cong \pi(\hat\EE).
\]
\end{thm}

\begin{ex}
As in \cite[\S 4.1]{J-dual}, 
let us consider $\pi = L(\Delta_\rho[1/2,-5/2], \rho|\cdot|^{-1/2}; \pi((1/2)^-,(1/2)^-))$. 
It is not of Arthur type, but we can write
\[
D_{\rho|\cdot|^{5/2}}(\pi) \cong \pi
\bordermatrix{
& -1/2 & 1/2 & 3/2 \cr
& \lhd & \rhd &\cr
& & \ominus & \cr
& & \ominus & \cr
& & \lhd & \rhd \cr
}_\rho.
\]
Here, we regard $\eta_1 = -1$ and $\eta_4 = +1$ by Definition \ref{hatE}. 
By Theorem \ref{aubert}, its Aubert dual is given by 
\begin{align*}
D_{\rho|\cdot|^{5/2}}(\pi)^{\widehat{\;}} 
&\cong \pi
\bordermatrix{
& -1/2 & 1/2 & 3/2 \cr
& \lhd & \oplus & \rhd \cr
& \lhd & \rhd & \cr
& \lhd & \rhd & \cr
& & \oplus & \cr
}_\rho
\\&\cong L\left(\Delta_\rho[-1/2,-3/2], (\rho|\cdot|^{-\half{1}})^2; \pi((1/2)^+, (1/2)^+)\right).
\end{align*}
Therefore, 
\[
\hat\pi \cong S_{\rho|\cdot|^{-5/2}}\left( D_{\rho|\cdot|^{5/2}}(\pi)^{\widehat{\;}} \right)
\cong 
L\left(\rho|\cdot|^{-\half{5}}, \Delta_\rho[-1/2,-3/2], (\rho|\cdot|^{-\half{1}})^2; \pi((1/2)^+, (1/2)^+)\right).
\]
This is consistent with the conclusion of \cite[\S 4.1]{J-dual}.
\end{ex}

In the rest of this section, we prove Theorem \ref{aubert}. 
A key idea to prove this theorem is to compare the original version of M{\oe}glin's construction of $\Pi_\psi$, 
and use Xu's reduction operator ``Change sign'' (\cite[Propositions 7.5, 7.6]{X3}).
The author is grateful to Bin Xu for suggesting this operator. 

\subsection{Review of M{\oe}glin's original construction}\label{s.moeglin}
In section \ref{s.extended}, we have constructed $A$-packets $\Pi_\psi$ using extended multi-segments. 
This construction is a refinement of M{\oe}glin's construction \cite{Moe06a, Moe06b, Moe09, Moe10, Moe11c}. 
The original version is much more difficult because it includes the (partial) Aubert duality (see \cite[\S 6]{X2}). 
However, for this reason, 
we can prove Theorem \ref{aubert} by comparing our construction with M{\oe}glin's original one. 
\par

In this subsection, we review M{\oe}glin's original construction.
Let $\psi = \oplus_{\rho}(\oplus_{i \in I_\rho} \rho \boxtimes S_{a_i} \boxtimes S_{b_i})$
be an $A$-parameter for $G_n$ of good parity. 
For $i \in I_\rho$, write $d_i = \min\{a_i,b_i\}$, 
and choose $\zeta_i \in \{\pm1\}$ such that $\zeta_i(a_i-b_i) \geq 0$.
We fix a total order $\succ$ on $I_\rho$ satisfying that
\[
\text{
$a_i+b_i > a_j+b_j$, $|a_i-b_i| > |a_j-b_j|$ and $\zeta_i = \zeta_j$
}
\implies i \succ j.
\]
For $i \in I_\rho$, take
\begin{itemize}
\item
$l_i \in \Z$ with $0 \leq l_i \leq d_i/2$; 
\item
$\eta_i \in \{\pm1\}$ such that 
\[
\prod_{\rho}\prod_{i \in I_\rho}\eta_i^{d_i}(-1)^{[d_i/2]+l_i} = 1.
\]
\end{itemize}
Note that all of these data depend on $\rho$. 
For these data $\ub{\zeta} = (\zeta_i)_{\rho,i}$, $\ub{l} = (l_i)_{\rho,i}$ and $\ub\eta = (\eta_i)_{\rho,i}$, 
M{\oe}glin constructed a representation $\pi^M_\succ(\psi, \ub{\zeta}, \ub{l}, \ub{\eta})$ of $G_n$. 

\begin{prop}\label{moeglin}
\begin{enumerate}
\item
The representation $\pi^M_\succ(\psi, \ub{\zeta}, \ub{l}, \ub{\eta})$ is irreducible or zero. 
\item
The $A$-packet $\Pi_\psi$ is given as 
\[
\Pi_\psi = \{\pi^M_\succ(\psi, \ub{\zeta}, \ub{l}, \ub{\eta}) \;|\; 
\text{$\ub{l} = (l_i)_{\rho,i}$ and $\ub\eta = (\eta_i)_{\rho,i}$ are as above}
\} \setminus \{0\}.
\]
\item
When $\zeta_i = +1$ (so that $a_i \geq b_i$) for all $\rho$ and $i \in I_\rho$, 
we have $\pi^M_\succ(\psi, \ub{\zeta}, \ub{l}, \ub{\eta}) = \pi(\EE)$, 
where $\EE = \cup_\rho \{([(a_i+b_i)/2-1,(a_i-b_i)/2]_{\rho}, l_i, \eta_i)\}_{i \in (I_\rho,\succ)}$. 
\item
Take non-negative integers $t_i$ for $i \in I_\rho$, and define 
$\psi_\succ = \oplus_{\rho}(\oplus_{i \in I_\rho} \rho \boxtimes S_{a'_i} \boxtimes S_{b'_i})$
such that 
\begin{itemize}
\item
$(a'_i,b'_i) = (a_i+2t_i,b_i)$ or $(a'_i,b'_i) = (a_i,b_i+2t_i)$ with $\zeta_i(a'_i-b'_i) \geq 0$; 
\item
$\psi_\succ \circ \Delta$ is multiplicity-free, 
where $\Delta \colon W_F \times \SL_2(\C) \rightarrow W_F \times \SL_2(\C) \times \SL_2(\C)$
is the diagonal map $\Delta(w,\alpha) = (w,\alpha,\alpha)$. 
\end{itemize}
Then 
\[
\pi^M_\succ(\psi, \ub{\zeta}, \ub{l}, \ub{\eta})
= 
\circ_\rho \circ_{i \in I_\rho}
\left(
D_{\rho|\cdot|^{\zeta_i(B_i+1)}, \dots, \rho|\cdot|^{\zeta_i(A_i+1)}}
\circ \dots \circ 
D_{\rho|\cdot|^{\zeta_i(B_i+t_i)}, \dots, \rho|\cdot|^{\zeta_i(A_i+t_i)}}
\right)
(\pi^M_\succ(\psi_\succ, \ub{\zeta}, \ub{l}, \ub{\eta}))
\]
with $A_i = (a_i+b_i)/2-1$ and $B_i = |a_i-b_i|/2$.
\item
The Aubert dual of $\pi^M_\succ(\psi, \ub{\zeta}, \ub{l}, \ub{\eta})$ is give as 
\[
\hat\pi^M_\succ(\psi, \ub{\zeta}, \ub{l}, \ub{\eta})
\cong \pi^M_\succ(\hat\psi, -\ub{\zeta}, \ub{l}, \ub{\eta}), 
\]
where $\hat\psi = \oplus_{\rho}(\oplus_{i \in I_\rho} \rho \boxtimes S_{b_i} \boxtimes S_{a_i})$
and $-\ub{\zeta} = (-\zeta_i)_{\rho,i}$.
\end{enumerate}
\end{prop}
For more detail including a proof, see \cite[\S 6--8]{X2}.
\par

\subsection{Proof of Theorem \ref{aubert}}
To prove Theorem \ref{aubert}, we will use the following property. 
In this lemma, for simplicity, we only consider $A$-parameters $\psi$ such that $\psi|W_F$ is isotypic. 
\begin{lem}\label{sgn}
Let $\psi = \oplus_{i \in I_\rho} \rho \boxtimes S_{a_i} \boxtimes S_{b_i}$
be an $A$-parameter for $G_n$ of good parity. 
Set $A_i = (a_i+b_i)/2-1$ and $B_i = |a_i-b_i|/2$.
Fix a total order $\succ$ on $I_\rho$ as in the previous subsection.
We denote the minimal element of $I_\rho$ by $1$. 
Suppose that $\zeta_1 = +1$ (so that $a_1\geq b_1$), and that $B_i \gg 0$ for any $i \not= 1$. 
Then 
\[
\pi^M_\succ(\psi, \ub{\zeta}, \ub{l}, \ub{\eta}) \cong
D_{\rho|\cdot|^{B_1-1}, \dots, \rho|\cdot|^{-A_1-1}}
\circ \dots \circ
D_{\rho|\cdot|^{-\delta}, \dots, \rho|\cdot|^{-(A_1+B_1+\delta)}}
\left( \pi^M_\succ(\psi', \ub{\zeta'}, \ub{l'}, \ub{\eta'}) \right),
\]
where 
\[
\psi' = \psi - \rho \boxtimes S_{a_1} \boxtimes S_{b_1} + \rho \boxtimes S_{a_1} \boxtimes S_{a_1+2\delta}
\]
with $\delta \in \{0,1/2\}$ such that $B_1-\delta \in \Z$, 
and 
\[
(\zeta'_i, l'_i, \eta'_i) = \left\{
\begin{aligned}
&(-1, l_1+B_1, \eta_1) \iif i=1, B_1 \in \Z, \\
&(-1, l_1+B_1+\half{1}\eta_1, -\eta_1) \iif i=1, B_1 \not\in \Z, \\
&(\zeta_i, l_i, \eta_i) \iif i \not= 1. 
\end{aligned}
\right. 
\]
Here, when $B_1 \not\in \Z$ and $l_1 = b_1/2$, we regard $\eta_1 = -1$.
\end{lem}
\begin{proof}
This is an extension of \cite[Propositions 7.5, 7.6]{X3}.
In fact, if $0 \leq a_1-b_1 \leq 1$, the assertion follows from these propositions. 
\par

In general, since $B_i \gg 0$ for $i \not= 1$, we have 
\[
D_{\rho|\cdot|^{\delta+1}, \dots, \rho|\cdot|^{A_1-B_1+\delta+1}}
\circ \dots \circ
D_{\rho|\cdot|^{B_1}, \dots, \rho|\cdot|^{A_1}}
\left(\pi^M_\succ(\psi, \ub{\zeta}, \ub{l}, \ub{\eta})\right)
\cong \pi^M_\succ(\psi_0, \ub{\zeta}, \ub{l}, \ub{\eta}), 
\]
where 
\[
\psi_0 = \psi - \rho \boxtimes S_{a_1} \boxtimes S_{b_1} + \rho \boxtimes S_{b_1+2\delta} \boxtimes S_{b_1}.
\]
We can apply \cite[Propositions 7.5, 7.6]{X3} to $\pi^M_\succ(\psi_0, \ub{\zeta}, \ub{l}, \ub{\eta})$, 
and we obtain that
\[
\pi^M_\succ(\psi_0, \ub{\zeta}, \ub{l}, \ub{\eta})
\cong 
\left\{
\begin{aligned}
&\pi^M_\succ(\psi'_0, \ub{\zeta'}, \ub{l}, \ub{\eta}) \iif \delta = 0, \\
&D_{\rho|\cdot|^{-1/2}, \dots, \rho|\cdot|^{-(A_1+B_1+1/2)}}
\left( \pi^M_\succ(\psi'_0, \ub{\zeta'}, \ub{l''}, \ub{\eta''}) \right)
\iif \delta = 1/2, 
\end{aligned}
\right. 
\]
where $(\psi_0, \ub{\zeta}, \ub{l}, \ub{\eta}) \mapsto (\psi_0', \ub{\zeta'}, \ub{l''}, \ub{\eta''})$ is as in the assertion for $\psi_0$.
Since $B_i \gg 0$, by M{\oe}glin's construction, 
\begin{align*}
&D_{\rho|\cdot|^{1}, \dots, \rho|\cdot|^{A_1-B_1+1}}
\circ \dots \circ
D_{\rho|\cdot|^{B_1}, \dots, \rho|\cdot|^{A_1}}
\\& \circ
D_{\rho|\cdot|^{B_1-1}, \dots, \rho|\cdot|^{-A_1-1}}
\circ \dots \circ
D_{\rho|\cdot|^{0}, \dots, \rho|\cdot|^{-(A_1+B_1)}}
\left( \pi^M_\succ(\psi', \ub{\zeta'}, \ub{l'}, \ub{\eta'}) \right)
\\&\cong 
\pi^M_\succ(\psi'_0, \ub{\zeta'}, \ub{l}, \ub{\eta})
\end{align*}
if $\delta= 0$, and 
\begin{align*}
&D_{\rho|\cdot|^{3/2}, \dots, \rho|\cdot|^{A_1-B_1+3/2}}
\circ \dots \circ
D_{\rho|\cdot|^{B_1}, \dots, \rho|\cdot|^{A_1}}
\\& \circ
D_{\rho|\cdot|^{B_1-1}, \dots, \rho|\cdot|^{-A_1-1}}
\circ \dots \circ
D_{\rho|\cdot|^{-1/2}, \dots, \rho|\cdot|^{-(A_1+B_1+1/2)}}
\left( \pi^M_\succ(\psi', \ub{\zeta'}, \ub{l'}, \ub{\eta'}) \right)
\\&\cong 
D_{\rho|\cdot|^{-1/2}, \dots, \rho|\cdot|^{-(A_1+B_1+1/2)}}
\left( \pi^M_\succ(\psi'_0, \ub{\zeta'}, \ub{l''}, \ub{\eta''}) \right)
\end{align*}
if $\delta = 1/2$.
Hence we obtain the assertion.
\end{proof}

Now we prove Theorem \ref{aubert}. 

\begin{proof}[Proof of Theorem \ref{aubert}]
To simplify the notation, we assume that only one $\rho$ appears. 
Namely, let $\EE = \{([A_i,B_i]_{\rho}, l_i, \eta_i)\}_{i \in (I_\rho,\succ)}$ be an extended multi-segment for $G_n$
with admissible order $\succ$ on $I_\rho$ satisfying $(\PP')$. 
Suppose that $\pi(\EE) \not= 0$.
\par

First, we assume that $B_i \geq 0$ for any $i \in I_\rho$. 
Write $I_\rho = \{1, \dots, m\}$ with $1 \prec \dots \prec m$ so that $B_1 \leq \dots \leq B_m$.
Take $\delta \in \{0,1/2\}$ such that $B_i-\delta \in \Z$.
For $0 \leq k \leq m$ and for positive integers $d_1, \dots, d_m$, we define 
\begin{itemize}
\item
an $A$-parameter $\psi^{(k)}_{(d_1, \dots, d_m)}$ by 
\[
\psi^{(k)}_{(d_1, \dots, d_m)} = 
\left(\bigoplus_{i \leq k} \rho \boxtimes S_{a_i} \boxtimes S_{a_i+2d_i+2\delta}\right)
\oplus
\left(\bigoplus_{i > k} \rho \boxtimes S_{a_i+2d_i} \boxtimes S_{b_i}\right); 
\]

\item
operators $T^{(i)}_{d_i}$, $S^{(i)}_{d_i}$ and $U^{(i)}$ by 
\begin{align*}
T^{(i)}_{d_i} &= 
D_{\rho|\cdot|^{B_i+1}, \dots, \rho|\cdot|^{A_i+1}}
\circ \dots \circ
D_{\rho|\cdot|^{B_i+d_i}, \dots, \rho|\cdot|^{A_i+d_i}}, \\
S^{(i)}_{d_i} &= 
D_{\rho|\cdot|^{-(1+\delta)}, \dots, \rho|\cdot|^{-(a_i+\delta)}}
\circ \dots \circ
D_{\rho|\cdot|^{-(d_i+\delta)}, \dots, \rho|\cdot|^{-(a_i+d_i+\delta-1)}}, \\
U^{(i)} &=
D_{\rho|\cdot|^{B_i-1}, \dots, \rho|\cdot|^{-A_i-1}}
\circ \dots \circ
D_{\rho|\cdot|^{-\delta}, \dots, \rho|\cdot|^{-(A_i+B_i+\delta)}}; 
\end{align*}

\item
triples $(\zeta_i^{(k)}, l_i^{(k)}, \eta_i^{(k)})$ by 
\[
(\zeta_i^{(k)}, l_i^{(k)}, \eta_i^{(k)}) = 
\left\{
\begin{aligned}
&(+1, l_i, (-1)^{a_1+\dots+a_k}\eta_i) \iif i > k, \\
&(-1, \hat{l}_i, \hat{\eta}_i) \iif i \leq k;
\end{aligned}
\right. 
\]

\item
an order $\succ_k$ on $I_\rho$ by
\[
\underbrace{k+1 \prec_k k+2 \prec_k \dots \prec_k m}_{m-k} 
\prec_k \underbrace{k \prec_k \dots \prec_k 2 \prec_k 1}_k.
\]
\end{itemize}
We claim that 
for any $0 \leq k \leq m$ and for any positive integers $d_1, \dots, d_m$ with 
\[
\underbrace{B_{k+1}+d_{k+1} \ll \dots \ll B_m+d_m}_{m-k} 
\ll 
\underbrace{d_k+\delta \ll \dots \ll d_1+\delta}_k,
\]
we have
\begin{align*}
\pi(\EE) \cong 
&(U^{(1)} \circ S^{(1)}_{d_1}) \circ \dots \circ (U^{(k)} \circ S^{(k)}_{d_k})
\\&
\circ T^{(m)}_{d_m} \circ \dots \circ T^{(k+1)}_{d_{k+1}}
\left(\pi^M_{\succ_k}(\psi^{(k)}_{(d_1, \dots, d_m)}, \ub\zeta^{(k)}, \ub{l}^{(k)}, \ub{\eta}^{(k)})\right).
\end{align*}
\par

We prove the claim by induction on $k$. 
When $k=0$, the claim says that 
\[
\pi(\EE) \cong T^{(m)}_{d_m} \circ \dots \circ T^{(1)}_{d_{1}}
\left(\pi^M_{\succ}(\psi^{(0)}_{(d_1, \dots, d_m)}, (+,\dots,+), \ub{l}, \ub{\eta})\right). 
\]
By definition, this equation holds for $d_1 \ll \dots \ll d_m$. 
\par

Suppose that we know the claim for $k < m$. 
Since $k+1$ is the minimal element of $I_\rho$ with respect to $\succ_k$, 
and since we have $B_{k+1}+d_{k+1} \ll \dots \ll B_m+d_m \ll d_k+\delta \ll \dots \ll d_1+\delta$,
we may replace $d_{k+1} = 0$. 
In this case, by Lemma \ref{sgn}, we have 
\[
\pi^M_{\succ_k}(\psi^{(k)}_{(d_1, \dots, d_m)}, \ub\zeta^{(k)}, \ub{l}^{(k)}, \ub{\eta}^{(k)})
\cong U^{(k+1)}\left(
\pi^M_{\succ_k}(\psi^{(k+1)}_{(d_1, \dots, d_m)}, \ub\zeta^{(k+1)}, \ub{l}^{(k+1)}, \ub{\eta'})
\right)
\]
with $\eta'_{k+1} = (-1)^{2\delta}\eta_{k+1}^{(k)}$ and $\eta'_{i} = \eta_{i}^{(k)}$ for $i \not= k+1$.
By \cite[Theorem 6.3]{X3}, we have
\[
\pi^M_{\succ_k}(\psi^{(k+1)}_{(d_1, \dots, d_m)}, \ub\zeta^{(k+1)}, \ub{l}^{(k+1)}, \ub{\eta'})
\cong
\pi^M_{\succ_{k+1}}(\psi^{(k+1)}_{(d_1, \dots, d_m)}, \ub\zeta^{(k+1)}, \ub{l}^{(k+1)}, \ub{\eta}^{(k+1)}).
\]
Now, by replacing $d_1, \dots, d_k$ if necessarily, 
for a positive integer $d_{k+1}$ with 
$B_{k+2}+d_{k+2} \ll \dots \ll B_m+d_m \ll d_{k+1}+\delta \ll \dots \ll d_1+\delta$,
we have
\begin{align*}
\pi(\EE) \cong 
&(U^{(1)} \circ S^{(1)}_{d_1}) \circ \dots \circ (U^{(k)} \circ S^{(k)}_{d_k})
\circ T^{(m)}_{d_m} \circ \dots \circ T^{(k+2)}_{d_{k+2}}
\\& 
\circ(U^{(k+1)} \circ S^{(k+1)}_{d_{k+1}})
\left(\pi^M_{\succ_{k+1}}(\psi^{(k+1)}_{(d_1, \dots, d_m)}, \ub\zeta^{(k+1)}, \ub{l}^{(k+1)}, \ub{\eta}^{(k+1)})\right). 
\end{align*}
Since $U^{(j)} \circ S^{(j)}_{d_{j}}$ commutes with $T_{d_i}^{(i)}$, 
we obtain the claim for $k+1$.
\par

By the claim for $k=m$, we conclude that 
\[
\pi(\EE) \cong 
(U^{(1)} \circ S^{(1)}_{d_1}) \circ \dots \circ (U^{(m)} \circ S^{(m)}_{d_m})
\left(\pi^M_{\hat{\succ}}(\psi^{(m)}_{(d_1, \dots, d_m)}, (-,\dots,-), \ub{\hat{l}}, \ub{\hat\eta})\right)
\]
for positive integers $d_m \ll \dots \ll d_1$.
Taking the Aubert duality, we have
\begin{align*}
\hat\pi(\EE) 
&\cong V_{d_1}^{(1)} \circ \dots \circ V_{d_m}^{(m)}
\left(\pi^M_{\hat{\succ}}(\hat\psi^{(m)}_{(d_1, \dots, d_m)}, (+,\dots,+), \ub{\hat{l}}, \ub{\hat\eta})\right)
\\&\cong V_{d_1}^{(1)} \circ \dots \circ V_{d_m}^{(m)}(\pi(\EE_{(d_1, \dots, d_m)})), 
\end{align*}
where
\[
V_{d_i}^{(i)} = D_{\rho|\cdot|^{-B_i+1}, \dots, \rho|\cdot|^{A_i+1}} 
\circ \dots \circ
D_{\rho|\cdot|^{d_i+\delta}, \dots, \rho|\cdot|^{a_i+d_i+\delta-1}}
\]
and 
\[
\EE_{(d_1, \dots, d_m)} = \{([a_i+d_i+\delta-1,d_i+\delta]_{\rho}, \hat{l}_i, \hat{\eta}_i)\}_{i \in (I_\rho,\hat{\succ})}.
\]
Therefore, by definition, we have $\hat\pi(\EE) \cong \pi(\hat\EE)$. 
\par

Finally, we prove the general case. 
Let $\EE = \{([A_i,B_i]_{\rho}, l_i, \eta_i)\}_{i \in (I_\rho,\succ)}$ with $\pi(\EE) \not= 0$. 
Take a positive integer $d$ such that $B_i+d \geq 0$ for any $i \in I_\rho$, 
and consider $\EE' = \{([A_i+d,B_i+d]_{\rho}, l_i, \eta_i)\}_{i \in (I_\rho,\succ)}$. 
Then by Theorem \ref{nonzero2}, we have 
$\pi(\EE) = D(\pi(\EE'))$ up to a multiplicity, 
where 
\[
D = 
(D_{\rho|\cdot|^{x_t+1}} \circ \dots \circ D_{\rho|\cdot|^{x_1+1}})
\circ \dots \circ 
(D_{\rho|\cdot|^{x_t+d}} \circ \dots \circ D_{\rho|\cdot|^{x_1+d}})
\]
with 
\[
\bigsqcup_{i \in I_\rho} [A_i,B_i]_\rho
= \{\rho|\cdot|^{x_1}, \dots, \rho|\cdot|^{x_t}\}
\]
as multi-sets such that $x_1 \leq \dots \leq x_t$.
In particular, $\pi(\EE') \not= 0$ so that $\hat\pi(\EE') = \pi(\hat\EE')$.
Since $\hat\EE' = \{([A_i+d,-B_i-d]_{\rho}, \hat{l}_i+d, \hat{\eta}_i)\}_{i \in (I_\rho,\hat{\succ})}$, 
by definition (\S \ref{s.segment-rep}), 
we have $\pi(\hat\EE) = D'(\pi(\hat\EE'))$ up to a multiplicity, 
where
\[
D' = 
D_{\rho|\cdot|^{-(x_1+1)}, \dots, \rho|\cdot|^{-(x_t+1)}} 
\circ \dots \circ 
D_{\rho|\cdot|^{-(x_1+d)}, \dots, \rho|\cdot|^{-(x_t+d)}}. 
\]
By the compatibility of the Aubert duality with derivatives, 
we conclude that $\hat\pi(\EE) \cong \pi(\hat\EE)$.
This completes the proof of Theorem \ref{aubert} when only one $\rho$ appears in $\EE$. 
The general case is proven similarly. 
\end{proof}

By the same argument, we can compare 
our construction of $A$-packets (Theorem \ref{reform}) with M{\oe}glin's original version.
Let $\psi = \oplus_{\rho}(\oplus_{i \in I_\rho} \rho \boxtimes S_{a_i} \boxtimes S_{b_i})$ 
be an $A$-parameter for $G_n$ of good parity. 
Set $A_i = (a_i+b_i)/2-1$ and $B_i = (a_i-b_i)/2$. 
Take $\zeta_i \in \{\pm1\}$ so that $\zeta_iB_i \geq 0$, 
and put $I_\rho^\pm = \{i \in I_\rho \;|\; \zeta_i = \pm1\}$. 
We fix a total order $\succ$ on $I_\rho = I_\rho^- \sqcup I_\rho^+$ satisfying that
\begin{itemize}
\item
if $\zeta_i = -1$ and $\zeta_j = +1$, then $i \prec j$; 
\item
if $\zeta_i = \zeta_j$ and $|B_i| < |B_j|$, then $i \prec j$.
\end{itemize}
From $\cup_\rho\{([A_i,-B_i]_\rho, l_i, \eta_i)\}_{i \in (I_\rho^-,\succ)}$, 
we obtain $\cup_\rho\{([A_i,B_i]_\rho, \hat{l}_i, \hat{\eta}_i)\}_{i \in (I_\rho^-,\hat{\succ})}$ 
by Definition \ref{hatE}. 

\begin{thm}\label{compare}
Notation is as above. 
Then $\pi^M_{\succ}(\psi, \ub{\zeta}, \ub{l}, \ub{\eta}) \cong \pi(\EE)$ with
\[
\EE = \bigcup_\rho
\{([A_i,B_i]_\rho, \hat{l}_i, \hat{\eta}_i)\}_{i \in (I_\rho^-,\hat{\succ})} 
\cup 
\{([A_j,B_j]_\rho, l_j, \eta_j)\}_{j \in (I_\rho^+,\succ)}, 
\]
where if $i \in I_\rho^-$ and $j \in I_\rho^+$, then we regard that $i$ is less than $j$.
\end{thm}

\begin{ex}
Let us consider 
\[
\psi = \1 \boxtimes S_2 + S_4 \boxtimes \1 + \1 \boxtimes S_6 \in \Psi_\gp(\SO_{13}(F)). 
\]
We write $\psi = \oplus_{i \in \{1,2,3\}} S_{a_i} \boxtimes S_{b_i}$
such that $\max\{a_i, b_i\} = 2i$. 
Take three orders $\succ, \succ', \succ''$ on $\{1,2,3\}$ so that 
\[
1 \prec 2 \prec 3, \quad
1 \prec' 3 \prec' 2, \quad
3 \prec'' 1 \prec'' 2.
\]
Note that 
$\succ$ (\resp $\succ'$, $\succ''$)
is used in \cite[Section 6]{X2} (\resp Theorem \ref{compare}, Example \ref{ex-super}). 
Define $\ub\zeta = (\zeta_1, \zeta_2, \zeta_3)$ and $\ub{l} = (l_1, l_2, l_3)$ 
by $\zeta_1 = \zeta_3 = -1$, $\zeta_2 = +1$ and $l_1 = l_2 = l_3 = 0$. 
M{\oe}glin's construction says that 
\[
\Pi_\psi = \left\{\pi^M_\succ(\psi, \ub\zeta, \ub{l}, \ub{\eta}) \;\middle|\; 
\ub\eta = (\eta_1,\eta_2,\eta_3) \in \{\pm1\}^3,\; \eta_1\eta_2\eta_3 = 1 \right\}.
\]
Moreover, by \cite[Theorem 6.3]{X3}, we have 
\[
\pi^M_\succ(\psi, \ub\zeta, \ub{l}, (\eta_1,\eta_2,\eta_3))
\cong 
\pi^M_{\succ'}(\psi, \ub\zeta, \ub{l}, (\eta_1,-\eta_2,-\eta_3)). 
\]
On the other hand, as we have seen in Example \ref{ex-super}, 
we have $\Pi_\psi = \{\pi(\EE_i) \;|\; i = 1,2,3,4\}$ with 
\begin{align*}
\EE_1 = 
\bordermatrix{
& -\half{5} & -\half{3} & -\half{1} & \half{1} & \half{3} & \half{5}\cr
& \lhd & \lhd & \lhd & \rhd & \rhd & \rhd \cr
&  &  & \lhd & \rhd &  &  \cr
&  &  &  &  & \oplus &  \cr
}, 
&\quad
\EE_2 = 
\bordermatrix{
& -\half{5} & -\half{3} & -\half{1} & \half{1} & \half{3} & \half{5}\cr
& \lhd & \lhd & \lhd & \rhd & \rhd & \rhd \cr
&  &  & \ominus & \oplus &  &  \cr
&  &  &  &  & \ominus &  \cr
}, 
\\
\EE_3 = 
\bordermatrix{
& -\half{5} & -\half{3} & -\half{1} & \half{1} & \half{3} & \half{5}\cr
& \lhd & \lhd & \oplus & \ominus & \rhd & \rhd \cr
&  &  & \lhd & \rhd &  &  \cr
&  &  &  &  & \ominus &  \cr
}, 
&\quad
\EE_4 = 
\bordermatrix{
& -\half{5} & -\half{3} & -\half{1} & \half{1} & \half{3} & \half{5}\cr
& \lhd & \lhd & \oplus & \ominus & \rhd & \rhd \cr
&  &  & \ominus & \oplus &  &  \cr
&  &  &  &  & \oplus &  \cr
}.
\end{align*}
By Theorem 6.6, we have 
\begin{align*}
\pi^M_{\succ'}(\psi, \ub\zeta, \ub{l}, (+1,+1,+1)) &\cong \pi(\EE_1), \\
\pi^M_{\succ'}(\psi, \ub\zeta, \ub{l}, (-1,-1,+1)) &\cong \pi(\EE_2), \\
\pi^M_{\succ'}(\psi, \ub\zeta, \ub{l}, (+1,-1,-1)) &\cong \pi(\EE_3), \\
\pi^M_{\succ'}(\psi, \ub\zeta, \ub{l}, (-1,+1,-1)) &\cong \pi(\EE_4).
\end{align*}
In conclusion, we have
\begin{align*}
\pi^M_{\succ}(\psi, \ub\zeta, \ub{l}, (+1,+1,+1)) 
&\cong 
L(|\cdot|^{-\half{5}}, \Delta[-1/2, -3/2]; \pi((1/2)^-, (3/2)^-)), \\
\pi^M_{\succ}(\psi, \ub\zeta, \ub{l}, (+1,-1,-1)) 
&\cong L(|\cdot|^{-\half{5}}, |\cdot|^{-\half{3}}, |\cdot|^{-\half{1}}, |\cdot|^{-\half{1}}; \pi((3/2)^+)), \\
\pi^M_{\succ}(\psi, \ub\zeta, \ub{l}, (-1,+1,-1)) 
&\cong \pi((1/2)^-, (3/2)^+, (5/2)^-), \\
\pi^M_{\succ}(\psi, \ub\zeta, \ub{l}, (-1,-1,+1)) 
&\cong L(|\cdot|^{-\half{5}}, |\cdot|^{-\half{3}}, |\cdot|^{-\half{1}}; \pi((1/2)^+, (3/2)^+)).
\end{align*}
It would be a hard exercise 
to check these isomorphisms from M{\oe}glin's original construction (\cite[Definition 6.3]{X2}).
\end{ex}


\end{document}